\documentclass[11pt,letterpaper]{amsart}
\usepackage{egen}
\usepackage{csquotes}
\usepackage[style=alphabetic, giveninits=true, url=false, backend=biber] {biblatex}
\allowdisplaybreaks
\usepackage[english]{babel}
\usepackage{amsfonts}
\usepackage{amsthm}
\usepackage{amssymb}
\usepackage{bbm}
\usepackage{cancel}
\usepackage{color}
\usepackage[curve]{xypic}
\usepackage{graphicx}
\usepackage{mathabx}
\usepackage{tikz-cd}
\usepackage{epigraph} %For a quote on the first page
\usepackage{hyperref}
\usepackage{appendix}
\usepackage{enumitem}
\DeclareMathOperator{\Exp}{Exp}

\DeclareMathOperator{\tr}{tr}
\newcommand{\loc}{\text{loc}}

\newtheorem{theorem}{Theorem}[section]
\newtheorem{corollary}[theorem]{Corollary}
\newtheorem{proposition}[theorem]{Proposition}

\theoremstyle{definition}
\newtheorem{definition}[theorem]{Definition}
\newtheorem{example}[theorem]{Example}
\newtheorem{remark}[theorem]{Remark}
\newtheorem{lemma}[theorem]{Lemma}

\theoremstyle{property}

\newtheorem{assumption}[theorem]{Assumption}

\usepackage{graphicx}
\newcommand{\intprod}{\mathbin{\raisebox{\depth}{\scalebox{1}[-1]{$\lnot$}}}}
\DeclareFontFamily{OT1}{rsfs}{}
\DeclareFontShape{OT1}{rsfs}{n}{it}{<-> rsfs10}{}
\DeclareMathAlphabet{\curly}{OT1}{rsfs}{n}{it}

\renewcommand\O{\mathcal O}
\newcommand\F{\mathcal F}
\renewcommand{\cH}{\mathcal{H}}
\newcommand\cN{\mathcal N}
\newcommand\T{\mathbb T}

\newcommand\C{\mathbb C}

\newcommand\G{\mathcal G}

\newcommand\sfZ{\mathsf Z}

\newcommand\Q{\mathbb Q}

\newcommand\Z{\mathbb Z}

\newcommand\NP{\mathrm{NP}}

\newcommand\ncQuot{\mathrm{ncQuot}}

\newcommand\Quot{\mathrm{Quot}}
\newcommand\ncQ{\mathrm{ncQuot}}

\newcommand\vd{\mathrm{vd}}
\newcommand\vir{\mathrm{vir}}

\newcommand\coker{\operatorname{coker}}
\renewcommand\hom{\mathcal{H}{\it{om}}}

\newcommand\mdot{{\scriptscriptstyle\bullet}}

\newcommand\INTO{\ar@{^{(}->}[r]}

\newcommand\bull{{\scriptscriptstyle\bullet}}
\renewcommand\udot{^\bull}

\newcommand{\glob}{\mathrm{glob}}
\newcommand{\sheaf}{\mathrm{sheaf}}
\DeclareRobustCommand{\SkipTocEntry}[4]{}

\DeclareMathOperator{\rev}{rev}
\newcommand{\op}{\mathrm{op}}

\begin{document}
\title[Proof of a magnificent conjecture]{Proof of a magnificent conjecture}
\author[M.~Kool and J.~V.~Rennemo]{M.~Kool and J.~V.~Rennemo}
\maketitle
\vspace{-1cm}
\begin{abstract}
Motivated by super-Yang--Mills theory on a Calabi--Yau 4-fold, Nekrasov and Piazzalunga have assigned weights to $r$-tuples of solid partitions and conjectured a formula for their weighted generating function.

We define $K$-theoretic virtual invariants of Quot schemes of 0-dimensional quotients of $\O_{\C^4}^{\oplus r}$ by realizing them as zero loci of isotropic sections of orthogonal bundles on non-commutative Quot schemes.
Via the Oh--Thomas localization formula, we recover Nekrasov--Piazzalunga's weights and derive their sign rule.

Our proof passes through refining the $K$-theoretic invariants to sheaves and describing them via Clifford modules, which lets us show that they arise from a factorizable sequence of sheaves in the sense of Okounkov.
Taking limits of the equivariant parameters, we then deduce the Nekrasov--Piazzalunga conjecture from its 3-dimensional analog.
\end{abstract}
\bigskip
\setlength{\epigraphwidth}{3in}
\epigraph{I saw the sign and it opened up my eyes \\
	I saw the sign \\
	Life is demanding without understanding \\
	I saw the sign and it opened up my eyes}{\textit{The Sign}, Ace of Base}

\section{Introduction}

\subsection{Enumerating partitions} 

The generating series of partitions, or Young diagrams, admits the following infinite product formula attributed to Euler
$$
\sum_{\pi} q^{|\pi|} = \prod_{n=1}^{\infty} (1-q^n)^{-1},
$$
where $|\pi|$ denotes the size of $\pi$. MacMahon found the following expression for the generating series of plane partitions, i.e.~3-dimensional piles of boxes
$$
\sum_{\pi} q^{|\pi|} = \prod_{n=1}^{\infty} (1-q^n)^{-n}.
$$
In this paper we consider \emph{solid partitions}, where a solid partition $\pi$ is a finite subset of $\mathbb{Z}_{\ge 0}^4$ such that if $(a_1,a_2,a_3,a_4) \in \pi$ and $0 \le b_i \le a_i$ for $i = 1, 2, 3, 4$, then also $(b_1,b_2,b_3,b_4) \in \pi$.

MacMahon initially expected the generating series of solid partitions to satisfy (see \cite{ABMM})
$$
\sum_{\pi} q^{|\pi|} \overset{?}{=} \prod_{n=1}^{\infty} (1-q^n)^{-\frac{1}{2}n(n+1)}.
$$
This formula predicts 141 solid partitions of size 6 -- one too many. As mentioned in \cite[Sect.~11]{Sta}, the study of solid partitions has seen little progress.\footnote{Stanley \cite{Sta} remarks: 
\noindent \emph{``The case $r = 2$ has a well-developed theory --- here 2-dimensional partitions are known as plane
partitions (...) For $r \geq 3$ almost nothing is known, and (...) casts only a faint glimmer of light
on a vast darkness.''}}

\subsection{Magnificent Four}

In his ``Magnificent Four'' paper, Nekrasov \cite{Nek} assigned a weight 
$$
(-1)^{\mu_\pi}[-\mathsf{v}_\pi]  \in \Q(t_1,t_2,t_3,t_4,y^{\frac{1}{2}}) \ \text{with }t_4 = t_1^{-1}t_2^{-1}t_3^{-1}
$$
to any solid partition $\pi$.
Here
$$
\mu_\pi = |\{(i,i,i,j) \in \pi \, : \, j > i\}|
$$
and the term $[-\mathsf{v}_\pi]$ is an explicit rational function determined by the character
\[
Z_\pi = \sum_{(a_1,a_2,a_3,a_4) \in \pi} t_1^{a_1}t_2^{a_2}t_3^{a_3}t_4^{a_4}.
\]
More generally, to an $r$-tuple $\vec{\pi} = (\pi_1, \ldots, \pi_r)$ of solid partitions, Nekrasov--Piazzalunga \cite{NP} assigned a weight
$$
(-1)^{\mu_{\vec{\pi}}} [-\mathsf{v}_{\vec{\pi}}] \in \Q(t_1,t_2,t_3,t_4,w_1, \ldots, w_r,y_1^{\frac{1}{2}}, \ldots, y_r^{\frac{1}{2}})\ \text{with }t_4 = t_1^{-1}t_2^{-1}t_3^{-1},
$$
where $\mu_{\vec{\pi}} = \sum_{\alpha} \mu_{\pi_\alpha}$. We recall the expressions for $[-\mathsf{v}_\pi]$ and $[-\mathsf{v}_{\vec \pi}]$ in Definition \ref{def:NPweight}. Although the sign $\mu_{\vec{\pi}}$ and weight $[-\mathsf{v}_{\vec{\pi}}]$ are not invariant with respect to permutations of $\{1,2,3,4\}$, the product $(-1)^{\mu_{\vec{\pi}}} [-\mathsf{v}_{\vec{\pi}}]$ is invariant under such permutations \cite[Thm.~2.9]{Mon}.

Having defined these weights, one may consider the formal power series
\begin{equation} \label{NPZ}
\mathsf{Z}_{r}^{\NP} = \sum_{\vec{\pi} = (\pi_1, \ldots, \pi_r)}  (-1)^{\mu_{\vec{\pi}}} [ -\mathsf{v}_{\vec{\pi}} ]  \, ((-1)^{r} q)^{|\vec{\pi}|},
\end{equation}
where $|\vec{\pi}| = \sum_{\alpha} |\pi_\alpha|$.
Nekrasov and Piazzalunga give a conjectural formula for $\mathsf{Z}_r^{\NP}$, and the main goal of this paper is to give a proof of this formula.

To describe the formula, recall that the \emph{plethystic exponential} of a Laurent series $F(z_1, \ldots, z_n)$ is defined by
\begin{equation}
\mathrm{Exp}(F(z_1, \ldots, z_n)) = \exp\Big( \sum_{m=1}^{\infty} \frac{1}{m} F(z_1^m, \ldots, z_n^m) \Big), \label{eqn:DefinitionOfPlethystic}
\end{equation}
where we assume that the coefficients of $F$ are such that this is well defined algebraically (i.e.~such that no infinite sums occur in the computation of any one coefficient of $\mathrm{Exp}(F)$).
For any monomial $z_1^{i_1} \cdots z_n^{i_n}$, we define 
$$
[z_1^{i_1} \cdots z_n^{i_n}] = z_1^{\frac{i_1}{2}} \cdots z_n^{\frac{i_n}{2}} - z_1^{-\frac{i_1}{2}} \cdots z_n^{-\frac{i_n}{2}}.
$$
\begin{theorem}[Nekrasov--Piazzalunga's conjecture] 
	\label{mainthm}
	We have
$$
\mathsf{Z}_{r}^{\NP} = \mathrm{Exp}\Bigg(\frac{[t_1t_2][t_1t_3][t_2t_3][y]}{[t_1][t_2][t_3][t_4][y^{\frac{1}{2}} q ]   [y^{\frac{1}{2}} q^{-1} ]}\Bigg),
$$
where $y=y_1 \cdots y_r$ and $t_1t_2t_3t_4 = 1$. In particular, $\mathsf{Z}^{\NP}_{r}$ is independent of $w_1, \ldots, w_r$.
\end{theorem}

\begin{remark}
After eliminating $t_4 = (t_1t_2t_3)^{-1}$, the argument of the plethystic exponential in the conjecture is a rational function in $t_1,t_2,t_3,y^{\frac{1}{2}},q$. 
This rational function can be expanded as a formal power series in positive powers of $q$, whose coefficients are rational functions in $t_i, y^{\frac{1}{2}}$.
One can then apply the plethystic exponential, because only positive powers of $q$ appear in the plethystic exponent.
\end{remark}

\begin{remark}
The $r=1$ case of Theorem \ref{mainthm} was conjectured by Nekrasov in \cite{Nek}.
A special limit of the $r = 1$ case, namely taking $t_i = \exp(b \lambda_i)$ and $y = \exp(- b (d_1\lambda_1 + d_2 \lambda_2 + d_3 \lambda_3 + d_4 \lambda_4))$ with $b \to 0$, was conjectured in \cite{CK}.\footnote{In \cite[Conj.~1.10(c)]{CK}, the authors also give a conjectural formula for the further specialization $d_1=d_2=d_3=0$ and $\lambda_4 = 0$. Though true for solid partitions of size $\leq 6$, it fails for solid partitions of size 8.} This specialization was recently proved using novel degeneration methods by Cao--Zhao--Zhou \cite{CZZ} (relying on this paper's proof of Nekrasov's sign formula  for the localization, see \cite[Sect.~6.4]{CZZ}).

The motivation in \cite{CK} comes from a parallel conjecture on tautological integrals on Hilbert schemes of points on \emph{compact} Calabi--Yau 4-folds. 
This conjecture was proved in the very ample case by H.~Park \cite{Par} using virtual pullback.
The general case was shown by Bojko \cite{Boj} to follow from a virtual 4-fold wall-crossing formula of the kind conjectured in \cite{GJT}.
This wall-crossing formula has since been established by Bojko et al.~in \cite{BKLT}.
\end{remark}

\subsection{Three dimensions} 

The Hilbert scheme parametrizing 0-dimensional subschemes of length $n$ on $\C^d$, denoted by $\Hilb^n(\C^d)$, is only smooth if $d\le 2$ or $n\le 3$. In general, even its dimension is unknown. 

For $d=3$, $\Hilb^n(\C^3)$ is the critical locus of a superpotential on a smooth ambient space. This structure was exploited by Behrend--Bryan--Szendr\H{o}i \cite{BBS} in their study of \emph{motivic} Donaldson--Thomas invariants and by Okounkov \cite{Oko} in his proof of the Nekrasov conjecture for \emph{$K$-theoretic} Donaldson--Thomas invariants. Denoting the twisted virtual structure sheaf $\O^{\vir} \otimes (\det T^{\vir})^{-1/2}$ of $\Hilb^n(\C^3)$ (and later the Quot scheme) by $\widehat{\O}^{\vir}$ we have:
\begin{theorem}[{\cite[Thm.~3.3.6]{Oko}}] \label{Okounkovthm}
$$
\sum_{n=0}^{\infty} \chi(\Hilb^n(\C^3), \widehat{\O}^{\vir}) \, (-q)^n = \mathrm{Exp}\Bigg(\frac{[t_1t_2][t_1t_3][t_2t_3]}{[t_1][t_2][t_3][\kappa^{\frac{1}{2}} q ]   [\kappa^{\frac{1}{2}} q^{-1} ]}\Bigg), \quad \kappa=t_1t_2t_3.
$$
\end{theorem}

As noted by Beentjes--Ricolfi \cite{BR}, the Quot scheme $\Quot^{n}_{r}(\C^3)$ parametrizing 0-dimensional quotients $\O_{\C^3}^{\oplus r} \twoheadrightarrow Q$ of length $n$ is also the critical locus of a superpotential. This has been used by Fasola--Monavari--Ricolfi \cite{FMR} and Arbesfeld--Kononov \cite{AK1} to establish the Awata--Kanno conjecture \cite{AK2}:

\begin{theorem}[Fasola--Monavari--Ricolfi, Arbesfeld--Kononov] \label{FMRthm}
$$
\sum_{n=0}^{\infty} \chi(\Quot^{n}_{r}(\C^3), \widehat{\O}^{\vir}) \, ((-1)^r q)^n = \mathrm{Exp}\Bigg(\frac{[t_1t_2][t_1t_3][t_2t_3][\kappa^r]}{[t_1][t_2][t_3][\kappa][\kappa^{\frac{r}{2}} q ]   [\kappa^{\frac{r}{2}} q^{-1} ]}\Bigg).
$$
\end{theorem}

In four dimensions, $\Hilb^n(\C^4)$ and $\Quot^{n}_{r}(\C^4)$ are not critical loci, but instead they can be realized as the zero locus of an isotropic section of an orthogonal bundle on a smooth ambient space. 
Using the $K$-theoretic sheaf counting theory for Calabi--Yau 4-folds defined by Oh--Thomas \cite{OT}, this allows us to give a global definition of Nekrasov--Piazzalunga's invariants and prove their conjecture.
The invariants defined by Oh--Thomas are generalizations and refinements of the sheaf counting invariants on Calabi--Yau 4-folds previously defined by Borisov--Joyce \cite{BJ}, and in special cases by Cao--Leung \cite{CL}.

\subsection{Oh--Thomas theory} \label{sec:OTtheory}

In this paper, we only need what we will call the \textit{standard model} version of Oh--Thomas's invariants \cite{OT}, which we now discuss. 
Let $A$ be a connected smooth quasi-projective variety and $(E,q)$ an orthogonal bundle of even rank $m$ on $A$, i.e.~$E$ is a rank $m$ vector bundle and $q \colon E \otimes E \rightarrow \O_A$ is a non-degenerate symmetric bilinear form. 
Suppose $\Lambda \subset E$ is a maximal isotropic subbundle of $E$.
This gives a short exact sequence of vector bundles
\begin{equation} \label{splitting}
0 \to  \Lambda \to E \to \Lambda^\vee \to 0.
\end{equation}
Suppose we have an isotropic section $s \in \Gamma(A,E)$, i.e.~$q(s,s)=0$, and let $M = Z(s) \subset A$. This gives rise to the following commutative diagram of sheaves on $M$
\begin{displaymath}
\xymatrix
{
T_A|_M \ar^<<<<{ds}[r] &E|_M \stackrel{q}{\cong} E^\vee|_M  \ar^<<<<<{(ds)^\vee}[r] \ar@{->>}^{s^\vee}[d] & \Omega_A|_M \ar@{=}[d] \\
& I/I^2|_M \ar^{d}[r] & \Omega_A|_M
}
\end{displaymath}
where $I = s^\vee(E) \subset \O_A$ is the ideal cutting out $M$.
We view the top row as a complex $E^\mdot$ of locally free sheaves concentrated in degrees $-2,-1,0$ satisfying $E^{\mdot} \cong (E^{\mdot})^{\vee}[2]$. The bottom row is the truncated cotangent complex $\tau^{\geq -1} \mathbb{L}_M$, and this diagram provides an obstruction theory $E^\mdot \to \tau^{\geq -1} \mathbb{L}_M$. Since $E^\mdot$ is \emph{not} 2-term, we cannot use the classical theory of virtual classes and virtual structure sheaves \cite{BF, LT}.

To get a well-defined virtual class, we must first choose an \textit{orientation} of $E$.
We refer to \cite[Sect.~2]{OT} for a more precise discussion of orientations and the subtle sign choices involved, and we only recall a few basic facts here.
Note that $\Lambda^{m} q$ induces an isomorphism $\det(E) \cong \det(E^\vee)$. An \emph{orientation} of $(E,q)$ is a map $o \colon \O_A \to \det(E)$ such that
the composed map
\[
\O_A \overset{o \otimes o}{\to} \det(E) \otimes \det(E) \overset{\id \otimes \Lambda^{m}q}\to \det(E) \otimes \det(E^\vee) \to \O_A
\]
equals $(-1)^{m/2}\id_{\O_A}$.\footnote{The homomorphism $\det(E) \otimes \det(E^\vee) \to \O_A$ is defined by $a_1\wedge \dots \wedge a_m \otimes a_m^\vee \wedge \dots \wedge a_1^\vee \mapsto 1$.} Under our assumption \eqref{splitting}, we get a canonical isomorphism 
\[
\pi_\Lambda \colon \det(E) \cong \det(\Lambda) \otimes \det(\Lambda^\vee) \cong \O_A,
\]
and, by \cite[Def. 2.2]{OT}, the orientation \textit{induced by} $\Lambda$ is
\[
o_\Lambda = (-i)^{m/2} \pi_{\Lambda}^{-1}.
\]
For a given orientation $o$ of $(E,q)$, we say that $\Lambda$ is \emph{positive} with respect to $o$, if $o_\Lambda = o$.
For our purposes, it is usually more convenient to specify an orientation through a choice of positive maximal isotropic subbundle $\Lambda$ rather than through a trivialization of $\det(E)$. 

Oh and Thomas construct a \emph{virtual class} $[M]^{\mathrm{vir}} \in A_n(M,\Z[\tfrac{1}{2}])$ and a \emph{twisted virtual structure sheaf} $\widehat{\mathcal{O}}_M^{\mathrm{vir}} \in K_0(M,\Z[\tfrac{1}{2}])$, which depend on the choice of orientation of $(E,q)$.
Here $A_n(M,\Z[\tfrac{1}{2}])$ is the Chow group of degree $n=\tfrac{1}{2}\rk(E^\mdot)$ with coefficients in  $\Z[\tfrac{1}{2}] = \{\frac{a}{2^b} \, : \, a, b \in \Z\}$, and $K_0(M,\Z[\tfrac{1}{2}]) = K_0(M) \otimes \mathbb Z[\tfrac12]$ denotes the Grothendieck group of coherent sheaves on $M$ ``with coefficients in $\Z[\tfrac{1}{2}]$''.

In the case of interest to us, $M$ is non-compact, but there is an action of an algebraic torus $T$ on $M$, induced by a $T$-action on $A$, with compact fixed locus $M^T$. 
If $E$ and $s$ are $T$-equivariant and $q$ is $T$-invariant, then Oh--Thomas also define the $T$-equivariant versions of these classes $[M]^{\vir} \in A_n^T(M,\Z[\frac12])$ and $\widehat{\O}^{\vir}_M \in K_0^T(M)_{\mathrm{loc}}$. Here $K_0^T(M)_{\mathrm{loc}}$ is defined by \cite[Sect.~7]{OT}
\[
K_0^T(M)_{\mathrm{loc}} := K^T_0(M) \otimes_{\Z[t,t^{-1}]} \Q(t^{\frac{1}{2}}),
\]
where $t = (t_1, \dots, t_{\dim T})$ is the tuple of $T$-equivariant parameters and furthermore $t^a = (t_1^a,\dots,t_{\dim T}^a)$. 

\begin{remark}
Note that the geometric setup described above is adapted to our concrete problem and is more restrictive than what \cite{OT} require, even locally.

Oh--Thomas also work in a global setting, where $M$ is a quasi-projective scheme with a $(-2)$-shifted symplectic structure (in the sense of \cite{PTVV}) and an orientation. From this data, one obtains an obstruction theory $E^\mdot \to \tau^{\geq -1} \mathbb{L}_M$, where $E^\mdot$ is a self-dual
3-term complex of locally free sheaves, and a choice of orientation for $E^\mdot$, i.e.~an appropriate trivialization of $\det(E^\mdot)$.
Oh--Thomas define a virtual class and virtual structure sheaf in this general setting as well.
\end{remark}

\begin{remark} At various points in this text we will be taking square roots of an equivariant line bundle $L$.
In the ring $K^0_T(M)_{\mathrm{loc}}$, a square root of the class of $L$ exists and is unique (see \cite[Sect.~5.1]{OT}), and so we may write $L^{\frac12}$ without further comment.

In Section \ref{sec:global}, where we prove that the generating function $\sfZ^{\NP}_1$ is factorizable, the problem of taking square roots becomes more subtle.
This is because the definition of factorizable sequence requires that certain sheaves are isomorphic, and so we are forced to specify square roots as line bundles rather than just $K$-theory classes, and in particular neither existence nor uniqueness of such square roots is automatic.
\end{remark}

In the setting of the standard model, and using the orientation induced by $\Lambda$, we use two facts from \cite{OT}:\ \\

\noindent \textbf{Virtual structure sheaf.}
 Let $\iota \colon M \hookrightarrow A$ denote the inclusion. 
 Then 
$$
\iota_* \widehat{\mathcal{O}}^{\vir}_M = \widehat{\Lambda}^\mdot \Lambda^\vee \otimes \det(T_A)^{-\frac{1}{2}} \in K_0^T(A)_{\loc},
$$
where, for any vector bundle $V$, we define
\begin{equation} \label{eqn:lambdahat}
\Lambda^{\mdot} V = \sum_{i = 0}^{\rk V} (-1)^i \Lambda^i V, \quad \widehat{\Lambda}^{\mdot} V = \frac{\Lambda^\mdot V}{\det(V)^{\frac{1}{2}}}.
\end{equation}

\noindent \textbf{Virtual localization.} Let $T$ be an algebraic torus acting on $A$. Suppose $E,\Lambda, s$ are $T$-equivariant, $q$ is $T$-invariant, and $M^T$ consists of isolated reduced points. Denote the fixed and moving part of a $T$-equivariant complex $V^\mdot$ at a fixed point $P$ by $(V^\mdot|_P)^f$ and $(V^\mdot|_P)^m$. Suppose $\dim( (T_{A} - \Lambda)|_P^f) = 0$ for all $\iota_P \colon \{P\} \subset M^T$, i.e.~all fixed points have virtual dimension 0. 
Then tracking the signs in the localization formula of \cite{OT}, we prove (Section \ref{sec:localfixisolated})
\begin{equation} \label{virloc1}
\widehat{\O}_M^{\vir} = \sum_{P \in M^T} (-1)^{n_P^\Lambda} \, \iota_{P*} \Bigg( \frac{1}{\widehat{\Lambda}^{\mdot} (\Omega_{A} - \Lambda^\vee)|_{P}} \Bigg)
\end{equation}
with
\begin{equation}
\label{eqn:CohomologicalFormulaForLocalisationSign}
n_P^\Lambda \equiv \dim  \mathrm{coker}(p_{\Lambda} \circ ds|_P)^f  \pmod 2,
\end{equation}
where $p_{\Lambda} \circ ds|_P$ is the composition
$$
T_{A}|_P \stackrel{ds|_P}{\to} E|_P  \stackrel{p_{\Lambda}}{\to} \Lambda|_P
$$
and $p_{\Lambda}$ is the projection onto $\Lambda|_P$ for any splitting $E|_P = \Lambda|_P \oplus \Lambda^\vee|_P$.

\subsection{Outline}

In Section \ref{sec:tools} we review and develop some of the general tools used in our proof, i.e.~factorizability, Clifford algebras, spin modules, and torus localization. In Section \ref{sec:global}, we show that $M=\Quot^n_r(\C^4)$ satisfies the ``standard model'' setup by embedding it into a non-commutative Quot scheme $A$. We define $K$-theoretic virtual invariants
$$
N_{r,n}^{\mathrm{glob}} = \chi(M, \widehat{\O}^{\vir}_M \otimes \widehat{\Lambda}^\mdot \hom(\cV,\cW)), 
$$
for a certain bundle $\hom(\cV,\cW)$ on $M$, and consider the generating function
\[
\mathsf{G}_{r} = \sum_{n=0}^{\infty} N_{r,n}^{\mathrm{glob}} \, q^n.
\]

For $r=1$, we use the Hilbert--Chow morphism 
$$
\nu_n \colon \Hilb^n(\C^4) \rightarrow \Sym^n(\C^4)
$$
and consider the $K$-theory classes 
\begin{equation} \label{KclassCh}
R\nu_{n*}( \widehat{\O}^{\vir} \otimes \widehat{\Lambda}^\mdot \hom(\cV,\cW) ).
\end{equation}
We refine these classes to actual equivariant $\Z/2$-graded coherent sheaves, and we show that these sheaves form a factorizable sequence in the sense of Okounkov (\cite[Sect.~5.3]{Oko}, \cite{Ren}). 
The refinement of \eqref{KclassCh} to a sheaf involves describing $\widehat{\O}^{\vir}_M$ in terms of spin modules. 
The fact that the sheaves refining \eqref{KclassCh} form a factorizable sequence implies, as in \cite{Oko}, that the generating function $\mathsf{G}_1$ has the form
\[
\mathrm{Exp}\Bigg(\frac{G_1}{[t_1][t_2][t_3][t_4]}\Bigg),
\]
where $G_1$ is a power series in $q$ whose coefficients are \emph{Laurent polynomials} in the variables $t_i, y^{\frac{1}{2}}$.

In Section \ref{sec:local}, we apply the Oh--Thomas localization formula \eqref{virloc1} to show that our global invariants recover Nekrasov--Piazzalunga's invariants i.e.~$\mathsf{G}_r = \sfZ_r^{\NP}$. The main difficulty is the computation of the parity of $n_P^\Lambda$ in \eqref{virloc1}, which we establish by determining the dimensions of certain Ext groups.

In Section \ref{sec:limits}, we apply Okounkov's rigidity principle to further examine the power series $G_1$ appearing above.
We first show that it is divisible by $[t_1t_2][t_2t_3][t_1t_3]$.
A simple formal argument allows us to then determine $\sfZ^{\NP}_1$ by using the fact that for $y=t_4$, $\sfZ_1^{\NP}$ reduces to the known generating series of $K$-theoretic Donaldson--Thomas invariants of $\C^3$ (Theorem \ref{Okounkovthm}). 
For general $r>1$, we take suitable limits of the framing parameters, which allows us to reduce to the $r=1$ case.

\subsection{Conventions}

The torus $(\C^*)^4 = \Spec \C[t_1^{\pm1}, t_2^{\pm1}, t_3^{\pm1}, t_4^{\pm1}]$ acts on closed points $(p_1,p_2,p_3,p_4) \in \C^4$ via
\[
(t_1, t_2, t_3, t_4) \cdot (p_1,p_2,p_3,p_4) = (t_1^{-1}p_1,t_2^{-1}p_2,t_3^{-1}p_3,t_4^{-1}p_4),
\]
so that the character of $H^0(\C^4, \O_{\C^4})$ is $\prod_{i=1}^4 (1-t_i)^{-1}$, and not $\prod_{i=1}^4 (1-t_i^{-1})^{-1}$.
This convention agrees with roughly half of the relevant literature.

For a quasi-projective scheme $Y$ with action by an algebraic torus $T$, we define $K_0^T(Y)_{\loc} := K_0^T(Y) \otimes_{K^T_0(\pt)} \Q(t^{\frac{1}{2}})$, where $K^T_0(\pt) = \Z[t^{\pm 1}]$ and  $t = (t_1, \ldots, t_{\dim T})$ is the tuple of $T$-equivariant parameters. By the $K$-theoretic localization formula \cite{Tho}, pushforward along the inclusion $Y^T \hookrightarrow Y$ gives $K^T_0(Y^T)_{\loc} \cong K_0^T(Y)_{\loc}$.

We use the phrase ``$\Z/2$-graded coherent sheaf'' to mean a sheaf $\F = \F_0 \oplus \F_1$ where the grading is to be interpreted as a cohomological grading.
In $K$-theory we have $[\F] = [\F_0] - [\F_1]$.
If $\F$ and $\G$ are $\Z/2$-graded sheaves, then in the isomorphism $\F \otimes \G \cong \G \otimes \F$ we apply the Koszul sign rule.

Given a 2-periodic complex $\F$, we write $\cH(\F) := H^{-1}(\F)[1] \oplus H^0(\F)$ for its $\Z/2$-graded cohomology sheaf.

\subsection{Acknowledgements} 

This work draws on essential conversations with many colleagues. We warmly thank D.~Maulik for outlining the importance of following the approach in \cite{Oko} early on, and J. Oh and R.P.~Thomas for crucial explanations of \cite{OT}. 
M.K. thanks Y.~Cao for many discussions on Donaldson--Thomas theory of Calabi--Yau 4-folds. 
We thank N.~Arbesfeld for explanations on $K$-theory, A.~Bojko for discussions on orientations, N.~Kuhn for helping us think about spin structures, S.~Monavari for discussions on limits of framing parameters, and F.~Thimm for important comments on an early version of this manuscript. 
We thank an anonymous referee for thorough and helpful comments on an earlier version of the paper.
Finally, we are indebted to A.~Okounkov for writing his lecture notes \cite{Oko} and, of course, to N.~Nekrasov and N.~Piazzalunga for discovering their beautiful formula.
M.K.~is supported by NWO grant VI.Vidi.192.012 and ERC Consolidator Grant
FourSurf 10108736. J.V.R.~is supported by Research Council of Norway grant no.~302277.

\subsection{AI declaration} 

ChatGPT (accessed June 2025) was used to help with details of the proof of Proposition \ref{thm:torsionInPicardGroup} and the related Lemmas \ref{thm:HilbIsSimplyConnected}, \ref{thm:YIsSimplyConnected},  \ref{thm:simplyConnectedMeansTorsionFree}, including locating the key references used in these proofs.
The authors take full responsibility for the content and accuracy of the work.

\section{Tools} \label{sec:tools}

In this section we review and develop the tools required in our proof. In particular, we discuss factorizable sequences of sheaves, Clifford algebras, spin modules, and torus localization. Some of these results may be of independent interest.

\subsection{Factorizability} \label{sec:factorizability}

In the work \cite{Oko}, Okounkov introduced the notion of a factorizable sequence of sheaves.
We here review the definition and the main result. Our formulation differs slightly from the one given in \cite{Oko}, see Remark \ref{rmk:diffoko}.

Let $X$ be a smooth, quasi-projective variety acted on by an algebraic torus $T$. Denote by $\Sym^n(X) = X^{\times n} / S_n$ the $n$th symmetric product of $X$. Let $n_1, \ldots, n_k >0$ and $N = \sum_i n_i$. Then there exists a natural morphism
$$
\sigma_{n_1, \ldots, n_k} \colon \prod_{i=1}^{k} \Sym^{n_i}(X) \to \Sym^N(X). 
$$
We define the open subset
$$
U_{n_1, \ldots, n_k} \subset \prod_{i=1}^{k} \Sym^{n_i}(X)
$$
of $(x_1, \ldots, x_k)$ such that the $x_i \in \Sym^{n_i}(X)$ have pairwise disjoint supports. The composition $U_{n_1, \ldots, n_k}  \to \Sym^N(X)$ is \'etale of degree $\binom{N}{n_1,\dots,n_k}$.\footnote{To see this, let $H = \prod_{i=1}^k S_{n_i} \subseteq S_N$, and let $V \subseteq X^N$ be the natural open set such that $U_{n_1,\dots,n_k} = V/H$. Every point in $V$ then has $S_N$-stabilizer contained in $H$, and so by \cite[Exp.~V, Prop.~2.2]{SGA1}, the map $V/H \to V/S_N \subseteq X^N/S_N = \Sym^N(X)$ is étale.} The \'etale property implies that certain sheaves and schemes living over $\prod_i \Sym^{n_i}(X)$ can be related to the analogous objects living on $\Sym^N(X)$ provided one restricts to $U_{n_1, \ldots, n_k}$. As an example, consider the Hilbert--Chow morphisms 
$$
\nu_{n_1, \ldots, n_k} \colon \prod_{i=1}^{k} \Hilb^{n_i}(X) \rightarrow \prod_{i=1}^{k} \Sym^{n_i}(X).
$$
We do not have a well-defined ``union'' morphism $\prod_{i=1}^k \Hilb^{n_i}(X) \to \Hilb^N(X)$, since the union of $k$ schemes of lengths $n_i$ can have length less than $N$. 
However, if the subschemes have disjoint support, this does not happen.
Thus taking the union gives a well-defined morphism over the open subset $U_{n_1, \ldots, n_k}$, and in fact we have a cartesian diagram
\[
\begin{tikzcd}
 \prod_{i=1}^{k} \Hilb^{n_i}(X) |_{U_{n_1, \ldots, n_k}} \arrow[r, " "]  \arrow[d, "\nu_{n_1, \ldots, n_k}"] & \Hilb^N(X)  \arrow[d, "\nu_N"] \\
 U_{n_1, \ldots, n_k} \arrow[r, "\sigma_{n_1, \ldots, n_k}"] & \Sym^N(X)
\end{tikzcd}
\]
where the horizontal arrows are \'etale.

We record the following proposition for future reference. 
\begin{proposition}
\label{thm:torsionInPicardGroup}
Let $d \ge 2$, and let $n_1,\dots, n_k \ge 1$.
Let $G$ be a connected algebraic group acting on $Y := \prod_{i=1}^{k} \Hilb^{n_i}(\C^d) |_{U_{n_1, \ldots, n_k}}$.
Then the group $\Pic^G(Y)$ is torsion free.
\end{proposition}
\begin{proof}
Combine Lemmas \ref{thm:YIsSimplyConnected} and \ref{thm:simplyConnectedMeansTorsionFree}.
\end{proof}

For $d, n \ge 1$, let $\Hilb^n(\C^d,0) \subseteq \Hilb^n(\C^d)$ be the punctual Hilbert scheme parametrizing subschemes supported at $0 \in \C^d$.
\begin{lemma}
\label{thm:HilbIsSimplyConnected}
    For all $d, n \ge 1$, the scheme $\Hilb^n(\C^d,0)$ is simply connected.
\end{lemma}
\begin{proof}
We adapt the proof of the main result in \cite{Hor}.
For this, let $\mathfrak m =(x_1,\dots, x_d)\subseteq \C[x_1,\dots, x_d]$, let $R = \C[x_1,\dots, x_d]/\mathfrak m^n$, and let $X = \Spec R$. We then have that $\Hilb^n(\C^d,0) \cong \Hilb^n(X)$ (taking the reduced scheme structure for both).
By \cite{Fog}, $\Hilb^n(X)$ is the fixed point locus of a unipotent group acting on a Grassmannian variety $\Gr$.
By \cite{Fuj}, it follows that $\pi_0(\Hilb^n(X)) = \pi_0(\Gr)$ and $\pi_1(\Hilb^n(X)) = \pi_1(\Gr)$, so $\Hilb^n(X)$ is simply connected.
\end{proof}

\begin{lemma}
\label{thm:YIsSimplyConnected}
With $Y$ as in Proposition \ref{thm:torsionInPicardGroup}, we have that $Y$ is simply connected.
\end{lemma}
\begin{proof}
Each fibre of the morphism $\nu_{n_1,\dots, n_k}$ is homeomorphic to a product of schemes of the form $\Hilb^i(\C^d,0)$, each of which by Lemma \ref{thm:HilbIsSimplyConnected} is simply connected.
By \cite{Sma}, we find that $\nu_{n_1,\dots,n_k}$ induces isomorphisms between $\pi_0$ and $\pi_1$ of its domain and codomain.
The space $U_{n_1,\dots,n_k}$ is a quotient of an open subset $V \subseteq \prod_{i=1}^k(\C^d)^{n_i}$ by the group $\prod_{i=1}^k S_{n_i}$.
The complement of $V$ has codimension $d \ge 2$, and hence $V$ is simply connected.
Since every element of $\prod_{i=1}^k S_{n_i}$ fixes some point in $V$, we then have $\pi_1(U_{n_1,\dots, n_k}) = 0$ by \cite{Arm}.
Thus $\pi_1(Y)= 0$.
\end{proof}

\begin{lemma}
\label{thm:simplyConnectedMeansTorsionFree}
If $Y$ is a simply connected scheme and $G$ is a connected algebraic group acting on $Y$, then $\Pic^G(Y)$ is torsion free.
\end{lemma}
\begin{proof}
Let $L$ be a $G$-equivariant line bundle on $Y$ which is $n$-torsion in $\Pic^G(Y)$, that is, such that $L^{\otimes n} \cong \cO_Y$ with the trivial $G$-equivariant structure.
With $\mathbb V(L)$ as the associated affine space bundle over $Y$, we have a generically $n$-to-1 morphism $\mathbb V(L) \to \mathbb V(\cO_Y) = Y \times \AA^1$.
The inverse image $Z$ of $Y \times \{1\}$ under this morphism gives a degree $n$ étale covering of $Y$, which is moreover $G$-invariant.
Since $Y$ is simply connected, $Z$ is the disjoint union of $n$ copies of $Y$, and since $G$ is connected, each of these components is preserved by $G$.
The isomorphism from $Y$ to one of these components defines a trivialization of $L$ as a $G$-equivariant line bundle.
\end{proof}

\begin{definition}
\label{def:factorisableSequence}
Let $\{\F_n\}_{n=1}^{\infty}$ be a sequence where each $\F_n$ is a $T$-equivariant $\Z/2$-graded coherent sheaf on $\Sym^n(X)$. Then \emph{factorization data} for $\{\F_n\}_{n=1}^{\infty}$ consists of a collection of $T$-equivariant $\Z/2$-graded isomorphisms
$$
\phi_{m,n} \colon \F_m \boxtimes \F_n|_{U_{m,n}} \stackrel{\cong}{\to} \F_{m+n}|_{U_{m,n}},
$$
for all $m,n \geq 1$, satisfying the following properties:
\begin{enumerate}
\item \textbf{(Associativity)} For any $m,n,p > 0$, we have $$\phi_{m,n+p} \circ (\id_{\F_m} \boxtimes \phi_{n,p}) = \phi_{m+n,p} \circ (\phi_{m,n} \boxtimes \id_{\F_p})$$ as isomorphisms on $U_{m,n,p}$.
\item \textbf{(Commutativity)} Let $m,n >0$ and denote by $$\tau \colon \Sym^m(X) \times \Sym^n(X) \to \Sym^n(X) \times \Sym^m(X)$$ the permutation map. Then the following diagram commutes
\begin{displaymath}
\xymatrix
{
\F_m \boxtimes \F_n |_{U_{m,n}} \ar^{\phi_{m,n}}[r] \ar_{\cong}^{v}[d] & \F_{m+n}|_{U_{m,n}} \ar_{\cong}^{w}[d] \\ 
\tau^*(\F_n \boxtimes \F_m |_{U_{n,m}}) \ar^{\tau^*\phi_{n,m}}[r] & \tau^*(\F_{n+m}|_{U_{n,m}})
}
\end{displaymath}
where the vertical arrows are canonical isomorphisms induced by the canonical isomorphism $\F_m \boxtimes \F_n \cong \tau^* (\F_n \boxtimes \F_m)$ and the commutative diagram
\begin{displaymath}
\xymatrix
{
U_{m,n} \ar^{\tau}[r] \ar_{\sigma_{m,n}}[dr] & U_{n,m} \ar^{\sigma_{n,m}}[d] \\
& \Sym^{m+n}(X)
}
\end{displaymath}
\end{enumerate}
We say that $\{\F_n\}_{n=1}^{\infty}$ is \emph{factorizable} if there exists factorization data for $\{\F_n\}_{n=1}^{\infty}$. 
\end{definition}

\noindent \textbf{Convention.} All sheaves in factorizable sequences, or more generally sheaves appearing in a sequence $\{\F_n\}_{n=1}^\infty$ where $\F_n$ lives on $\Sym^n(X)$, are assumed $T$-equivariant and $\ZZ/2$-graded, and homomorphisms between them are assumed to be $T$-equivariant and to respect the $\ZZ/2$-grading. We usually omit these adjectives.

\begin{remark} \label{rmk:diffoko}
Our definition relates to the one of Okounkov \cite{Oko} as follows:
Associativity implies that for any $n_1, \ldots, n_k > 0$ with $N = \sum_{i} n_i$, we have well-defined isomorphisms 
$$
\phi_{n_1, \ldots, n_k} \colon \bigboxtimes_{i=1}^k \F_{n_i} |_{U_{n_1, \ldots, n_k}} \stackrel{\cong}{\to} \F_{N}|_{U_{n_1, \ldots, n_k}}.
$$
Suppose we write $n_1, \ldots, n_k$ as a length $k$ partition $1^{m_1} 2^{m_2} \ldots$ of $N$, then the sheaf $\bigboxtimes_{i=1}^k \F_{n_i}$ descends to a sheaf 
$$
\bigboxtimes_{j} \Sym^{m_j}(\F_{j}) \quad \textrm{on} \quad \prod_j \Sym^{m_j} \Sym^j(X).
$$
The image  $U \subset \prod_j \Sym^{m_j} \Sym^j(X)$ of $U_{n_1, \ldots, n_k}$ is an open subset, and commutativity guarantees that the isomorphisms $\phi_{n_1, \ldots, n_k}$ descend to isomorphisms
$$
\bigboxtimes_j \Sym^{m_j}(\F_{j}) |_{U} \stackrel{\cong}{\to} \F_{N}|_{U}.
$$
\end{remark}

We have the following famous ``lemma'' of Okounkov \cite[Lem.~5.3.4]{Oko}:
\begin{lemma}[Okounkov] \label{prop:Oko}
	Let $\{\F_n\}_{n=1}^{\infty}$ be a factorizable sequence on $\Sym^{\mdot}(X)$. Then there exist classes $\G_n \in K_0^{T}(X)$, for all $n \geq 1$, such that
	$$
	1 + \sum_{n=1}^{\infty} \chi(\Sym^n(X),\F_n) \, q^n = \mathrm{Exp}\Bigg( \sum_{n=1}^{\infty} \chi(X,\G_n) \, q^n \Bigg).
	$$
\end{lemma}

A detailed and self-contained proof is provided in \cite{Ren}. For a recent application of the factorizability method to the setting of $K$-theoretic Donaldson--Thomas invariants of $[\C^2 / \mu_2] \times \C$, we refer to \cite{Thi}.

The following useful proposition states that the commutative diagrams in the associativity and commutativity axioms in the definition of a factorizable sequence only have to be checked up to non-zero constants.
\begin{proposition} \label{prop:uptocst}
Let $\{\F_n\}_{n=1}^{\infty}$ be a sequence of coherent sheaves on $\Sym^\mdot(X)$, and assume we have isomorphisms
$$
\phi_{m,n} \colon \F_m \boxtimes \F_n|_{U_{m,n}} \stackrel{\cong}{\to} \F_{m+n}|_{U_{m,n}}
$$
for all $m,n > 0$. 
Assume the support of $\F_1$ has dimension $\geq 1$, and that the following hold:
\begin{enumerate}
\item For any $m,n,p > 0$, there exists a constant $c_{m,n,p} \in \C^*$ such that $$\phi_{m,n+p} \circ (\id_{\F_m} \boxtimes \phi_{n,p}) = c_{m,n,p} \cdot \phi_{m+n,p} \circ (\phi_{m,n} \boxtimes \id_{\F_p})$$ as isomorphisms on $U_{m,n,p}$.
\item For any $m,n >0$ there exist constants $d_{m,n }\in \C^*$ such that $w \circ \phi_{m,n}= d_{m,n} \cdot \tau^* \phi_{n,m} \circ v$, where $v$ and $w$ are as in Definition \ref{def:factorisableSequence}.
\end{enumerate}

Then
\begin{enumerate}[label=(\alph*)]
    \item We may find scalars $\lambda_{m,n}$ such that, by replacing $\phi_{m,n}$ with $\lambda_{m,n}\phi_{m,n}$ we may assume $c_{m,n,p} = 1$ for all $m,n,p$.
    \item After such a rescaling, we have either $d_{1,1} = 1$ or $d_{1,1} = -1$. 
If $d_{1,1} = 1$, the sequence $\{\F_n\}_{n=1}^{\infty}$ is factorizable.  
If $d_{1,1} = -1$, the sequence $\{\F_n[n]\}_{n=1}^{\infty}$ is factorizable, where $[n]$ denotes the shift in $\Z/2$-grading.
\end{enumerate}
\end{proposition}
\begin{proof}
We will define scalars $\lambda_{m,n}$ as in the statement, and let $\psi_{m,n} = \lambda_{m,n}\phi_{m,n}$.
First, let $\lambda_{1,n} = 1$ for all $n > 0$.

For any $k>0$, define
$$
\psi_{1^k} = \phi_{1,k-1} \circ (\id_{\F_1} \boxtimes \phi_{1,k-2}) \circ \cdots \circ (\id_{\F_1} \boxtimes \cdots \boxtimes \id_{\F_1} \boxtimes \phi_{1,1})
$$
as an isomorphism from $\F_1 \boxtimes \cdots \boxtimes \F_1 |_{U_{1, \ldots, 1}} \to \F_k |_{U_{1, \ldots, 1}}$. 
By repeated use of (1), we find that there is a scalar $\lambda_{m,n}$ such that 
\[
\lambda_{m,n}\phi_{m,n} \circ (\psi_{1^m} \boxtimes \psi_{1^n}) = \psi_{1^{m+n}},
\]
which gives
\begin{equation} \label{eqn:redef1}
\psi_{m,n} \circ (\psi_{1^m} \boxtimes \psi_{1^n}) = \psi_{1^{m+n}}.
\end{equation}
Taking $c'_{m,n,p} = c_{m,n,p}\lambda_{m,n+p}\lambda_{n,p}\lambda_{m+n,p}^{-1}\lambda_{m,n}^{-1}$, we have 
\begin{equation} \label{eqn:redef2}
\psi_{m,n+p} \circ (\id_{\F_m} \boxtimes \psi_{n,p}) = c'_{m,n,p} \cdot \psi_{m+n,p} \circ (\psi_{m,n} \boxtimes \id_{\F_p})
\end{equation}
as isomorphisms on $U_{m,n,p}$ for all $m,n,p$. Precomposing \eqref{eqn:redef2} by $\psi_{1^m} \boxtimes \psi_{1^n} \boxtimes \psi_{1^p}$ and using \eqref{eqn:redef1}, we find
$$
\psi_{1^{m+n+p}} = c'_{m,n,p} \cdot \psi_{1^{m+n+p}}.
$$
Since $\F_1$ has dimension $\geq 1$, we have that $\psi_{1^k}$ is nonzero and $c'_{m,n,p}=1$ for all $m,n,p$.
This proves the claim \textit{(a)}.

For claim \textit{(b)}, we now assume that $\phi_{m,n}$ satisfies \textit{(1)} and \textit{(2)} with $c_{m,n,p}= 1$.
Define a graded $\C$-algebra $A = \bigoplus_{n \geq 0} A_n$ as follows. Let $A_0 = \C$ and 
$$
A_n = \bigoplus_{x \in \Sym^n(X) \setminus \Delta} \F_n|_x,
$$
where $\Delta$ denotes the big diagonal. The isomorphisms $\phi_{m,n}$ induce multiplication maps
$$
\mu_{x,y} \colon \F_m|_x \otimes \F_n|_y \stackrel{\cong}{\to} \F_{m+n}|_{x+y},
$$
for all $x \in \Sym^m(X)\setminus \Delta$ and $y \in \Sym^n(X) \setminus \Delta$ with disjoint support.
Define multiplication maps
$$
\mu_{m,n} \colon A_m \otimes_{\C} A_n \to A_{m+n}
$$
by summing the $\mu_{x,y}$ over all $x,y$ with disjoint support. Then property \textit{(1)} of the $\phi_{m,n}$ implies that $A$ is an associative $\C$-algebra. 

As $\cF_1$ has support of dimension $\ge 1$, we can find an infinite sequence of distinct points $(p_i)_{i=1}^\infty$ in the support of $\cF_1$.
Taking $x = \sum_{i=1}^m p_i$, $y = \sum_{i=m+1}^{m+n}p_i$, we have $\mu_{x,y}$ is non-zero.
It follows that $\mu_{m,n}$ is non-zero for all $m,n$.
Property \textit{(2)} implies that $a \cdot b = d_{m,n} b \cdot a$ for all $a \in A_m$ and $b \in A_n$.
Therefore, we have $d_{m,n} d_{n,m} = 1$ and in particular $d_{1,1} = \pm 1$. Similarly, we obtain $d_{m,n} = (d_{1,1})^{mn}$ (again use that the support of $\F_1$ has dimension $\geq 1$ to ensure that the map $\F_1^{\otimes m} \to \F_m$ is non-0). 
For $d_{1,1} = 1$, we have obtained factorization data. 
For $d_{1,1} = -1$, we  obtain factorization data for $\{\F_n[n]\}_{n=1}^{\infty}$ because $d_{m,n}$ gets replaced by $(-1)^{mn} d_{m,n}$. This is due to the fact that 
\begin{displaymath}
\xymatrix
{
E[m] \otimes F[n] \ar[r] \ar[d] & (E \otimes F)[m+n] \ar[d] \\
F[n] \otimes E[m] \ar[r] & (F \otimes E)[m+n] \\
}
\end{displaymath}
only commutes up to the sign $(-1)^{mn}$ for any $T$-equivariant $\Z/2$-graded coherent sheaves $E,F$.
\end{proof}

\subsection{Clifford algebras and spin modules}

Let $X$ be a $\C$-scheme.
Throughout this section, let $E$ be an even rank $2n \geq 2$ vector bundle on $X$, with non-degenerate quadratic form $q \colon \Sym^2 E \to \cO_X$ and compatible orientation $o$. We say that $(E,q,o)$ is an oriented quadratic vector bundle. 

Let $C(E,q)$ be the Clifford algebra, that is the sheaf of non-commutative $\O_X$-algebras on $X$ defined as the quotient of the tensor algebra $T(E)$ by the ideal generated (locally) by relations 
\begin{equation} \label{eqn:fundamentalrel}
s \otimes s - q(s,s),
\end{equation}
for $s$ a section of $E$ \cite{Bou, Knu}.
Going forward, we suppress the quadratic form from the notation and write $C(E)$ for $C(E,q)$.

The Clifford algebra $C(E)$ is a locally free $\cO_X$-module of rank $2^{2n}$. 
The canonical homomorphism $E \to T(E) \to C(E)$ is injective, and we use this to consider sections of $E$ as sections of $C(E)$.
The algebra $C(E)$ is $\ZZ/2$-graded by declaring that $\prod_{i=1}^m s_i$ has degree $m \pmod 2$ for sections $s_1, \dots, s_m$ of $E$.

There is a unique ``order-reversing'' isomorphism $\rev \colon C(E) \to C(E)^{\op}$ of $\cO_X$-algebras mapping a section $s$ of $E$ to itself, so if $s_1,\dots, s_m$ are sections of $E$, we have
\[
\rev(s_1\cdots s_m) = s_m\cdots s_1.
\]

Let $M$ be a left $C(E)$-module, i.e.~a sheaf of $\cO_X$-modules equipped with a homomorphism\footnote{When we write $\otimes$, we mean $\otimes_{\O_X}$ and likewise $\sHom$ stands for $\sHom_{\cO_X}$.} 
\[
C(E) \otimes M \to M
\]
satisfying the usual module properties.
By the isomorphism $\rev$, we can consider $M$ as a $C(E)^{\mathrm{op}}$-module, or equivalently as a right $C(E)$-module.
We write $M^\dagger$ to denote $M$ considered as a right $C(E)$-module.

If $\Lambda$ is a maximal isotropic subbundle of $E$, then there is an injective homomorphism $i_\Lambda \colon \det \Lambda \to C(E)$ of $\cO_X$-modules which sends a local section $e_1 \wedge \cdots \wedge e_n$ of $\det \Lambda$ to $e_1\cdots e_n$.
We write $C(E)\Lambda$ and $\Lambda C(E)$ for the left and right ideals of $C(E)$ generated by $\det \Lambda$.
The isomorphism $\rev$ sends $C(E)\Lambda$ to $\Lambda C(E)$, and so gives an isomorphism of right $C(E)$-modules
\begin{equation}
	\label{eqn:lambdaOpposite}
\rev \colon (C(E)\Lambda)^\dagger \cong \Lambda C(E).
\end{equation}

The algebra $C(E)$ is an Azumaya algebra on $X$, that is it is étale locally on $X$ isomorphic to $\sHom(F, F)$, where $F$ is a locally free $\cO_X$-module.
If $\Lambda$ is a maximal isotropic subbundle, then we may take $F = C(E)\Lambda$, more precisely the homomorphism of $\cO_X$-algebras 
\[
C(E) \to \sHom(C(E)\Lambda, C(E)\Lambda)
\]
induced by the multiplication in $C(E)$ is an isomorphism of $\cO_X$-algebras. In this case, $C(E)$ is Zariski locally $\sHom(F, F)$, where $F$ is a locally free $\cO_X$-module. 
Note that after restriction to a closed point, $C(E)\Lambda$ is the unique (up to isomorphism) simple $C(E)$-module.

Given a maximal isotropic subbundle $\Lambda \subset E$ with local frame $e_1,\dots, e_n$, by a version of the Gram--Schmidt procedure (\cite[Sect.~2.3]{OT}), one can construct local sections $f_1,\dots, f_n$ of $E$ such that 
\begin{equation} \label{eqn:usualbasis}
q(e_i,e_j) = q(f_i,f_j) = 0, \quad q(e_i,f_j) = \delta_{ij},
\end{equation}
for all $i$ and $j$. We will use this notation throughout this section.

\begin{lemma}
	\label{thm:detIsomorphism}
	Let $\Lambda$ be a maximal isotropic subbundle of $E$.
	There is an isomorphism of line bundles
	\[
	\det \Lambda \cong \Lambda C(E) \otimes_{C(E)} C(E)\Lambda,
	\]
	%defined 
	defined on a local frame $e_1,\dots, e_n$ of $\Lambda$ by
    %, and let $f_1,\dots, f_n$ be sections of $E$ such that $q(e_i,e_j) = q(f_i,f_j) = 0$ and $q(e_i,f_j) = \delta_{ij}$ for all $i$ and $j$.
	%the isomorphism is defined by
	\[
	e_1 \wedge \dots \wedge e_n \mapsto e_1\cdots e_n f_1\cdots f_n \otimes e_1 \cdots e_n.
	\]
\end{lemma}
\begin{proof}
	We first show that the local homomorphism in the statement is an isomorphism.
	As $\Lambda C(E)$ and $C(E)\Lambda$ are locally free right and left modules of rank $2^n$ for the Azumaya algebra $C(E)$, it is clear that $\Lambda C(E) \otimes_{C(E)} C(E)\Lambda$ is a line bundle.
	So it is enough to check that the homomorphism is surjective.
	For $I \subseteq [n]:=\{1, \ldots, n\}$, let $e_I = \prod_{i \in I} e_i$ and $f_I = \prod_{i \in I} f_i$, the product taken in ascending order of $i$.
		
	Then $\Lambda C(E) \otimes_{C(E)} C(E)\Lambda$ is spanned as an $\cO_X$-module by sections
	\[
	e_{[n]} f_I \otimes f_J e_{[n]} = e_{[n]} f_If_J \otimes e_{[n]} = \begin{cases} 0 \text{ if }I \cap J \not= \varnothing \\
		\pm e_{[n]} f_{I \cup J} \otimes e_{[n]} \text{ if } I \cap J = \varnothing.	\end{cases}
	\]
	If $I \subsetneq [n]$, then $e_{[n]} f_{[n]} e_{[n] \setminus I} = \pm 2^{n - |I|} e_{[n]} f_I$ by \eqref{eqn:fundamentalrel} and \eqref{eqn:usualbasis}. Thus
	\begin{equation}
		\label{eqn:smallIVanish}
	e_{[n]}f_I \otimes e_{[n]} = \pm 2^{|I| - n} e_{[n]}f_{[n]}e_{[n] \setminus I} \otimes e_{[n]} = \pm 2^{|I| - n}e_{[n]}f_{[n]} \otimes e_{[n] \setminus I}e_{[n]} = 0.
	\end{equation}
	It follows that $\Lambda C(E) \otimes_{C(E)} C(E)\Lambda$ is locally spanned by $e_{[n]}f_{[n]} \otimes e_{[n]}$ (giving a second proof that it is a line bundle), so our homomorphism is surjective.

	We next show that the local homomorphism defined above is independent of the choice of $f_i$.
	Fixing the $e_i$, if $f_1,\dots, f_n$ and $f_1', \dots, f_n'$ are two valid choices of $f_i$, then since $f_i - f'_i$ is orthogonal to $e_1,\dots, e_n$, we can find functions $a_{ij}$ locally on $X$ such that $f_i' = f_i + \sum_j a_{ij}e_j$.
	It is then easy to see that
	\[
	e_{[n]}f_{[n]} - e_{[n]}f'_{[n]}
	\]
	is a sum of terms $a_Ie_{[n]}f_{I}$ with $I \subsetneq [n]$.
	By \eqref{eqn:smallIVanish}, we get that 
	\[
	e_{[n]}f_{[n]}\otimes e_{[n]} = e_{[n]}f_{[n]}' \otimes e_{[n]}.
	\]
	Finally, two bases $e_1,\dots, e_n$ and $e'_1,\dots, e_n'$ of $\Lambda$ can be related by a sequence of rescalings $e_i \mapsto ae_i$ and transformations $e_i \mapsto e_i + e_j$, correspondingly changing $f_i \mapsto a^{-1}f_i$ and $f_j \mapsto f_j - f_i$.
	Each of these gives rise to the same morphism.
	Since the homomorphism is independent both of the choice of $e_i$ and $f_i$, it is well-defined globally.
\end{proof}

The relevance of spin geometry in the Donaldson--Thomas theory of Calabi--Yau 4-folds was recently worked out by N.~Kuhn \cite{Kuh}.
\begin{definition}
	A \emph{spin module} for $(E,q)$ is a left $C(E)$-module $S$, (Zariski) locally free of finite rank as $\O_X$-module, such that there exists an isomorphism of $\cO_X$-modules
	\begin{equation}
		\label{eqn:spinModuleStructure}
		\phi \colon S^\dagger \otimes_{C(E)} S \cong \cO_X.
	\end{equation}
\end{definition}
For given $(E,q)$, the natural category of pairs $(S,\phi)$, where $S$ is a spin module and $\phi$ an isomorphism, is equivalent to the category of $\text{Spin}(2n)$-structures on the $SO(2n)$-bundle $E$; see \cite[Sect.~A.6]{Kuh} for a parallel discussion of $\text{Spin}^{\mathbb C}(2n)$-structures.

Let $S$ be a spin module, and let $S^\vee = \sHom(S, \cO_X)$, which is a right $C(E)$-module.
We have that the homomorphism 
\[
C(E) \to \sHom(S, S) = S \otimes S^\vee
\]
is an isomorphism of $C(E)$-bimodules (this can be checked locally, using e.g.~Remark \ref{rmk:imageOfLambdaSquareRoot}).
The isomorphism \eqref{eqn:spinModuleStructure} induces an isomorphism of right $C(E)$-modules $S^\dagger \cong S^\vee$, and so we get an isomorphism of $C(E)$-bimodules
\[
S \otimes S^\dagger \cong C(E).
\]

\begin{example} \label{ex:sqrtdetgivesspin}
	Assume that $E$ has a maximal isotropic subbundle $\Lambda$, and that $L$ is a square root of $\det \Lambda$.
	Then $C(E)\Lambda \otimes L^{-1}$ is a spin module for $E$, since
	\begin{align*}
	(C(E)\Lambda \otimes L^{-1})^\dagger \otimes_{C(E)} (C(E)\Lambda \otimes L^{-1}) &\cong L^{-2} \otimes \Lambda C(E) \otimes_{C(E)} C(E)\Lambda \\
	&\cong \det(\Lambda)^{-1} \otimes \det(\Lambda) \cong \cO_X,
	\end{align*}
using the isomorphisms from \eqref{eqn:lambdaOpposite} and Lemma \ref{thm:detIsomorphism}.
\end{example}

Conversely, if $\Lambda$ is a maximal isotropic subbundle of $E$ and $S$ is a spin module, then we have an isomorphism
\begin{align*}
	(S^\dagger \otimes_{C(E)} C(E)\Lambda)^{\dagger} \otimes S^\dagger \otimes_{C(E)} C(E)\Lambda &\cong \Lambda C(E) \otimes_{C(E)} S \otimes S^\dagger \otimes_{C(E)} C(E)\Lambda \\
	&\cong \Lambda C(E) \otimes_{C(E)} C(E) \Lambda \cong \det \Lambda.
\end{align*}
It follows that
\begin{equation}
	\label{eqn:squarerootDefinition}
(\det \Lambda)_S^{1/2} :=  \Lambda C(E) \otimes_{C(E)} S
\end{equation}
is a square root of $\det \Lambda$.
We have an inclusion of $\cO_X$-modules
\begin{equation}
	\label{eqn:inclusionOfSquareRoot}
(\det \Lambda)^{1/2}_S = \Lambda C(E) \otimes_{C(E)} S \to C(E) \otimes_{C(E)} S \to S.
\end{equation}
\begin{remark}
	\label{rmk:imageOfLambdaSquareRoot}
If $S = C(E)\Lambda \otimes L^{-1}$, with $L$ a square root of $\det(\Lambda)$, then a simple computation shows that the image of $(\det \Lambda)^{1/2}_S \to S$ is the submodule spanned by $e_1\cdots e_n \otimes t$, for some local basis $e_1,\dots, e_n$ of $\Lambda$ and trivializing section $t$ of $L^{-1}$. We note that \'etale locally, any spin module $S$ can be expressed as $S \cong C(E)\Lambda$ for some maximal isotropic subbundle $\Lambda$. Indeed, étale locally such a $\Lambda$ exists. The Morita equivalence between $C(E)$ and $\cO_X$ sends both $S$ and $C(E)\Lambda$ to line bundles.
Étale locally, these line bundles are isomorphic, and so $S \cong C(E)\Lambda$.
\end{remark}

So far, we have not used the orientation $o$ on $(E,q)$. The orientation allows us to endow each left $C(E)$-module with a $\Z/2$-grading.
\begin{definition}
	The orientation $o$ of $E$ induces a \textit{volume element} $\omega_o$ in the Clifford algebra, see \cite[Def.~A.11]{Kuh}.
	Working \'etale locally, we may choose $\Lambda$ and $\Lambda'$ to be dual maximal isotropic subbundles of $E$, with $e_1,\dots, e_n$ a basis of $\Lambda$ and $f_1, \dots f_n$ (as above) a dual basis of $\Lambda'$ 
    %such that $q(e_i,e_j) = q(f_i,f_j) = 0$ and $q(e_i,f_j) = \delta_{ij}$ for all $i,j=1, \ldots, n$, 
    and moreover such that $\Lambda$ is positive.
	Then
	\[
	\omega_o = (e_1f_1 -1)\cdots(e_n f_n -1).
	\]
\end{definition}
We have $\omega_o^2 = 1$, and for every section $s$ of $E$, that 
\begin{equation}
	\label{eqn:commutatorOmega}
	\omega_o s + s\omega_o = 0.
\end{equation}
If $M$ is a left $C(E)$-module, then the action of $\omega_o$ on $M$ induces a splitting $M = M_0 \oplus M_1$ of $\cO_X$-modules, where sections of $M_j$ satisfy $\omega_o m = (-1)^jm$.
By \eqref{eqn:commutatorOmega}, if $m \in M$ is such that $\omega_o m = \pm m$ and $s$ is a section of $E$, then $\omega_o sm = \mp sm$.
It follows that the splitting $M_0 \oplus M_1$ gives $M$ a structure of $\ZZ/2$-graded $C(E)$-module with respect to the $\ZZ/2$-grading of $C(E)$.

\begin{definition} \label{def:complex}
	Let $s$ be an isotropic section of $E$, and let $S$ be a spin module.
	Define $(S,s)$ to be the $\ZZ/2$-graded complex whose underlying sheaf is the $\ZZ/2$-graded $\cO_X$-module $S$, and with differential given by $d(m) = sm$.
\end{definition}
It follows from the fact that $s$ is isotropic that $s^2 = 0$, so that $d^2(m) = s^2m = 0$, i.e., $d$ is a differential.

\begin{example} \label{ex:reltoperiodcx}
	Suppose $\Lambda \subset E$ is positive maximal isotropic, $E = \Lambda \oplus \Lambda^\vee$, and $s$ is an isotropic section. Suppose furthermore that $L$ is a square root of $\det(\Lambda)$. Then $S := L^{-1} \otimes C(E) \Lambda $ is a spin module (Example \ref{ex:sqrtdetgivesspin}) and
	\[
	L \otimes \Lambda^\mdot \Lambda^\vee \cong S
	\]
	as $\Z/2$-graded $C(E)$-modules, where the map, in the usual local basis, and with $t$ a section of $L$ squaring to $e_1 \wedge \dots \wedge e_n$, is induced by 
    \[
    t \otimes f_{i_1} \wedge \cdots \wedge f_{i_k} \mapsto t^\vee \otimes f_{i_1} \cdots f_{i_k} e_1 \cdots e_n.
    \]
    The differential $s = (s_{\Lambda},s_{\Lambda^\vee})$ then corresponds to $s_{\Lambda^\vee} \wedge (-) + 2 s_{\Lambda}\intprod (-)$ on $L \otimes \Lambda^\mdot \Lambda^\vee$ and we obtain a 2-periodic complex
	\[
	\cdots \stackrel{s_{\Lambda^\vee} \wedge + 2 s_{\Lambda}\intprod}{\xrightarrow{\hspace*{1.5cm}}} L \otimes \Lambda^{\mathrm{even}} \Lambda^\vee \stackrel{s_{\Lambda^\vee} \wedge + 2 s_{\Lambda}\intprod}{\xrightarrow{\hspace*{1.5cm}}} L \otimes \Lambda^{\mathrm{odd}} \Lambda^\vee \stackrel{s_{\Lambda^\vee} \wedge + 2 s_{\Lambda}\intprod}{\xrightarrow{\hspace*{1.5cm}}} \cdots 
	\]
\end{example}
	
It is well-known that the cohomology sheaves of the 2-periodic complex of the previous example are supported on $Z(s) = Z(s_{\Lambda}) \cap Z(s_{\Lambda^\vee})$ \cite[Rem.~3.2]{OS}. Since any spin module is \'etale locally of the above form, we deduce the following lemma.

\begin{lemma} \label{lem:suppZ(s)}
For any spin module $S$ and isotropic section $s$ of $(E,q)$, the sheaf $\mathcal{H}(S,s) = H^{-1}(S,s)[1] \oplus H^{0}(S,s)$ is supported on $Z(s)$.
\end{lemma}

We now turn our attention to tensor products of spin modules.
	Given oriented quadratic vector bundles $(E_1,q_1,o_1), (E_2,q_2,o_2)$ of even rank on $X$, the quadratic bundle $(E_1 \oplus E_2,q_1 \oplus q_2)$ has an induced orientation $o_1 \otimes o_2$. 

\begin{definition}
	If $A, B$ are two $\ZZ/2$-graded sheaves of $\O_X$-algebras, then we write $A \, \widehat{\otimes} \, B$ for the algebra $A \otimes B$, equipped with the natural tensor product grading, and with the product for homogeneous sections $a,a'$ of $A$ and $b, b'$ of $B$ determined by
	\[
	(a \otimes b)(a' \otimes b') = (-1)^{|b||a'|}aa' \otimes bb'.
	\]
\end{definition}

The following lemma is straightforward.
\begin{lemma} \label{lem:tensorCliff}
	Given oriented quadratic bundles $(E_1,q_1,o_1), (E_2,q_2,o_2)$ on $X$, we have an isomorphism
	\[
	C(E_1) \, \widehat{\otimes} \, C(E_2) \cong C(E_1 \oplus E_2),
	\]
	defined by
	\[
	s_1 \otimes 1 \mapsto (s_1,0)
	\]
	and 
	\[
	1 \otimes s_2 \mapsto (0,s_2).
	\]
	This satisfies
	\[
	\omega_{o_1} \otimes \omega_{o_2} \mapsto \omega_{o},
	\]
	where $o$ is the orientation $o_1 \otimes o_2$ of $(E_1 \oplus E_2,q_1 \oplus q_2)$.
\end{lemma}
Going forward, we will implicitly identify $C(E_1 \oplus E_2)$ and $C(E_1) \, \widehat{\otimes} \, C(E_2)$.
If $M_1$ and $M_2$ are $\ZZ/2$-graded modules for $C(E_1)$ and $C(E_2)$ respectively, the sheaf $M_1 \otimes M_2$ is a $C(E_1) \, \widehat{\otimes} \, C(E_2)$-module by
\[
(a \otimes b)(m \otimes m') = (-1)^{|b||m|} am \otimes bm'.
\]

\begin{lemma}
	\label{thm:tensorProduct}
	Let $M_1, N_1, M_2, N_2$ be $\ZZ/2$-graded modules of $C(E_1)$ and $C(E_2)$ respectively.
	We have an isomorphism of $\cO_X$-modules
	\[
	(M_1 \otimes M_2)^\dagger \otimes_{C(E_1 \oplus E_2)} (N_1 \otimes N_2) \cong (M_1^\dagger \otimes_{C(E_1)} N_1) \otimes (M_2^\dagger \otimes_{C(E_2)} N_2).
	\]
\end{lemma}
\begin{proof}
	The isomorphism is given on $\ZZ/2$-homogeneous sections by
	\[
	a \otimes b \otimes c \otimes d \mapsto (-1)^{(|a| + |c|)|b|} a\otimes c \otimes b \otimes d.
	\]
	It is straightforward to check that this gives a well-defined homomorphism (that is, respects the relations defining the tensor product), and that the obvious inverse is also well-defined.
\end{proof}

\begin{corollary}
	Given oriented quadratic vector bundles $E_1, E_2$ on $X$, with spin modules $S_1, S_2$ for $E_1, E_2$ respectively, the $C(E_1) \, \widehat{\otimes} \, C(E_2)$-module $S_1 \otimes S_2$ is a spin module for $E_1 \oplus E_2$.\footnote{Note that the tensor product grading of $S_1 \otimes S_2$ coincides with the one coming from the orientation on $E_1 \oplus E_2$ by Lemma \ref{lem:tensorCliff}.}
\end{corollary}
\begin{proof}
	By Lemma \ref{thm:tensorProduct}, we get
	\[
	(S_1 \otimes S_2)^\dagger \otimes_{C(E_1 \oplus E_2)} (S_1 \otimes S_2) \cong (S_1^\dagger \otimes_{C(E_1)} S_1) \otimes (S_2^\dagger \otimes_{C(E_2)} S_2) \cong \cO_X \otimes \cO_X \cong \cO_X.
	\]
\end{proof}

Let $E_1, E_2$ be oriented quadratic vector bundles on $X$, let $S$ be a spin module for $E_1 \oplus E_2$ and $S_1$ a spin module for $E_1$.
Give $S_1^\vee$ a $\ZZ/2$-grading as a right $C(E_1)$-module such that the homomorphism $m \colon S_1 \otimes S_1^\vee \to C(E_1)$ is homogeneous.
Let 
\[
S \setminus S_1 = S_1^\vee \otimes_{C(E_1)} S,
\]
considered as a $C(E_2)$-module by, for a section $s_2$ of $E_2$ and homogeneous sections $a$ of $S_1^\vee$ and $b$ of $S$
\[
s_2(a \otimes b) = (-1)^{|a|} a \otimes s_2b.
\]
\begin{lemma}
	\label{thm:subtractingSpinModules}
	The $C(E_2)$-module $S \setminus S_1$ defined above is a spin module for $E_2$.
\end{lemma}
\begin{proof}
	We have an isomorphism of $C(E_1)$-modules
	\[
	S_1 \otimes (S \setminus S_1) = S_1 \otimes S_1^\vee \otimes_{C(E_1)} S \overset{m \otimes \id_S}{\cong} C(E_1) \otimes_{C(E_1)} S \cong S,
	\]
	and we claim this is also an isomorphism of $C(E_1) \, \widehat{\otimes} \, C(E_2)$-modules.
	It is enough to check that the homomorphism commutes with multiplication by a section $s_2$ of $E_2$.
	So let $a, b, c$, be homogeneous sections of $S_1, S_1^\vee$ and $S$, respectively, then we have
	\[
	s_2(a \otimes (b \otimes c)) = (-1)^{|a|} a \otimes s_2(b\otimes c)  = (-1)^{|a| + |b|} a \otimes b \otimes s_2c.
	\]
	It follows that
	\begin{align*}
	m \otimes \id_S(s_2(a \otimes b \otimes c)) &= (-1)^{|a| + |b|} m(a \otimes b) \otimes s_2 c \\
	&= (-1)^{|m(a \otimes b)|} m(a \otimes b) \otimes s_2 c = s_2(m(a \otimes b) \otimes c).
	\end{align*}
We now have, by Lemma \ref{thm:tensorProduct}, that
\begin{align*}
\cO_X &\cong S^{\dagger} \otimes_{C(E_1 \oplus E_2)} S \cong (S_1 \otimes (S\setminus S_1))^\dagger \otimes_{C(E_1 \oplus E_2)} (S_1 \otimes (S\setminus S_1)) \\
&\cong (S_1^\dagger \otimes_{C(E_1)} S_1) \otimes ((S\setminus S_1)^\dagger \otimes_{C(E_2)} (S\setminus S_1)) \cong (S\setminus S_1)^\dagger \otimes_{C(E_2)} (S\setminus S_1). \qedhere
\end{align*}
\end{proof}
If $\Lambda_1$ and $\Lambda_2$ are maximal isotropic subbundles of quadratic bundles $E_1$ and $E_2$ on $X$, then from \eqref{eqn:inclusionOfSquareRoot} we get inclusions
\begin{gather*}
\det(\Lambda_1)^{1/2}_{S_1} \otimes \det(\Lambda_2)^{1/2}_{S_2} \to S_1 \otimes S_2 \\
\det(\Lambda_1 \oplus \Lambda_2)^{1/2}_{S_1 \otimes S_2} \to S_1 \otimes S_2.
\end{gather*}
These have the same image, and so give an isomorphism 
\begin{equation} \label{eqn:sqrtdetdirectsum}
\det(\Lambda_1)^{1/2}_{S_1} \otimes \det(\Lambda_2)^{1/2}_{S_2} \cong \det(\Lambda_1 \oplus \Lambda_2)^{1/2}_{S_1 \otimes S_2}.
\end{equation}

\subsection{Pullback of spin modules} \label{sec:pullbackspin}

In this section, we consider the following setup.
Let $A$ be a smooth variety of dimension $d$, let $E$ be an oriented quadratic vector bundle of rank $2n$ on $A$ with quadratic form $q$, let $s$ be an isotropic section of $E$, and let $S$ be a spin module for $C(E)$.
Let $A_0$ be a smooth variety of dimension $d_0$, and let $f \colon A_0 \to A$ be an unramified morphism.
Let $E_0$ be an oriented quadratic vector bundle of rank $2n_0$, with quadratic form $q_0$, which embeds as a subbundle of $f^*E$, which we also denote by $E|_{A_0}$, compatibly with the quadratic forms. 
Let $S_0$ be a spin module for $C(E_0)$.
Assume that the section $s$ lies in $E_0$ after restriction to $A_0$, and let $s_0 = f^*s$ be the resulting section of $E_0$, which we will also write as $s|_{A_0}$.
We further assume that $$d - d_0 = n - n_0,$$ and write $m$ for this integer.
As $E_0$ is a quadratic subbundle of $E|_{A_0}$, we have a quadratic bundle $E' = E_0^\perp$ of rank $2m$ with quadratic form $q'$ such that $(E|_{A_0},q) = (E_0 \oplus E', q_0 + q')$. 

We are interested in relating $f^*(S,s)$ (or rather its cohomology sheaves) to $(S_0,s_0)$.
This will be possible under the following assumption, which will be satisfied in the setting we explore in Section \ref{sec:global}.
Let $N = N_{A_0/A} = f^*T_{A}/T_{A_0}$.

\begin{assumption}
\label{assume:LocalForm}
We make the following two assumptions on our data.
\begin{enumerate}
\item Working on $Z(s_0)$, define $\phi$ by the commutative diagram
\[
\begin{tikzcd}
    0\arrow[r]&T_{A_0}\arrow[r]\arrow[d,"ds_0"] &f^*T_A\arrow[r]\arrow[d, "ds"] &N\arrow[r]\arrow[d, "\phi"] &0 \\
    0\arrow[r]&E_0\arrow[r] &f^*E\arrow[r] &E'\arrow[r] &0
\end{tikzcd}
\]
Then $\phi$ is an injective map of vector bundles.

\item For $p \in Z(s_0)$, let $\widehat{\cO}_{A,f(p)}$ be the completion of the local ring at $f(p)$.
We assume that for every $p \in Z(s_0)$, after restricting along the horizontal maps in the diagram
\[
\begin{tikzcd}
\Spec \widehat{\cO}_{A_0,p} \arrow[r] \arrow[d, "\widehat f"] & A_0 \arrow[d, "f"] \\
\Spec \widehat{\cO}_{A,f(p)} \arrow[r] & A,
\end{tikzcd}
\]
the following conditions hold.
There is a choice of local coordinates
\[
\widehat{\cO}_{A,f(p)} \cong \CC[[x_1,\dots, x_{d_0}, y_1,\dots y_{m}]]
\]
and a trivialization
\[
\widehat E \cong \widehat{\cO}_A^{2n},
\]
where $\widehat{\cO}_A^{2n}$ has the quadratic form
\begin{equation} \label{eqn:standardquadraticform}
q(a_1,\dots, a_{2n_0}, b_1, \dots, b_{2m}) = \sum_{i=1}^{n_0}a_ia_{i+n_0} + \sum_{i=1}^{m}b_ib_{i+m},
	\end{equation}
such that 
\begin{enumerate}
    \item $A_0$ is defined by the vanishing of the $y_i$.
    \item the subbundle $E_0 \to E|_{A_0}$ is identified with the inclusion $\cO_{A_0}^{2n_0} \to \cO_A^{2n}|_{A_0}$ in the first $2n_0$ coordinates.
	\item the section $s$ of $E$ has the form
    \begin{equation}
		\label{eqn:specialFormOfS}
	(g_1,\dots, g_{2n_0}, h_1,\dots, h_{2m}),
	\end{equation}
	where each $g_i$ lies in
    \[
    \CC[[x_1,\dots, x_{d_0}]] + (y_1,\dots, y_{m})^2
    \]
    with $(y_1, \ldots, y_m)^2$ viewed as an ideal in $\CC[[x_1,\dots, x_{d_0}, y_1,\dots y_{m}]]$.
\end{enumerate}
\end{enumerate}
\end{assumption}
Note that conditions $(2a)$ and $(2b)$ are normalizing assumptions which can always be obtained in our general setup, and the crux of part $(2)$ is $(2c)$.

\begin{remark}
Assumption \ref{assume:LocalForm} implies the following condition: Formally locally near each $p \in Z(s_0)$, the 3-term obstruction theories on $Z(s_0)$ and $Z(s)$ are related in the sense that we can find a homomorphism represented by the dashed arrow in the following diagram to give an isomorphism of 3-term complexes.
\[
\begin{tikzcd}
T_{A_0} \arrow[r, "ds_0"] \arrow[d, "df"] & E_0 \arrow[r, "(ds_0)^\vee"] \arrow[d, hook] & \Omega_{A_0} \arrow[d, dashed]\\
f^*T_{A} \arrow[r, "ds"] & f^*E \arrow[r, "(ds)^\vee"] & f^*\Omega_{A}.
\end{tikzcd}
\]
We do not know if Assumption \ref{assume:LocalForm} is strictly stronger than this condition.
\end{remark}

With the above normalizations, we (formally locally) have $E' \cong \cO_{A_0}^{2m} \subseteq \cO_{A_0}^{2n}$, consisting of the last $2m$ factors.

\begin{lemma}
\label{thm:strongerLocalForm}
If Assumption \ref{assume:LocalForm} holds, we may find local coordinates $x_i, y_i$ and a trivialization of $E$ such that (2a) and (2b) of Assumption \ref{assume:LocalForm} hold, and such that the section $s$ has the form
\[
(g_1,\dots, g_{2n_0}, y_1,\dots, y_{m},\overbrace{0,\dots, 0}^{m}),
\]
where each $g_i$ lies in $\C[[x_1,\dots, x_{d_0}]]$, that is, it is independent of the $y_j$.
\end{lemma}
\begin{proof}
    Consider the bigrading on monomials of $R = \C[[x_1,\dots, x_{d_0}, y_1,\dots, y_{m}]]$ where $x_i^ay_j^b$ has degree $(a,b)$.
    Write $I_{a,b}$ for the ideal of $R$ generated by all monomials of degree $(a,b)$.
    Since $s|_{A_0} = s_0$ and lies in $E_0$, we have that all $h_i$ vanish on $A_0$, in other words $h_i \in I_{0,1}$.

    Let $s'$ be the section of $E$ defined by the terms of $s$ of degree $(0,1)$.
	Since $s$ has no constant term, the degree $(0,2)$ part of the relation $q(s,s) = 0$ is the same as that in $q(s',s')$.
    This implies that $s'$ is also isotropic.
    By assumption, the $g_i$ lie in $\CC[[x_j]] + I_{0,2}$, so $s'$ is of the form
    \[
    (0,\dots, 0, h_1',\dots, h_{2m}').
    \]
    The constant sections $\frac{\partial}{\partial y_i}(s')$ of $E'$ are isotropic and mutually orthogonal since $s'$ is, and they are linearly independent by the assumption that $\phi$ is injective.
    As the rank of $E'$ is $2m$, they span a maximal isotropic subbundle of $E'$.
    
    We may therefore choose a new trivialization of $E'$ where the first $m$ component vectors are the vectors $\frac{\partial}{\partial y_i}(s')$.
	It follows that $s$ has the form
	\[
	(g_1,\dots, g_{2n_0}, y_1 + h_1'', \dots, y_{m} + h_{m}'', h_{m+1}'', \dots, h_{2m}''),
	\]
	where each $h_i'' \in I_{1,1} + I_{0,2}$.
    We make the variable change setting $y_i$ equal to $y_i + h_i''$, and so $s$ has the form
	\[
	(g_1',\dots, g_{2n_0}', y_1, \dots, y_{m}, h_{m+1}''', \dots, h_{2m}'''),
	\]
    where $g_i' \in \CC[[x_j]] + I_{0,2}$ and $h_i''' \in I_{1,1} + I_{0,2}$.
    
	Let us write the local frame for $E$ as 
    \[
    (v_1,\dots, v_{n_0}, v'_1, \dots, v'_{n_0},w_1,\dots, w_{m}, w_1',\dots, w_{m}').
    \]
	We now iteratively change these trivial sections to obtain each of the conditions claimed.

    By assumption, $q(s,v_1) = g'_{n_0+1}$ lies in $\C[[x_j]] + I_{0,2}$, so we can write
	\[
	q(s,v_1) = F(x) + \sum_{i=1}^{m}y_iF_i(x,y),
	\]
	where $F_i(x,y) \in I_{0,1}$ and $F(x)$ has no $y_i$-dependence.
	Now replace $v_1$ by $v_1 - \sum_{i=1}^{m} F_iw'_i$ and replace each $w_i$ by $w_i + F_iv'_1$.
    We then have that $q(s, v_1)$ has no $y_i$-dependence.
	Iterating this over $v_2,\dots, v_{n_0}$, and next similarly over $v'_1, \dots, v'_{n_0}$, we end up with no $y_i$-dependence in $q(s,v_i)$ and $q(s,v'_i)$.
    In other words, with respect to the new trivialization of $E$, the section $s$ has the form
    \[
    (g_1'',\dots, g_{2n_0}'', y_1,\dots, y_{m},h_{m+1}'''',\dots, h_{2m}''''),
    \]
    where the $g_i''$ have no $y_i$-dependence and the $h_i''''$ lie in $I_{1,1} + I_{0,2}$.
	Note that as the frame is unchanged modulo the $y_i$, it is still true that $E_0$ is identified with the subbundle spanned by the first $2n_0$ factors of $E$ over $A_0$.

    Now for $i \ge 2$, write $h_{m + i}'''' = y_1G_i(x,y) + F_i(x,y)$, where $F_i$ has no $y_1$-dependence.
    Replacing $w_1$ by $w_1 + \sum_{i=2}^{m} G_iw'_i$, and for all $i \ge 2$, replacing $w_i$ by $w_i - G_iw'_1$, we get instead that for all $i \ge 2$
    \[
    q(s,w_i) = F_i(x,y),
    \]
    so has no $y_1$-dependence.
    Since $s$ is isotropic, we then have
    \[
    0 = q(s,s) = \sum_{i=1}^{n_0} g_i''g_{i + n_0}'' + \sum_{i=1}^{m}y_iq(s,w_i).
    \]
    Since the only term of this sum which has $y_1$-dependence is $y_1q(s,w_1)$, we must have $q(s,w_1) = 0$.

    Iterating this argument to redefine $w_i$ for all $i \ge 2$, we find a normalization such that $q(s,w_i) = 0$ for all $i$, which is what we wanted.
\end{proof}

Let $S' = S_0^{\vee} \otimes_{C(E_0)} S|_{A_0}$, the spin module of $C(E')$ from Lemma \ref{thm:subtractingSpinModules}.
Recall that from \eqref{eqn:inclusionOfSquareRoot} we have an inclusion of $\cO_{Z(s_0)}$-modules $(\det N)^{1/2}_{S'}|_{Z(s_0)} \to S'|_{Z(s_0)}$.
\begin{lemma}
	\label{thm:comparisonHomomorphismExists}
	Suppose $N$ is positive. Under Assumption \ref{assume:LocalForm}, the homomorphisms of $\ZZ/2$-graded sheaves
	\[
	f^*(\cH(S,s)) \to \cH(f^*(S,s)) \cong \cH((S_0,s_0) \otimes (S',0)) \cong \cH(S_0,s_0) \otimes (S',0)|_{Z(s_0)}
	\]
	and
	\[
	\cH(S_0,s_0) \otimes (\det N)^{1/2}_{S'} \to \cH(S_0,s_0) \otimes (S',0)|_{Z(s_0)}
	\]
    are both injective and have the same image. In particular, we have an induced isomorphism of $\Z/2$-graded sheaves
	\[
	f^*(\cH(S,s)) \stackrel{\cong}{\to} \cH(S_0,s_0) \otimes (\det N )_{S'}^{1/2}.
	\]
\end{lemma}
\begin{proof}
	The claim can be verified formally locally at a point $p \in Z(s_0)$, so we may assume that we have made a choice of coordinates $x_i, y_i$ and trivialization $E \cong \cO_A^{2n}$ satisfying Assumption \ref{assume:LocalForm}, as well as the stronger assumption on $s$ granted by Lemma \ref{thm:strongerLocalForm}.
	We let $\overline E_0$ and $\overline E'$ be the subbundles of $\cO_A^{2n}$ spanned by the first $2n_0$ and last $2n - 2n_0$ unit vectors respectively, and so $E = \overline E_0 \oplus \overline E'$, while $\overline E_0|_{A_0} = E_0$ and $\overline E'|_{A_0} = E'$.
	
	Let $\Lambda_0 \subseteq \overline E_0$ and $\Lambda' \subseteq \overline E'$ be the maximal isotropic subbundles spanned by the first $n_0$ and first $n - n_0$ unit vectors, respectively, and let $\Lambda = \Lambda_0 \oplus \Lambda' \subseteq E$. Accordingly, we can split the section $s = (\overline{s}_0, \overline{s}')$.
    Note that $\Lambda'|_{A_0} = N$ as subbundles of $E'$.
	
	We may choose isomorphisms $S \cong C(E)\Lambda$ and $S_0 \cong C(E_0) \Lambda_0|_{A_0}$ (see Remark \ref{rmk:imageOfLambdaSquareRoot}).
	We also have an isomorphism
	\[
	S \cong C(E)\Lambda \cong C(\overline E_0)\Lambda_0 \otimes C(\overline E')\Lambda',
	\]
	which moreover agrees with differentials, so we have an isomorphism of complexes
	\[
	(S,s) \cong (C(\overline E_0)\Lambda_0, \overline s_0) \otimes (C(\overline E')\Lambda', \overline s').
	\]

	By the assumption on the $g_i$ in \eqref{eqn:specialFormOfS}, the complex $(C(\overline E_0)\Lambda_0, \overline s_0)$ is a free $\cO_A$-module with differential only depending on the $x_i$ and $(C(\overline E')\Lambda', \overline s')$ is a free $\cO_A$-module with differential only depending on the $y_i$.
	
	The first homomorphism in the lemma factors through a sequence of isomorphisms (we suppress differentials from the notation)
	\begin{align*}
	&f^*\cH(S) \cong f^*(\cH(\overline S_0 \otimes \overline S')) \overset{(1)}{\cong}f^*(\cH(\overline S_0) \otimes \cH(\overline S')) \cong f^*(\cH(\overline S_0))\otimes f^*(\cH(\overline S')) \\ &\overset{(2)}{\cong} \cH(f^*\overline S_0) \otimes f^*(\cH(\overline S')) \cong \cH(S_0) \otimes f^*(\cH(\overline S')) \to \cH(S_0) \otimes \cH(f^*(\overline S')) \cong \cH(S_0) \otimes S'.
	\end{align*}
	Here the isomorphism (1) follows from the fact that $\overline S_0$ and $\overline S'$ have differentials depending on $x_i$ and $y_i$ respectively.
	The isomorphism (2) follows from the differential of $\overline S_0$ depending only on $x_i$, and $f \colon A_0 \to A$ being the inclusion of the vanishing locus of the $y_i$.
	
	Our claim now follows from showing that the homomorphisms $f^*(\cH(\overline S')) \to \cH(f^*(\overline S')) \cong S'$ and $(\det N)^{1/2}_{S'} \to S'$ are injective with the same image.
	Under the isomorphisms $\overline S' = C(\overline E')\Lambda'$ and $S' = C(E')\Lambda'|_{A_0}$, we find that $\overline S'$ is the Koszul complex associated with $y_1,\dots, y_{d-d_0}$, and that the first homomorphism is injective with image $\det (\Lambda')|_{A_0} \subseteq C(E')\Lambda'|_{A_0}$.
	As $N = \Lambda'|_{A_0}$, the image of the second homomorphism is the same $\det(\Lambda')|_{A_0}$, by Remark \ref{rmk:imageOfLambdaSquareRoot}. Finally, since $N$ is positive, the isomorphism obtained this way is $\Z/2$-graded.
\end{proof}
From Lemma \ref{thm:comparisonHomomorphismExists}, we get an isomorphism 
\begin{equation}
	\label{eqn:comparisonHomomorphism}
f^*(\cH(S,s)) \overset{\cong}{\to} \cH(S_0,s_0) \otimes (\det N)_{S'}^{1/2}.
\end{equation}

\subsubsection*{Associativity} 

We now assume that $A_0 \overset{f_1}{\to} A_1 \overset{f_2}{\to} A_2$ are unramified morphisms of smooth varieties, and that for $i = 0, 1, 2$, we have an oriented quadratic vector bundle $E_i$ on $A_i$ of rank $2n_i$ with an isotropic section $s_i$ and a spin module $S_i$ for $E_i$.
We further assume that for $i=0, 1$, we have an inclusion $E_i \subseteq f_{i+1}^*E_{i+1}$, such that Assumption \ref{assume:LocalForm} holds.
%the assumptions above Lemma \ref{thm:comparisonHomomorphismExists} hold.
It follows from this that we also have an inclusion $E_0 \subseteq (f_2 \circ f_1)^*E_2$ such that the assumptions hold.
For $0 \le i < j \le 2$, write $N_{ij}$ for the normal bundle of $A_i$ in $A_j$, and write 
\[
S_{j} \setminus S_i = S_i^\vee \otimes_{C(E_i)} S_j|_{A_i},
\]
for the spin module of $C(E_j|_{A_i}/E_i)$ from Lemma \ref{thm:subtractingSpinModules}.
Suppose that, over $Z(s_i)$, the maximal isotropic subbundles $N_{ij}$ of $E_{j}|_{A_i}/E_i$ are positive. By \eqref{eqn:inclusionOfSquareRoot} we have an inclusion
\[
\det(N_{ij})_{S_{j} \setminus S_i}^{1/2} |_{Z(s_i)} \to S_{j} \setminus S_i |_{Z(s_i)}.
\]
After restricting all modules to $A_0$, we have an isomorphism
\begin{align*}
(S_1 \setminus S_0) \otimes (S_2 \setminus S_1) &= S_0^\vee \otimes_{C(E_0)} S_1 \otimes S_1^\vee \otimes_{C(E_1)} S_2 \cong S_0^\vee \otimes_{C(E_0)} C(E_1) \otimes_{C(E_1)} S_2 \\ 
&\cong S_0^\vee \otimes_{C(E_0)} S_2 = S_2 \setminus S_0,
\end{align*}
and this gives an isomorphism on $A_0$ of 
\begin{equation} \label{eqn:eqn:sqrtdetdirectsumN}
(\det N_{01})_{S_1 \setminus S_0}^{1/2} \otimes (\det N_{12})_{S_2 \setminus S_1}^{1/2}|_{Z(s_0)} \cong (\det N_{02})_{S_2 \setminus S_0}^{1/2}|_{Z(s_0)}
\end{equation}
by \eqref{eqn:sqrtdetdirectsum}.
\begin{lemma} \label{lem:associativity}
	The isomorphisms
	\[
	f_1^*(\cH(S_1,s_1)) \to \cH(S_0,s_0) \otimes (\det N_{01})^{1/2}_{S_1\setminus S_0}
	\]
	and
	\[
	f_2^*(\cH(S_2,s_2)) \to \cH(S_1,s_1) \otimes (\det N_{12})^{1/2}_{S_2 \setminus S_1}
	\]
	induce an isomorphism
	\[
	\cH(S_2,s_2)|_{A_0} \cong \cH(S_0,s_0) \otimes (\det N_{01})^{1/2}_{S_1 \setminus S_0} \otimes (\det N_{12})^{1/2}_{S_2 \setminus S_1}|_{A_0}.
	\]
	Under the isomorphism \eqref{eqn:eqn:sqrtdetdirectsumN}, this agrees with the isomorphism
	\[
	\cH(S_2,s_2)|_{A_0} \to \cH(S_0,s_0) \otimes (\det N_{02})^{1/2}_{S_2\setminus S_0}.
	\]
\end{lemma}
\begin{proof}
	Working on $Z(s_0) \subset A_0$, we have a commutative diagram
	\[
	\begin{tikzcd}
		S_0 \otimes (S_2 \setminus S_0) \arrow[rrr, bend left] & \arrow[l] S_2 \arrow[r] & \arrow[r] S_1 \otimes (S_2 \setminus S_1) & S_0 \otimes (S_1 \setminus S_0) \otimes (S_2 \setminus S_1) \\
		S_0 \otimes (\det N_{02})^{1/2}_{S_2 \setminus S_0} \arrow[u] \arrow[rrrd] & & \arrow[u] S_1 \otimes (\det N_{12})^{1/2}_{S_2 \setminus S_1} \arrow[r] & \arrow[u] S_0 \otimes (S_1 \setminus S_0) \otimes (\det N_{12})^{1/2}_{S_2 \setminus S_1} \\
		& & & \arrow[u] S_0 \otimes (\det N_{01})^{1/2}_{S_1 \setminus S_0} \otimes (\det N_{12})^{1/2}_{S_2 \setminus S_1}
	\end{tikzcd}
	\]
	Up to taking cohomology sheaves and inserting pullbacks, the two maps to compare are obtained starting from $S_2$ and going ``left, down'' and ``right, down, right, down'', respectively.
	Commutativity of the diagram gives the claim.
\end{proof}

\subsubsection*{Commutativity}

Assume now that we have an unramified morphism of smooth varieties $f \colon A_0 \to A$, with oriented quadratic vector bundles $E_0$ and $E$, and isotropic sections as at the start of this section.
Suppose that, on $Z(s_0)$, the normal bundle $N$ of $A_0$ in $A$ is positive.
Assume further that we have oriented quadratic vector bundles $F_1, F_2$ on $A_0$ equipped with corresponding spin modules $S_1, S_2$, such that $E_0 = F_1 \oplus F_2$.
Accordingly, we write $s_0 = (s_1,s_2)$ for the isotropic section. Then $S_1 \otimes S_2$ and $S_2 \otimes S_1$ are both spin modules for $C(E_0)$ with corresponding differentials $s_1 \otimes 1 + 1 \otimes s_2$ and $s_2 \otimes 1 + 1 \otimes s_1$, and by Lemma \ref{thm:comparisonHomomorphismExists}, we get isomorphisms
\[
\cH(S,s)|_{A_0} \cong \cH((S_1,s_1) \otimes (S_2,s_2)) \otimes (\det N)^{1/2}_{S \setminus S_1 \otimes S_2}
\]
and
\[
\cH(S,s)|_{A_0} \cong \cH((S_2,s_2) \otimes (S_1,s_1)) \otimes (\det N)^{1/2}_{S \setminus S_2 \otimes S_1}.
\]
The isomorphism
\[
S_1 \otimes S_2 \cong S_2 \otimes S_1
\]
given by $a \otimes b \mapsto (-1)^{|a||b|} b \otimes a$ of $C(E_0)$-modules induces an isomorphism $S \setminus (S_1 \otimes S_2) \cong S \setminus (S_2 \otimes S_1)$ of $C(E|_{A_0}/E_0)$-modules, and so an isomorphism
\[
(\det N)^{1/2}_{S \setminus S_1 \otimes S_2}|_{Z(s_0)} \cong (\det N)^{1/2}_{S \setminus S_2 \otimes S_1}|_{Z(s_0)}.
\]
\begin{lemma} \label{lem:commutativity}
The diagram of isomorphisms
\[
\begin{tikzcd}
\cH(S,s)|_{A_0} \arrow[r] \arrow[dr] & \cH((S_1,s_1) \otimes (S_2,s_2)) \otimes (\det N)^{1/2}_{S \setminus S_1 \otimes S_2} \arrow[d] \\
& \cH((S_2,s_2) \otimes (S_1,s_1)) \otimes (\det N)^{1/2}_{S \setminus S_2 \otimes S_1}
\end{tikzcd}
\]
commutes.
\end{lemma}
\begin{proof}
	Let $E' = E|_{A_0}/E_0$.
	Working over $Z(s_0)$, the claim follows from the commutativity of the diagram
	\[
	\begin{tikzcd}
		S|_{A_0} \arrow[d] \arrow[dr] & \\
		(S_1 \otimes S_2) \otimes (S \setminus (S_1 \otimes S_2)) \arrow[r] & (S_2 \otimes S_1) \otimes (S \setminus (S_2 \otimes S_1)) \\
		(S_1 \otimes S_2) \otimes C(E') \otimes_{C(E')} S \setminus (S_1 \otimes S_2) \arrow[r] \arrow[u] & (S_2 \otimes S_1) \otimes C(E') \otimes_{C(E')} S \setminus (S_2 \otimes S_1) \arrow[u] \\
				(S_1 \otimes S_2) \otimes (\det N)^{1/2}_{S \setminus (S_1 \otimes S_2)}\arrow[r] \arrow[u] & (S_2 \otimes S_1) \otimes (\det N)^{1/2}_{S \setminus (S_2 \otimes S_1)}  \arrow[u]
	\end{tikzcd}
	\]
\end{proof}

\subsection{Torus localization} \label{sec:localfixisolated}

Let $M$ be a quasi-projective scheme with 3-term symmetric obstruction theory $\phi \colon E^\mdot \to \tau^{\geq -1} \LL_M$ coming from a $(-2)$-shifted derived scheme (so that the intrinsic normal cone is isotropic \cite{OT}). Then, by \cite[Prop.~4.1]{OT}, $E^\mdot$ is quasi-isomorphic to a self-dual complex 
\begin{equation} \label{eqn:sdres}
[F \stackrel{a}{\to} E \cong E^\vee \stackrel{a^\vee}{\to} F^\vee]
\end{equation}
in degrees $-2,-1,0$ where $F$ is a vector bundle and $(E,q)$ is a quadratic vector bundle, which we assume to be of even rank. Suppose $M$ has an action by an algebraic torus $T$, $\phi$ is $T$-equivariant, and $E^\mdot$ is $T$-equivariantly symmetric. 
Then the above self-dual complex can be chosen $T$-equivariantly \cite[Sect.~7]{OT}. On the fixed locus $\iota \colon M^T \subset M$, which we assume to be compact, we can decompose $E^\mdot|_{M^T}$ into a fixed and moving part 
$$
E^\mdot|_{M^T} = E^\mdot|_{M^T}^f \oplus E^\mdot|_{M^T}^m.
$$
Then $N^{\vir} := (E^\mdot|_{M^T}^m)^\vee$ is the virtual normal bundle. The complex $E^\mdot|_{M^T}^f$ maps to $(\tau^{\geq -1} \mathbb{L}_{M})
|_{M^T}^f$ and we can further compose with the natural map $(\tau^{\geq -1} \mathbb{L}_{M})|_{M^T}^f \to \tau^{\geq -1} \mathbb{L}_{M^T}$.
This gives a 3-term symmetric obstruction theory on $M^T$. Suppose we have orientations $o \colon \O_M \cong \det(E^\mdot)$, $o^f \colon \O_{M^T} \cong \det(E^\mdot|_{M^T}^f)$, and $o^m \colon \O_{M^T} \cong \det(E^\mdot|_{M^T}^m)$. Then there are classes
\begin{align*}
[M]^{\vir} &\in A_*^T(M, \Z[\tfrac{1}{2}]), \quad \widehat{\O}_{M}^{\vir} \in K_0^T(M)_{\loc}, \\
[M^T]^{\vir} &\in A_*(M^T, \Z[\tfrac{1}{2}]), \quad \widehat{\O}_{M^T}^{\vir} \in K_0(M^T,\Z[\tfrac{1}{2}]).
\end{align*}
Supposing that the orientations are compatible, i.e., $o|_{M^T} = o^f \otimes o^m$, the Oh--Thomas localization formulae are \cite[Sect.~7]{OT}
\begin{equation} \label{eqn:OTlocgeneral}
[M]^{\vir} = \iota_* \Bigg( \frac{1}{\sqrt{e}(N^{\vir})} \cap [M^T]^{\vir} \Bigg), \quad \widehat{\O}_{M}^{\vir} = \iota_* \Bigg(\frac{1}{\sqrt{\mathfrak{e}}(N^{\vir})} \cdot \widehat{\O}_{M^T}^{\vir} \Bigg),
\end{equation}
where the square root Euler class $\sqrt{e}(-)$ and its $K$-theoretic analog $\sqrt{\mathfrak{e}}(-)$ were defined in \cite{OT}. Suppose we have a $T$-equivariant maximal isotropic subbundle $\Lambda \subset E$ of the self-dual resolution \eqref{eqn:sdres}. Then there are induced maximal isotropic subbundles
$$
\Lambda|_{M^T}^m \subset E|_{M^T}^m, \quad \Lambda|_{M^T}^f \subset E|_{M^T}^f,
$$
which induce (compatible) orientations on the complexes $E^\mdot|_{M^T}^f, N^{\vir}$. We are interested in making \eqref{eqn:OTlocgeneral} explicit under the following assumption: \\

\noindent \textbf{Assumption.} The fixed locus $M^T$ is 0-dimensional and reduced. \\

With this assumption, on each fixed point $P \in M^T$ we get a 3-term symmetric obstruction theory $E^\bullet|_P^f$.
The invariants associated with such are described by the following lemma, which follows easily from the definitions in \cite{OT}.
\begin{lemma}
Let $P = \Spec \CC$ be endowed with a 3-term symmetric obstruction theory $E^\mdot$ with self-dual resolution $[F \stackrel{a}{\to} E \cong E^\vee \stackrel{a^\vee}{\to} F^\vee]$ with $\rk(E)$ even. Then $a$ is injective and $\vd P \leq 0$. 

We have $\vd P = 0$ if and only if $E^\mdot$ is acyclic if and only if $a(F) \subset E$ is maximal isotropic. 
If this holds, then with the orientation induced by $a(F)$, we have $[P]^{\vir} = [P]$ and $\widehat{\O}_{P}^{\vir} = \O_P$.

If $\vd P < 0$, $[P]^{\vir} = 0$ and $\widehat{\O}^{\vir}_P = 0$.
\end{lemma}

Denoting the inclusion of $P$ in $M^T$ by $\iota_P$, from the lemma we obtain
$$
[M]^{\vir} = \sum_{P \in M^T \atop \vd P = 0} (-1)^{n_P^\Lambda} \iota_{P*} \Bigg(\frac{1}{e((F-\Lambda)|_P)} \Bigg), \quad \widehat{\O}_{M}^{\vir} = \sum_{P \in M^T \atop \vd P = 0} (-1)^{n_P^\Lambda} \iota_{P*} \Bigg(\frac{1}{ \widehat{\Lambda}^{\mdot} (F^\vee-\Lambda^\vee)|_P} \Bigg),
$$
where
\begin{equation} \label{defn}
(-1)^{n^\Lambda_P} = \left\{ \begin{array}{cc} 1 & \textrm{if } a(F|_P^f) \textrm{ is positively oriented with respect to } \Lambda|_P^f \\ -1 & \textrm{otherwise.} \end{array}\right.
\end{equation}

We want to make this sign more explicit. We will repeatedly use the following fact \cite[Sect.~6.1, p.~735]{GH}: for any even-dimensional quadratic vector space $(E,q)$ and maximal isotropic subspaces $\Lambda, \Lambda' \subset E$, we have that $\Lambda, \Lambda'$ induce the same orientation if and only if 
\begin{equation} \label{eqn:GHfact}
\dim(\Lambda \cap \Lambda') \equiv \tfrac{1}{2} \dim(E) \pmod 2.
\end{equation}

We start with the following lemma.
\begin{lemma} \label{lem:nnprime}
Let $P \in M^T$ have $\vd P= 0$, and let $E|_P = \Lambda|_P \oplus  \Lambda^\vee|_P$ be a splitting into maximal isotropic subspaces with projection maps $p_\Lambda$ and $p_{\Lambda^\vee}$.
We then have
 $$
\dim \coker (p_{\Lambda} \circ a|_P)^f \equiv  \dim \coker (p_{\Lambda^\vee} \circ a|_P)^f + \dim(\Lambda^\vee|_P^f) \pmod 2.
 $$
\end{lemma}
 \begin{proof}
 Let $l = \dim(\Lambda|_P^f) = \dim (F|_P^f)$.
 It is easy to see that $\dim \coker (p_{\Lambda^\vee} \circ a|_P)^f = \dim (a(F|_P^f) \cap \Lambda|_P^f)$.
 If $l$ is even, then $\Lambda|_P^f$ and $\Lambda^\vee|_P^f$ are deformation equivalent as maximal isotropic subspaces of $E|_P^f$, and so by \eqref{eqn:GHfact}
 %\cite[Sect.~6.1, p.~735]{GH} 
 \[
 \dim(\Lambda|_P^f \cap a(F|_P^f)) \equiv \dim(\Lambda^\vee|_P^f \cap a(F|_P^f)) \pmod 2.
 \]
 If $l$ is odd, then $\Lambda|_P^f$ and $\Lambda^\vee|_P^f$ are not deformation equivalent as maximal isotropic subspaces of $E|_P^f$, and so
 \[
\dim(\Lambda|_P^f \cap a(F|_P^f)) \not\equiv \dim(\Lambda^\vee|_P^f \cap a(F|_P^f)) \pmod 2. \qedhere
 \]
 \end{proof}
 
We obtain the following formula for $n_P^\Lambda$. 
\begin{proposition} \label{prop:nP}
For any isolated reduced $P \in M^T$ with $\vd P = 0$, we have
$$
n_P^\Lambda \equiv \dim  \mathrm{coker}(p_{\Lambda} \circ a|_P)^f \pmod 2,
$$
where $p_\Lambda \colon E|_P \to \Lambda|_P$ is the projection for any splitting $E|_P = \Lambda|_P \oplus  \Lambda^\vee|_P$.
\end{proposition}
\begin{proof}
We have that $\Lambda|_P^f$, $a(F|_P^f)$ determine the same orientation if and only if
$$
\dim(a(F|_P^f ) \cap \Lambda|_P^f) \equiv \tfrac{1}{2} \dim(E|_P^f) = \dim(\Lambda^\vee|_P^f) \pmod 2.
$$
In the proof of Lemma \ref{lem:nnprime}, we saw that 
\begin{align*}
\dim \coker (p_{\Lambda^\vee} \circ a|_P)^f &= \dim (a(F|_P^f) \cap \Lambda|_P^f) \\
\dim \coker (p_{\Lambda} \circ a|_P)^f &\equiv  \dim \coker (p_{\Lambda^\vee} \circ a|_P)^f + \dim(\Lambda^\vee|_P^f) \pmod 2.
\end{align*}
From \eqref{defn}, we deduce
\begin{align*}
n_P^{\Lambda} &\equiv \dim(a(F|_P^f ) \cap \Lambda|_P^f) - \dim(\Lambda^\vee|_P^f) \\
&=\dim \coker (p_{\Lambda^\vee} \circ a|_P)^f - \dim(\Lambda^\vee|_P^f)  \\
&\equiv \dim \coker (p_{\Lambda} \circ a|_P)^f  \pmod 2. \qedhere 
\end{align*}
\end{proof}

\section{Global} \label{sec:global}

In this section we describe $\Quot_r^n(\C^4)$, the Quot scheme parametrizing 0-dimensional quotients of length $n$ of $\O_{\C^4}^{\oplus r}$, as the vanishing locus of an isotropic section of a quadratic vector bundle over a smooth variety.
This leads to a definition of the twisted virtual structure sheaf on the Quot scheme, to the $K$-theory class
\[
\cN^\glob_{r,n} \in K_0^{\mathbb{T}'}(\Quot_r^n(\C^4)),
\]
for some algebraic torus $\TT'$ acting on $\Quot_r^n(\C^4)$, and to 
\[
N^{\glob}_{r,n} = \chi(\Quot_r^n(\C^4), \cN^\glob_{r,n}).
\]
Using spin modules, we show that the classes $\cN_{r,n}^{\glob}$ lift to sheaves and give rise to a factorizable sequence.

\subsection{Isotropic zero locus} \label{sec:zeroloc} 

 Let $n \in \Z_{\geq 0}$ and let $V$ be an $n$-dimensional complex vector space. Let $r \in \Z_{>0}$ and consider the following quiver:
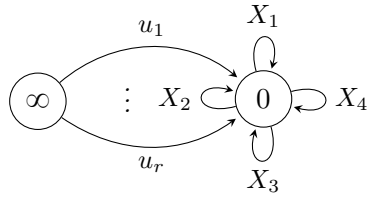
\begin{figure}[ht]
\begin{tikzpicture}[>=stealth,->,shorten >=2pt,looseness=.5,auto]
  \matrix [matrix of math nodes,
           column sep={3cm,between origins},
           row sep={3cm,between origins},
           nodes={circle, draw, minimum size=7.5mm}]
{ 
|(A)| \infty & |(B)| 0 \\         
};
\tikzstyle{every node}=[font=\small\itshape]
\path[->] (B) edge [loop above] node {$X_1$} ()
              edge [loop left] node {$X_2$} ()
              edge [loop right] node {$X_4$} ()
              edge [loop below] node {$X_3$} ();

\node [anchor=west,right] at (-0.5,0.1) {$\vdots$};
\node [anchor=west,right] at (-0.3,0.95) {$u_1$};              
\node [anchor=west,right] at (-0.3,-0.85) {$u_r$};              
%\draw (A) to [bend left=25,looseness=1] (B) node [midway,above] {};
\draw (A) to [bend left=40,looseness=1] (B) node [midway] {};
\draw (A) to [bend right=35,looseness=1] (B) node [midway,below] {};
\end{tikzpicture}
%\caption{}
\end{figure}

We consider representations of this quiver with dimension vector $(1,n)$, i.e.~we put the vector space $\C$ at the node $\infty$ and $V$ at the node $0$. 
Representations with this dimension vector correspond to elements of 
\[
W = \C^4 \otimes \End(V) \oplus \Hom(\C^r,V).
\]
The group $\GL(V)$ acts on $W$ by
\begin{align*}
&g \cdot (X_1,\ldots,X_4, u_1, \ldots, u_r) = (g X_1 g^{-1},\ldots,g X_4 g^{-1},g u_1, \ldots, g u_r),\\ 
&X_i \in \End(V), \quad u_i \in V.
\end{align*}
Let $U \subset W$ be the open subset consisting of representations satisfying
$$
\C\langle X_1,\ldots,X_4\rangle \cdot \langle u_1, \ldots , u_r \rangle = V,
$$ 
i.e.~acting with polynomial expressions on the span of the vectors $u_i$ generates $V$.
Then $\GL(V)$ acts freely on $U$ and the \emph{non-commutative Quot scheme} is defined as the open substack
$$
\ncQuot_{r}^{n}(\C^4) = [U / \GL(V)] \subset [W / \GL(V)].
$$
In fact, $\ncQuot_{r}^{n}(\C^4)$ is a smooth variety of dimension $3n^2+rn$ (see also \cite{BR}).

On $U$, we have the trivial vector bundle
$$
E = \Lambda^2 \C^4 \otimes \End(V) \otimes \O_U
$$
with quadratic form $q$ determined (after fixing a non-zero $\omega \in \Lambda^4 \C^4$) by
\begin{align*}
&(\Lambda^2 \C^4 \otimes \End(V)) \otimes (\Lambda^2 \C^4 \otimes \End(V)) \rightarrow \Lambda^4 \C^4  \overset{\omega^\vee(-)}{\cong} \C \\
&\quad\quad (\omega_1 \otimes f_1) \otimes (\omega_2 \otimes f_2) \mapsto \omega_1 \wedge \omega_2 \cdot \tr(f_1 \circ f_2).
\end{align*}
We let $(e_1,e_2,e_3,e_4)$ be the standard basis of $\C^4$.
Then $E$ has a section
$$
s \in H^0(U,E), \quad \Big(\sum_{i=1}^{4} e_i \otimes f_i , u_1, \ldots, u_r \Big) \mapsto \sum_{i,j=1}^{4} (e_i \wedge e_j) \otimes (f_i \circ f_j),\ \ f_i \in \End(V).
$$

Letting $\C^3=\langle e_1,e_2,e_3 \rangle$, we define the subbundle $\Lambda \subset E$ by
\begin{equation} \label{defLambda}
\Lambda = (\langle e_4 \rangle \wedge \C^3) \otimes \End(V) \otimes \O_U.
\end{equation}
Since $E, \Lambda, s$ are naturally $\GL(V)$-equivariant and $q$ is $\GL(V)$-invariant, they descend to $\ncQuot_r^n(\C^4)$, where we use the same notation.
\begin{lemma}
The section $s$ is isotropic and $Z(s) \cong \Quot_r^n(\C^4)$. Furthermore $\Lambda \subset E$ is a maximal isotropic subbundle.
\end{lemma}
\begin{proof}
The section $s$ vanishes at a point $(\sum_i e_i \otimes f_i , u_1, \ldots, u_r)$ if and only if
$$
[f_i,f_j] = f_i \circ f_j - f_j \circ f_i = 0, \quad \forall i,j=1,\ldots, 4.
$$
This induces a (scheme theoretic) isomorphism $Z(s) \cong \Quot_r^n(\C^4)$. Moreover
\begin{align*}
&q\Big( \sum_{i,j=1}^{4} (e_i \wedge e_j) \otimes (f_i \circ f_j), \sum_{i,j=1}^{4} (e_i \wedge e_j) \otimes (f_i \circ f_j)\Big) \\
&\quad \quad  \quad \quad = e_1 \wedge e_2 \wedge e_3 \wedge e_4 \otimes \tr(\sum_{\sigma \in S_4} (-1)^{|\sigma|} f_{\sigma(1)} \circ f_{\sigma(2)} \circ f_{\sigma(3)} \circ f_{\sigma(4)}) = 0.
\end{align*}
Any element of the form $\lambda = \sum_{i=1}^{3} (e_4 \wedge e_i) \otimes f_{i} \in  (\langle e_4 \rangle \wedge \C^3) \otimes \End(V)$ satisfies $q(\lambda,\lambda) = 0$, so $\Lambda$ is isotropic.
\end{proof}

\subsection{Definition of invariants} \label{sec:definv}

We introduce three tori acting on $\ncQuot^n_r(\C^4)$, denoted $T_t, T_w$ and $T_y$.
The torus $(\C^*)^4$ acts naturally on $\C^4$, and we let $T_t$ be the subtorus preserving the volume form on $\C^4$.
We denote its coordinates by $t_1,t_2,t_3,t_4$, these are subject to the relation $t_1t_2t_3t_4 = 1$.

The tori $T_w$ and $T_y$ are $r$-dimensional tori with coordinates $w_1, \ldots, w_r$ and $y_1, \ldots, y_r$, respectively.
The parameters $w_i$ are referred to as \textit{framing parameters}.
The parameters $y_i$ are referred to as \textit{mass parameters}.
There is an action of $\mathbb{T} = T_t \times T_w \times T_y$ on $A = \ncQuot_{r}^{n}(\C^4)$ given by
\begin{align*}
&(t_1, \ldots, t_4, w_1, \ldots, w_r, y_1, \ldots, y_r) \cdot (X_1, \ldots, X_4, u_1, \ldots, u_r) \\
 &= (t_1^{-1} X_1, \ldots, t_4^{-1} X_4, w_1^{-1} u_1, \ldots, w_r^{-1} u_r).
\end{align*}

The $\GL(V)$-equivariant vector bundle $V \otimes \O_U$ descends to a tautological rank $n$ vector bundle $\cV$ on $\ncQuot_r^n(\C^4)$. Note that
$$
\cV|_{\Quot_r^n(\C^4)} \cong \pi_* \cQ
$$
where $\cQ$ denotes the universal quotient sheaf on $\Quot_r^n(\C^4) \times \C^4$ and $\pi$ denotes projection onto the first factor. We give the quadratic vector bundle $E = \Lambda^2\C^4 \otimes \cEnd(\cV)$ the $\mathbb T$-equivariant structure coming from the $\mathbb T$-action on $\C^4$. Then $\Lambda$ and the quadratic form $q$ are $\mathbb{T}$-invariant and $s$ is $\mathbb T$-equivariant.\footnote{This gives one reason for why we restrict to the subtorus $T_t$ of $(\C^*)^4$: there is no way to lift the $T_t$-equivariant structure on $E$ to a $(\C^*)^4$-equivariant one in such a way that $s$ is $(\C^*)^4$-equivariant and $q$ is $(\CC^*)^4$-invariant.} 

As we discuss in Section \ref{sec:local}, the natural $\mathbb{T}$-action on $\Quot_r^n(\C^4)$ has isolated reduced fixed points, and we are therefore in the setting of the ``standard model'' of Oh--Thomas discussed in the introduction. 
Thus we obtain a $\mathbb{T}$-equivariant virtual class and virtual structure sheaf for $M=\Quot_r^n(\C^4)$
\begin{align*}
[M]^{\vir} \in A_{rn}^{\mathbb{T}}(M,\Z[\tfrac{1}{2}]), \quad \widehat{\O}^{\vir}_M &\in K_0^{\T}(M)_{\loc}. 
\end{align*}

We consider also the $\mathbb T$-equivariant vector bundle
\[
\cW = \bigoplus_{i=1}^{r} \O_A \otimes w_i \otimes y_i,
\]
where $\O_A$ has its natural $\TT$-equivariant structure which, in the $i$th summand, we tensorize by the weight one characters $w_i$, $y_i$ of $T_w, T_y$. 

\begin{definition} \label{defNgenus}
We define $K$-theory classes
\[
\cN^\glob_{r,n} = \widehat{\O}^{\vir} \otimes \widehat{\Lambda}^\mdot \hom(\cV,\cW) \in K_0^{\mathbb T}(\Quot^n_r(\C^4))_{\loc}
\]
and invariants 
$$
N_{r,n}^\glob = \chi(\Quot_r^n(\C^4), \cN^\glob_{r,n}) \in \Q(t_1^{\frac{1}{2}}, \ldots , t_4^{\frac12}, w_1^{\frac{1}{2}}, \ldots, w_r^{\frac{1}{2}}, y_1^{\frac{1}{2}}, \ldots, y_r^{\frac{1}{2}}), \ t_4 = t_1^{-1}t_2^{-1}t_3^{-1}.
$$
Here $\widehat{\Lambda}^\mdot$ was defined in \eqref{eqn:lambdahat} and we interpret $\chi$ via the general $K$-theoretic localization formula, that is we let
\[
\chi(\Quot_r^n(\C^4), \cN^\glob_{r,n}) = \chi(\Quot_r^n(\C^4)^{\T}, \widetilde{\cN}^\glob_{r,n}),
\]
where $\widetilde{\cN}^\glob_{r,n}$ is the unique class in 
$K_0^{\T}(\Quot_r^n(\C^4)^{\T})_{\mathrm{loc}}$
which under pushforward along $\Quot_r^n(\C^4)^{\T} \into \Quot_r^n(\C^4)$ maps to $\cN^\glob_{r,n}$, see \cite{Tho}, \cite[Sect.~2.3]{Oko}.
Furthermore, we define the generating series
\begin{align}
\begin{split} \label{defG}
\mathsf{G}_{r} &= \sum_{n=0}^{\infty} N_{r,n}^\glob \, q^n,
\end{split}
\end{align}
which has leading term 1.
\end{definition}

\begin{remark} \label{def:Tprime}
Let $\mathbb{T}' = T_t \times T_w \times \widetilde{T}_y$, where $\widetilde{T}_y$ is the $2^r$-fold cover of $T_y$ such that all characters of $T_y$ have a square root defined on $\widetilde T_y$. At the end of this section, we will in fact see that $\cN^\glob_{r,n} \in K_0^{\mathbb T}(\Quot^n_r(\C^4))_{\loc}$ admits a lift to $K_0^{{\mathbb T}'}(\Quot^n_r(\C^4))$.
Consequently, $N_{r,n}^\glob \in \Q(t_1, \ldots, t_4, w_1, \ldots, w_r, y_1^{\frac{1}{2}}, \ldots, y_r^{\frac{1}{2}})$ with $t_4 = t_1^{-1}t_2^{-1}t_3^{-1}$. 
In Section \ref{sec:limits}, we will show that $N_{r,n}^\glob \in \Q(t_1, \ldots , t_4, (y_1 \cdots y_r)^{\frac{1}{2}})$ with $t_4 = t_1^{-1}t_2^{-1}t_3^{-1}$.
\end{remark}

\begin{remark} \label{dualVW}
We note the identity
$$
\widehat{\Lambda}^\mdot \hom(\cV,\cW) = (-1)^{rn} \widehat{\Lambda}^\mdot \hom(\cW,\cV). 
$$
In Section \ref{sec:local}, we prefer to work with the representation on the right hand side in order to match existing conventions for the vertex formalism in the literature. 
\end{remark}

In the ``standard model'' for 2-term perfect obstruction theories, where $M$ is the zero locus of a $T$-equivariant section $s$ of a $T$-equivariant vector bundle $E$ on a smooth connected quasi-projective variety $A$ acted on by an algebraic torus $T$, the virtual structure sheaf $\O_M^{\mathrm{vir}} \in K^T_0(M)$ lifts to a $\Z/2$-graded $T$-equivariant coherent sheaf on $M$ as follows. The Koszul complex $(\Lambda^\mdot E^\vee,s \intprod)$ gives rise to the following $K$-theory class
\begin{equation}
\label{eqn:koszul}
\sum_{i \geq 0} (-1)^i \Lambda^i E^\vee = \sum_{i \geq 0} (-1)^i H^i(\Lambda^\mdot E^\vee, s\intprod) \in K^T_0(A).
\end{equation}
The Koszul complex is a sheaf of differential graded $\O_A$-algebras and, by the Leibniz rule, we see that the cohomology sheaves $H^i(\Lambda^\mdot E^\vee, s\intprod)$ are sheaves of $\O_M$-modules. 
Therefore, the right hand side of \eqref{eqn:koszul} gives a class in $K^T_0(M)$ known as the virtual structure sheaf $\O_M^{\mathrm{vir}}$. In this ``standard model'' setting, we can view $\O_M^{\mathrm{vir}}$ as the class of the $\Z/2$-graded $T$-equivariant coherent sheaf $H^{\mathrm{odd}}(\Lambda^\mdot E^\vee)[1] \oplus H^{\mathrm{even}}(\Lambda^\mdot E^\vee)$.
Note that $\O_M^{\vir} = \O_M$ when $s$ is regular.

In the case $E = \Omega_A$ and $s = dw$, for a $T$-invariant regular function $w \colon A \to \mathbb{A}^1$, the perfect obstruction theory is symmetric, $K_M^{\mathrm{vir}}$ has the square root $(K_M^{\mathrm{vir}})^{\frac12}=K_A$, and one obtains the twisted (or \textit{symmetrized}) virtual structure sheaf \cite[3.2.8]{Oko}
$$
\widehat{\O}_M^{\mathrm{vir}} = \O_M^{\mathrm{vir}} \otimes (K_M^{\mathrm{vir}})^{\frac{1}{2}} \in K^T_0(M). 
$$

Suppose we are in the setting of the ``standard model'' of Section \ref{sec:OTtheory}, i.e.~we have a quadratic vector bundle $(E,q)$ of even rank over a smooth connected quasi-projective variety $A$, an isotropic section $s$ of $E$ such that $M=Z(s)$, and a maximal isotropic subbundle $\Lambda \subset E$. Recall that $\Lambda$ induces an orientation on $(E,q)$. Now suppose additionally that $\det(\Lambda)$ has a square root $L$ (as a line bundle) and suppose the entire setup is $T$-equivariant with respect to the action of an algebraic torus $T$ on $A$ as before. By Example \ref{ex:sqrtdetgivesspin}, we then obtain a spin module
\[
S := C(E)\Lambda \otimes L^{-1}.  
\]
Furthermore, the isotropic section $s$ gives rise to a complex $(S,s)$ (Definition \ref{def:complex}). In the case $E = \Lambda \oplus \Lambda^\vee$ splits, we can write $s = (s_{\Lambda},s_{\Lambda^\vee})$ and we have an isomorphism (of $T$-equivariant $\Z/2$-graded left $C(E)$-modules)
\begin{equation} \label{eqn:isototwoperiodic}
S \cong L \otimes \Lambda^\mdot \Lambda^\vee,
\end{equation}
where multiplication by $s$ corresponds to $s_{\Lambda^\vee} \wedge + s_{\Lambda} \intprod$ (Example \ref{ex:reltoperiodcx}). 
Note that the $\Z/2$-graded sheaf $\cH(S,s)$ is supported on $M$ (Lemma \ref{lem:suppZ(s)}) and therefore defines a class in $K_0(M)$.
As shown by Oh--Sreedhar \cite{OS}, this class agrees with the one defined via Kiem--Li cosection localization for the section and cosection $s_\Lambda$ and $s_{\Lambda^\vee}^\vee$ of $\Lambda$.
It follows from this that the Oh--Thomas virtual structure sheaf in this ``standard model'' setup can be expressed using \eqref{eqn:isototwoperiodic}, and we have:
\begin{proposition} \label{Kcompare}
In the above setup, we have
$$
\cH(S,s) \cdot \det(T_A|_M)^{-\frac{1}{2}} = \widehat{\mathcal{O}}^{\vir}_M \in K^T_0(M)_{\mathrm{loc}}.
$$
\end{proposition}

We now return to the setting of Section \ref{sec:zeroloc} with $A = \ncQuot_{r}^n(\C^4)$, and the data $(E,q)$, $s$ and $\Lambda$ as defined there with $Z(s) = \Quot_r^n(\C^4)$. Note that we have a $\TT$-equivariant splitting $E =\Lambda \oplus \Lambda^\vee$. In $K^0_{\TT}(A)$, we have the identity
\[
\Lambda = \langle e_4 \rangle \wedge \C^3 \cdot \mathcal{E}{\it{nd}}(\cV),
\]
where $\cV$ is naturally a $\TT$-equivariant vector bundle as described above and $\langle e_4 \rangle \wedge \C^3$ denotes the $\TT$-representation $t_1^{-1}t_4^{-1}+t_2^{-1}t_4^{-1}+t_3^{-1}t_4^{-1}$. Since the determinant map factors through $K$-theory, we obtain
\begin{equation} \label{eqn:detLambda}
\det(\Lambda) = t_4^{-2n^2} \otimes \O_A.
\end{equation}
Therefore, we can choose the square root $L = t_4^{-n^2} \otimes \O_A$ and we are in the setting described in Proposition \ref{Kcompare}. 

We next define a $\Z/2$-graded sheaf $\cN^{\text{sheaf}}_{r,n}$ by
\[
\cN^{\text{sheaf}}_{r,n} = \cH(S,s) \otimes \Lambda^\mdot \hom(\cV,\cW) \otimes (\det(T_A) \otimes \det(\hom(\cV,\cW)))^{-\frac12},
\]
which, by Lemma \ref{lem:suppZ(s)}, is supported on $\Quot_r^n(\C^4)$.
The quotient description of $A$ gives a short exact sequence of sheaves
\begin{equation}
\label{eqn:tangentBundleExactSequence}
0 \to \mathcal{E}{\it{nd}}(\cV) \to \C^4 \otimes \mathcal{E}{\it{nd}}(\cV) \oplus \bigoplus_{i=1}^{r} \cV \otimes w_i^{-1} \to T_A \to 0.
\end{equation}
In $K^0_{\TT}(A)$, we have the identity
\begin{equation} \label{eqn:TAinKtheory}
T_A = (\C^4 - 1) \mathcal{E}{\it{nd}}(\cV) + \sum_{i=1}^{r} \cV \otimes w_i^{-1},
\end{equation}
where $\cV$ is naturally a $\TT$-equivariant vector bundle as described above and $\C^4 =t_1^{-1}+ \cdots + t_4^{-1}$ is the $\TT$-representation generated by the tangent vectors $\frac{\partial}{\partial x_i}$. We obtain
\[
\det(T_A) = (w_1\cdots w_r)^{-n}\det(\cV)^{\otimes r}
\]
and
\[
\det(\hom(\cV,\cW)) = (y_1 w_1 \cdots y_r w_r)^{n}\det(\cV)^{-\otimes r}.
\]
The line bundle 
\begin{equation} \label{eqn:cancel}
\det(T_A) \otimes \det(\hom(\cV,\cW)) \cong (y_1 \cdots y_r)^{n} \otimes \O_A
\end{equation}
admits a $\mathbb T'$-equivariant square root $(y_1\cdots y_r)^{\frac n2}\otimes \O_A$. Thus, $\cN^{\text{sheaf}}_{r,n}$ is $\TT'$-equivariant for the torus $\TT'$ defined in Remark \ref{def:Tprime}.

By Proposition \ref{Kcompare} we then have 
\[
\cN^{\mathrm{glob}}_{r,n} = [\cN^{\mathrm{sheaf}}_{r,n}] \in K_0^{\mathbb{T}}(\ncQuot_r^n(\C^4))_{\loc},
\]
so $\cN^{\mathrm{sheaf}}_{r,n}$ is indeed a refinement of $\cN^{\mathrm{glob}}_{r,n}$ as claimed. In fact, our analysis shows 
\[
\cN^{\mathrm{glob}}_{r,n}  \in K_0^{\mathbb{T}'}(\ncQuot_r^n(\C^4)),
\]
which establishes the claim in Remark \ref{def:Tprime}.

\subsection{Factorization data} \label{sec:factorforNek}

In this section, we construct a factorizable sequence $\{\F_n\}_{n=1}^{\infty}$ for $\Sym^n(\C^4)$ (Definition \ref{def:factorisableSequence}).

Recall the description of $\Quot_1^n(\C^4) = \Hilb^n(\C^4)$ as the zero locus of an isotropic section of a quadratic vector bundle on a smooth ambient space $\mathrm{ncQuot}_1^n(\C^4) = \mathrm{ncHilb}^n(\C^4)$ from Section \ref{sec:zeroloc}. We denote the corresponding data by $A_n, E_n, s_n, \Lambda_n$.
Since the tori $T_w$ and $T_y$ are 1-dimensional, we write $w$ and $y$ instead of $w_1$ and $y_1$.

We will use the notation of Section \ref{sec:factorizability}. From the quotient description of $A_n$, it is clear that for any $m,n \geq 1$ there is a rational map
\begin{equation} \label{eqn:ratmap}
A_m \times A_n \dashrightarrow A_{m+n}.
\end{equation}

Let $O_{m,n} \subseteq A_m \times A_n$ be the open set of pairs
\[
((X_1,\dots,X_4,u), (Y_1,\dots, Y_4,v))
\]
such that there exists a $4$-tuple $(a_1,\dots,a_4) \in \C^4\setminus \{0\}$ such that the spectra of $\sum_{i=1}^4 a_iX_i$ and $\sum_{i=1}^4 a_i Y_i$ are disjoint.

\begin{lemma}
\label{thm:linearAlgebra}
Let $n_1,n_2 \ge 1$ and let $A_{11}, A_{22}$ be matrices of size $(n_1 \times n_1)$ and $(n_2 \times n_2)$, respectively.
Assume the spectra of $A_{11}$ and $A_{22}$ are disjoint, and let
\[
A = \begin{pmatrix} A_{11} & 0 \\ 0 & A_{22} \end{pmatrix}.
\]
Then
\begin{enumerate}
    \item There are complex polynomials $f_1, f_2$ such that
    \[
    f_1(A) = \begin{pmatrix} I_{n_1} & 0 \\ 0 & 0 \end{pmatrix},\ f_2(A) = \begin{pmatrix} 0 & 0 \\ 0 & I_{n_2} \end{pmatrix}.
    \]
    \item Let $B_{12}$ be an $(n_1 \times n_2)$-matrix, and let $B_{21}$ be an $(n_2 \times n_1)$-matrix, such that
    \[
    A_{11}B_{12}=B_{12}A_{22} \text{ and }A_{22}B_{21} = B_{21}A_{11}.
    \]
    Then $B_{12} = B_{21} = 0$.
\end{enumerate}
\end{lemma}
\begin{proof}
Part (1) follows from the Cayley--Hamilton theorem.
Indeed, for the first equality, letting $f(x)$ be the characteristic polynomial of $A_{22}$, we get $f(A) = \begin{pmatrix} f(A_{11}) & 0 \\ 0 & 0 \end{pmatrix}$. By the assumption on disjoint spectra, $f(A_{11})$ is invertible.
If $g(x)$ is the characteristic polynomial of $f(A_{11})$, taking $f_1 = (1 - g/g(0))\circ f$ gives the desired result, using Cayley--Hamilton for $f(A_{11})$.

For part (2), we have
\[
\begin{pmatrix} 0 & B_{12} \\
B_{21} & 0
\end{pmatrix} A
= A \begin{pmatrix} 0 & B_{12} \\
B_{21} & 0
\end{pmatrix}.
\]
The same identity holds after replacing $A$ by $f_1(A)$ and $f_2(A)$.
This gives the claim.
\end{proof}

\begin{lemma}
	The map \eqref{eqn:ratmap} is defined on $O_{m,n}$. Moreover $$\Hilb^{m}(\C^4) \times \Hilb^{n}(\C^4)|_{U_{m,n}} = (\Hilb^m(\C^4) \times \Hilb^n(\C^4)) \cap O_{m,n}.$$
\end{lemma}
\begin{proof}
Take an element $((X_1, \ldots, X_4,u), (Y_1, \ldots, Y_4,v))$ of $O_{m,n}$ and consider the direct sum
\begin{equation}
\label{eqn:blockDiagonals}
	\left( \begin{pmatrix}
		X_1 & 0 \\
		0 & Y_1
	\end{pmatrix}, \begin{pmatrix}
		X_2 & 0 \\
		0 & Y_2
	\end{pmatrix}, \begin{pmatrix}
		X_3 & 0 \\
		0 & Y_3
	\end{pmatrix}, \begin{pmatrix}
		X_4 & 0 \\
		0 & Y_4
	\end{pmatrix},\begin{pmatrix}
		u \\
		v
	\end{pmatrix} \right),	
\end{equation}
where $\C\langle X_1, \ldots, X_4 \rangle \cdot u = \C^m$ and $\C\langle Y_1, \ldots, Y_4 \rangle \cdot v = \C^n$. We must show that, acting on $(u,v)$, these block diagonal matrices generate $\C^{m+n}$. By assumption, there is an $(a_1, \ldots, a_4) \in \C^4 \setminus 0$ such that $X' := a_1 X_1 + \dots + a_4 X_4$ and $Y' := a_1 Y_1 + \dots + a_4 Y_4$ have disjoint spectra.
Letting
\[
A = \begin{pmatrix} X' & 0 \\ 0 & Y'\end{pmatrix},
\]
then we can find polynomials $f_1, f_2$ as in Lemma \ref{thm:linearAlgebra}.
We then have $(u,0) = f_1(A)(u,v)$ and $(0,v) = f_2(A)(u,v)$.
Thus $(u,0)$ and $(0,v)$ are both in the subspace generated from $(u,v)$, by \eqref{eqn:blockDiagonals}, and it follows that the whole of $\C^{m+n}$ is generated by \eqref{eqn:blockDiagonals}.

For the second part, note that if $((X_1, \ldots, X_4,u), (Y_1, \ldots, Y_4,v))$ lies in $\Hilb^{m}(\C^4) \times \Hilb^{n}(\C^4)$, then the $X_i$ (resp.~$Y_i$) commute.
We then have the joint spectra
\[
S(X) = \{(\lambda_1,\lambda_2,\lambda_3,\lambda_4) \in \C^4 \, : \, \exists v \in \C^m \setminus \{0\} \text{ such that }X_i v = \lambda_iv\}
\]
and
\[
S(Y) = \{(\lambda_1,\lambda_2,\lambda_3,\lambda_4) \in \C^4 \, : \, \exists v \in \C^n\setminus\{0\} \text{ such that }Y_i v = \lambda_iv\}.
\]
Letting $l(\lambda_1,\lambda_2,\lambda_3,\lambda_4) = \sum_{i=1}^4 a_i\lambda_i$, the spectrum of $\sum_{i=1}^4a_iX_i$ (resp.~$\sum_{i=1}^4a_iY_i)$ is $l(S(X))$ (resp.~$l(S(Y))$).
It is then easy to see that the pair $((X_i),u),((Y_i),v)$ lies in $O_{m,n}$ if and only if $S(X) \cap S(Y) = \varnothing$.

On the other hand, $S(X)$ and $S(Y)$ can be identified with the set-theoretic support of the associated subschemes of $\Hilb^m(\C^4)$ and $\Hilb^n(\C^4)$, so that $S(X) \cap S(Y) = \varnothing$ if and only if the pair lies over $U_{m,n}$.
This proves the claim.
\end{proof}

We obtain a commutative diagram
\begin{equation} \label{diag:sqrs1}
\begin{tikzcd} 
O_{m,n}  \arrow[r, " "] & A_{m+n} \\
\Hilb^{m}(\C^4) \times \Hilb^{n}(\C^4)|_{U_{m,n}} \arrow[r, " "] \arrow[u, hook] \arrow[d, "\nu_{m,n}"] & \Hilb^{m+n}(\C^4) \arrow[u, hook] \arrow[d, "\nu_{m+n}"] \\
 U_{m,n} \arrow[r, "\sigma_{m,n}"] & \Sym^{m+n}(\C^4)
\end{tikzcd}
\end{equation}
where the squares are cartesian, and the bottom map is \'etale. Abusing notation, we denote all horizontal arrows by $\sigma_{m,n}$.

\begin{lemma}
The morphism $O_{m,n} \to A_{m+n}$ is unramified.
\end{lemma}
\begin{proof}
Representing $P \in O_{m,n}$ by $(X_1^P, \ldots, X_4^P, u^P)$ and $Q$ by $(X_1^Q, \ldots, X_4^Q, u^Q)$, the image point $P+Q$ is represented by
$$
\Big( \Big(\begin{array}{cc} X_1^P & 0 \\ 0 & X_1^Q \end{array}\Big), \ldots, \Big(\begin{array}{cc} X_4^P & 0 \\ 0 & X_4^Q \end{array}\Big), \Big( \begin{array}{c} u^P \\ u^Q \end{array} \Big) \Big).
$$
Consider the induced map on tangent spaces
\begin{equation} \label{mapontangents}
T_{A_m}|_P \oplus T_{A_n}|_Q \to T_{A_{m+n}}|_{P+Q}, 
\end{equation}
and take a vector in $T_{A_m}|_P$ represented by $\xi_1 = (Y_1, \ldots, Y_4, u)$ and a vector in $T_{A_n}|_Q$ represented by $\xi_2 = (Z_1, \ldots, Z_4, v)$. Using the quotient construction of $A_{m+n}$ and \eqref{eqn:tangentBundleExactSequence}, we see that if $(\xi_1,\xi_2)$ maps to zero under \eqref{mapontangents} then there exist matrices $g_{11} \in \End(\C^m)$, $g_{12} \in \Hom(\C^n,\C^m)$, $g_{21} \in \Hom(\C^m,\C^n)$, $g_{22} \in \End(\C^n)$ such that
\begin{align*}
[g_{11}, X_i^P] &= Y_i, \quad [g_{22}, X_i^Q] = Z_i, \\ 
g_{12} X_i^Q &= X_i^P g_{12}, \quad g_{21} X_i^P = X_i^Q g_{21}, \\ 
g_{11} u^P + g_{12} u^Q &= u, \quad g_{21} u^P + g_{22} u^Q= v,
\end{align*}
for all $i=1,2,3,4$, and where $[-,-]$ denotes the commutator. By assumption, there exists $(a_1, \ldots, a_4) \in \C^4 \setminus 0$ such that $\sum_i a_i X_i^P$ and $\sum_i a_i X_i^Q$ have disjoint spectra.
By the second line of equations, we may apply Lemma \ref{thm:linearAlgebra}, with $A_{11} = \sum_i a_iX_i^P, A_{22} = \sum_i a_iX_i^Q$, $B_{12} = g_{12}$, and $B_{21} = g_{21}$. 
We conclude that $g_{12} = 0$ and $g_{21} = 0$. 
The remaining equations then imply $(\xi_1,\xi_2)$ equals  $(0,0)$ as elements of $T_{A_m}|_P \oplus T_{A_n}|_Q$.
\end{proof}

We define the quadratic vector bundle
$$
E_{m,n} := E_m \boxplus E_n \quad \mathrm{on} \quad A_m \times A_n,
$$
which has a maximal isotropic subbundle $\Lambda_{m,n} = \Lambda_m \boxplus \Lambda_n$.
There is a natural inclusion 
$$
E_{m,n} \into \sigma_{m,n}^*E_{m+n}
$$
under which the quadratic forms and sections of $(s_m, s_n)$ and $s_{m+n}$ agree. Recall that on $A_n$, we also have a rank $n$ tautological vector bundle $\cV_n$ and we define 
\[
\cV_{m,n} = \cV_m \boxplus \cV_n \quad \mathrm{on} \quad A_m \times A_n.
\]
We define the bundle
\[
F_{m,n} = \C^4 \otimes (\sHom (\cV_m, \cV_n) \oplus \sHom(\cV_n,\cV_m))
\]
on $A_{m} \times A_n$.
We have a natural identification of $\sigma^*(E_{m+n})/E_{m,n}$ with
\[
\Lambda^2 \C^4 \otimes (\sHom (\cV_m, \cV_n) \oplus \sHom(\cV_n,\cV_m)).
\]
We get a homomorphism of sheaves
\[
\phi \colon F_{m,n} \to \sigma^*(E_{m+n})/E_{m,n}
\]
defined by
\[
\sum_i e_i \otimes Y_i \mapsto \sum_{i < j} (e_i \wedge e_j) \otimes ([X_i,Y_j] - [X_j,Y_i]),
\]
where $X_i \in \cEnd(\cV_m \boxplus \cV_n)$ are the tautological endomorphisms.
\begin{lemma}
\label{thm:inclusionOfMaximaIsotropic}
    Restricting to $\Hilb^m(\C^4) \times \Hilb^n(\C^4)|_{U_{m,n}}$, the image $\phi(F_{m,n})$ is an isotropic subsheaf of $\sigma_{m,n}^*(E_{m+n})/E_{m,n}$.
    At any point $p \in \Hilb^m(\C^4) \times \Hilb^n(\C^4)|_{U_{m,n}}$, we have that $\phi(F_{m,n})|_p \subseteq \sigma^*_{m,n}(E_{m+n})/E_{m,n}|_p$ is of dimension $6mn$, so is a maximal isotropic subspace, and this is moreover positively oriented with respect to our chosen orientations of $E_m, E_n$ and $E_{m+n}$.
\end{lemma}
\begin{proof}
The fact that $\phi(F_{m,n})|_p \subseteq \sigma_{m,n}^*(E_{m+n})/E_{m,n}|_p$ is an isotropic subspace follows from the fact that (since the $X_i$ commute),
\[
\sum_{\sigma \in S_4}(-1)^{|\sigma|}\tr(([X_{\sigma(1)}, Y_{\sigma(2)}] + [X_{\sigma(2)}, Y_{\sigma(1)}])([
X_{\sigma(3)},Y_{\sigma(4)}] + [X_{\sigma(4)}, Y_{\sigma(3)}])) = 0.
\]
We now show that $\phi(F_{m,n})|_p$ has the maximal dimension $6mn$.
Let us write $V_i$ for $\cV_i|_p$, and let $W = \Hom(V_m, V_n) \oplus \Hom(V_n, V_m)$.
The image of $\phi|_p$ is identified with the image of a map
\[
\C^4 \otimes W \to \Lambda^2 \C^4 \otimes W.
\]
Letting $\C^3 = \langle e_2,e_3,e_4\rangle$, it is enough to show that the composed homomorphism
\[
\C^3 \otimes W \to \C^4 \otimes W \to \Lambda^2 \C^4 \otimes W \to e_1 \otimes \C^3 \otimes W
\]
has rank $\ge 6mn$.
This homomorphism is explicitly described by
\begin{equation}
\label{eqn:explicitForm}
(Y_2,Y_3,Y_4) \mapsto ([Y_2,X_1], [Y_3,X_1], [Y_4,X_1]).
\end{equation}
After acting by $\mathrm{GL}(4,\C)$, we may assume, since $p \in O_{m,n}$, that the spectrum of $X_1$ acting on $V_m$ is disjoint from the spectrum of $X_1$ acting on $V_n$.
It then follows from Lemma \ref{thm:linearAlgebra} that \eqref{eqn:explicitForm} is injective, i.e.~has rank $6mn$.

We have shown that, on $\Hilb^{m}(\C^4) \times \Hilb^{n}(\C^4)|_{U_{m,n}}$, $\im(\phi)$ is an isotropic subbundle of maximal rank. 
It remains to see that it is positively oriented.
By connectedness of $\Hilb^{m}(\C^4) \times \Hilb^{n}(\C^4)|_{U_{m,n}}$, this can be checked at any given point, and we may choose a point $(Z_1,Z_2)$ contained in the line $\{x_1 = x_2 = x_3 = 0\} \subset \C^4$.
For our choice of $\Lambda$ in \eqref{defLambda}, here a simple computation shows that in fact $\im(\phi) = \Lambda_{m+n}/ \Lambda_{m,n}$. 
\end{proof}

\begin{lemma} \label{lem:N}
On $\Hilb^{m}(\C^4) \times \Hilb^{n}(\C^4)|_{U_{m,n}}$, the section $s_{m+n}$ of $E_{m+n}$ induces an inclusion of vector bundles
\[
ds_{m+n} \colon N_{A_m \times A_n/A_{m+n}} \into \sigma_{m,n}^*(E_{m+n})/E_{m,n}.
\]
With respect to the orientations on $E_{m,n}$ and $E_{m+n}$ given by $\Lambda_{m,n}$ and $\Lambda_{m+n}$, the bundle $ds_{m+n}(N_{A_m \times A_n/A_{m+n}})$ is a positive maximal isotropic subbundle of the quadratic vector bundle $\sigma^*_{m,n}(E_{m+n})/E_{m,n}$.
\end{lemma}
\begin{proof}
We work on $\Hilb^{m}(\C^4) \times \Hilb^{n}(\C^4)|_{U_{m,n}}$, and suppress all pullbacks from the notation.
We first show that the homomorphism $$ds_{m+n} \colon N_{A_m \times A_n/A_{m+n}} \to \sigma^*_{m,n}(E_{m+n})/E_{m,n}$$ is injective as a map of vector bundles.
Since the sections $s_m,s_n$ and $s_{m+n}$ vanish on $U_{m,n}$, we have a commutative diagram
\begin{equation}
\label{eqn:commDiagT}
\begin{tikzcd}
    T_{A_m \times A_n} \arrow[r] \arrow[d, "ds_m \oplus ds_n"] & T_{A_{m + n}} \arrow[d, "ds_{m+n}"] \\
    E_m \oplus E_n \arrow[r] & E_{m+n}
\end{tikzcd}
\end{equation}
and taking cokernels gives the homomorphism 
\[
N_{A_m \times A_n/A_{m+n}} \to E_{m+n}/(E_m \oplus E_n).
\]
We have a short exact sequence
\[
0 \to \cEnd(\cV_{m+n}) \to \C^4 \otimes \cEnd(\cV_{m+n}) \oplus \cV_{m+n} \to T_{A_{m+n}} \to 0.
\]
There is moreover a natural homomorphism $F_{m,n} \to \C^4 \otimes \cEnd(\cV_{m+n})$.
The homomorphism $\phi$ agrees with the composed homomorphism
\[
F_{m,n} \to \C^4 \otimes \cEnd(\cV_{m+n}) \to T_{A_{m+n}} \to N_{A_m \times A_n/A_{m+n}} \to E_{m+n}/(E_m \oplus E_n).
\]
By Lemma \ref{thm:inclusionOfMaximaIsotropic}, our claims follow.
\end{proof}

\begin{lemma} \label{lem:detN}
There is a $\TT'$-equivariant isomorphism $\det(N_{A_m \times A_n / A_{m+n}}) \cong \O$.
\end{lemma}
\begin{proof}
Let $N := N_{A_m \times A_n/A_{m+n}}$. In $K_{\TT'}^0(O_{m,n})$, we have the identity $$N = T_{A_{m+n}}|_{O_{m,n}} - (T_{A_m} \boxplus T_{A_n})|_{O_{m,n}}.$$ Since the determinant map factors through $K$-theory, using \eqref{eqn:TAinKtheory}, we immediately find $\det(N) \cong \O$ ($\TT'$-equivariantly). 
\end{proof}

We now have an unramified morphism $O_{m,n} \to A_{m+n}$ and quadratic vector bundles $E_{m,n}$, $E_{m+n}$ with maximal isotropic subbundles $\Lambda_{m,n}$, $\Lambda_{m+n}$ (and hence orientations) and compatible sections $(s_m,s_n)$, $s_{m+n}$. By Example \ref{ex:sqrtdetgivesspin}, using the square roots of $\det(\Lambda_{m+n})$, $\det(\Lambda_m)$, $\det(\Lambda_n)$ in Section \ref{sec:definv} (see \eqref{eqn:detLambda}), we thus have spin modules
\[
S_m \otimes S_n, \quad S_{m+n}
\]
for $C(E_{m,n}) \cong C(E_m) \, \widehat{\otimes} \, C(E_n)$ and $C(E_{m+n})$. It remains to verify the local structure described in \eqref{eqn:specialFormOfS}.

\begin{proposition} \label{prop:localstruccheck}
The isotropic section $s_{m+n}$ satisfies \eqref{eqn:specialFormOfS} in suitable local formal coordinates. As a consequence, we have an induced $\TT'$-equivariant isomorphism
\begin{align*}
\sigma_{m,n}^* (\cH(S_{m+n},s_{m+n})) &\cong \cH((S_m,s_m) \otimes (S_n,s_n))|_{O_{m,n}} \\
&\cong \cH(S_m,s_m) \otimes \cH(S_n,s_n) |_{O_{m,n}}.
\end{align*}
\end{proposition}
\begin{proof}
	We may write $A_m$ as a GIT quotient $Z_m/\GL(m,\CC)$, where $Z_m$ is the open semistable locus of $\C^{4m^2 + m}$ denoted $U$ in Section \ref{sec:zeroloc}.
	The map $A_m \times A_n \dashrightarrow A_{m+n}$ arises from the natural map $Z_m \times Z_n \dashrightarrow Z_{m + n}$ along with the group inclusion $\GL(m,\CC) \times \GL(n, \CC) \to \GL(m+n,\CC)$.
	
	We fix a point $(P,Q) \in \Hilb^m(\C^4) \times \Hilb^n(\C^4)|_{U_{m,n}}$ at which we want to prove the claim.
    Say $(P,Q)$ is the image of $(\widetilde P, \widetilde Q) \in Z_m \times Z_n$, with  $\widetilde P = (\mathbf{A}_1, \ldots, \mathbf{A}_4,u)$ and $\widetilde Q = (\mathbf{D}_1, \ldots, \mathbf{D}_4,v)$.
	After acting by $\GL(4,\CC)$, we may assume that the spectra of $\mathbf{A}_1$ and $\mathbf{D}_1$ are disjoint.
	
    The map $Z_m \times Z_n \dashrightarrow Z_{m +n}$ sends $(\widetilde{P},\widetilde{Q})$ to 
	\[
	\left( \begin{pmatrix}
		\mathbf{A}_1 & 0 \\
		0 & \mathbf{D}_1
	\end{pmatrix}, \begin{pmatrix}
		\mathbf{A}_2 & 0 \\
		0 & \mathbf{D}_2
	\end{pmatrix}, \begin{pmatrix}
		\mathbf{A}_3 & 0 \\
		0 & \mathbf{D}_3
	\end{pmatrix}, \begin{pmatrix}
		\mathbf{A}_4 & 0 \\
		0 & \mathbf{D}_4
	\end{pmatrix},\begin{pmatrix}
		u \\
		v
	\end{pmatrix} \right).	\]
	Here $\mathbf{A}_i$ has size $m \times m$, $\mathbf{D}_i$ has size $n \times n$, $\C \langle \mathbf{A}_1, \ldots, \mathbf{A}_4 \rangle \cdot u = \C^m$, $\C \langle \mathbf{D}_1, \ldots, \mathbf{D}_4 \rangle \cdot v = \C^n$, and $\mathbf{A}_1, \mathbf{D}_1$ have disjoint spectra.
	We define $\widetilde{Z}_{m,n}$ as the locally closed subvariety of $Z_{m+n}$ of all such tuples satisfying these constraints.
	We next define $\widetilde Z_{m+n}$ to be the closed subvariety of $Z_{m+n}$ of tuples of the form
	\[
	\left( \begin{pmatrix}
		\mathbf{A}_1 & 0 \\
		0 & \mathbf{D}_1
	\end{pmatrix}, \begin{pmatrix}
		\mathbf{A}_2 & \mathbf{B}_2 \\
		\mathbf{C}_2 & \mathbf{D}_2
	\end{pmatrix}, \begin{pmatrix}
		\mathbf{A}_3 & \mathbf{B}_3 \\
		\mathbf{C}_3 & \mathbf{D}_3
	\end{pmatrix}, \begin{pmatrix}
		\mathbf{A}_4 & \mathbf{B}_4 \\
		\mathbf{C}_4 & \mathbf{D}_4
	\end{pmatrix} ,\begin{pmatrix}
		u \\
		v
	\end{pmatrix} \right).	 
	\]
	 The diagram
	\begin{equation}
    \label{eqn:ZTildeDiagram}
	\begin{tikzcd}
		\widetilde{Z}_{m,n} \arrow[hookrightarrow]{r} \arrow[d] & \widetilde Z_{m+n} \arrow[d]\\
		A_{m} \times A_n \arrow[dashed,r] & A_{m + n}
	\end{tikzcd}
	\end{equation}
	is commutative.
	We claim both vertical arrows are submersions locally near $(\widetilde{P},\widetilde{Q})$.
	The left hand morphism is locally a free group quotient, hence a submersion.
	For the right hand morphism, we may describe the tangent space of $A_{m+n}$ at a point by the cokernel in the short exact sequence
	\[
	0 \to \mathfrak{gl}_{m+n} \to \CC^4 \otimes \End(\CC^{m+n}) \oplus \C^{m+n} \to T \to 0.
	\]
	The vertical arrow maps the tangent space $T'$ of $\widetilde Z_{m+n}$ at a point $(\widetilde{P},\widetilde{Q})$ to the corresponding subspace of $\CC^4 \otimes \End(\CC^{m+n}) \oplus \C^{m+n}$, and the claim to be verified is that the composition of this to $T$ is surjective.
	
	We claim that $T' \cap \mathfrak{gl}_{m+n} \subseteq \mathfrak{gl}_m \times \mathfrak{gl}_n$.
	To see this, note the embedding of $\mathfrak{gl}_{m+n}$ in $\CC^4 \otimes \End(\CC^{m+n}) \oplus \C^{m+n}$ at a point $(X_1, \ldots, X_4,u)$ is given by commutators $g \mapsto ([g,X_1], \dots, [g,X_4],g u)$.
	Writing
	\[
	g = \begin{pmatrix} g_{11} & g_{12} \\ g_{21} & g_{22} \end{pmatrix}, \quad X_1 = \begin{pmatrix} \mathbf{A}_1 & 0 \\ 0 & \mathbf{D}_1 \end{pmatrix},
	\]
	the condition that $[g, X_1]$ is block diagonal means that $\begin{pmatrix} 0 & g_{12} \\ g_{21} & 0 \end{pmatrix}$ commutes with $X_1$.
	Since the spectra of $\mathbf{A}_1$ and $\mathbf{D}_1$ are assumed disjoint, this implies $g_{12}$ and $g_{21}$ are both 0 by Lemma \ref{thm:linearAlgebra}. 
	A dimension count now implies the surjectivity we want.
	
	Working around the point $(\widetilde{P},\widetilde{Q})  \in \widetilde{Z}_{m,n}$, imposing a general system of $m^2 + n^2$ linear equations on the $\mathbf{A}_i$ and $\mathbf{D}_i$, and $(u,v) \in \CC^{m+n}$, we get transversal slices of the vertical maps in \eqref{eqn:ZTildeDiagram}.\footnote{Note: in the previous paragraph, $\mathbf{A}_i$ and $\mathbf{D}_i$ were fixed, whereas now they vary.}
	We therefore obtain a formal local model of $A_m \times A_n \dashrightarrow A_{m +n}$ in this way.
	Letting the $x_i$ be the coordinates arising from the $\mathbf{A}_j$, $\mathbf{D}_j$, $u,v$ and the $y_i$ be the coordinates arising from the $\mathbf{B}_j$ and $\mathbf{C}_j$, we find that $A_{m} \times A_n$ is cut out by the vanishing of the $y_i$. We furthermore shift the coordinates $x_i$ so that the point defined by putting all $x_i$ and $y_i$ equal to zero corresponds to $(\widetilde{P},\widetilde{Q})$.
	
	On $\widetilde Z_{m + n}$, the bundle $E_{m + n}$ is identified with the trivial vector bundle with fibre $\Lambda^2 \CC^4 \otimes \End(\CC^{m+n})$, and similarly the bundles $E_{m}, E_n$ are identified with the trivial vector bundles with fibre $\Lambda^2 \CC^4 \otimes \End(\CC^m)$ and $\Lambda^2 \CC^4 \otimes \End(\CC^n)$ respectively.
	We let $E_{m,n}'$ correspond to the trivial vector bundle with fibre  $\Lambda^2 \CC^4 \otimes (\Hom(\CC^m, \CC^n) \oplus \Hom(\CC^n, \CC^m))$. 
    We then get a split exact sequence of quadratic vector bundles
	\[
	0 \to E_m \oplus E_n \to E_{m+n} \to E_{m,n}' \to 0.
	\]
	 	
	The canonical sections of $E_{m+n}$, $E_m$, and $E_n$ are quadratic in the $x_i$ and $y_i$, have no constant term, and have the following properties:
    \begin{itemize}
    \item the terms of the components of the section in $E_m \oplus E_n$ are of the form $x_i, x_ix_j$ or $y_iy_j$;
    \item the terms of the components of the section in $E_{m,n}'$ are of the form $x_i y_j$ or $y_i$;
    \item the degree 1 terms in the $y_i$ define an inclusion of the normal bundle $N_{\widetilde Z_{m,n}/\widetilde Z_{m+n}}$ in $E_{m,n}'$ (Lemma \ref{lem:N}).
    \end{itemize} 
	
This means that Assumption \ref{assume:LocalForm} holds.
We may therefore apply Lemma \ref{thm:comparisonHomomorphismExists}. 
By Lemma \ref{lem:detN} combined with Proposition \ref{thm:torsionInPicardGroup}, after restricting to $\Hilb^m(\C^4) \times \Hilb^n(\C^4)|_{U_{m,n}}$, any square root of $\det(N_{A_m \times A_n/A_{m+n}})$ is $\TT'$-equivariantly trivial, so we obtain a $\TT'$-equivariant isomorphism
\[
\sigma_{m,n}^* (\cH(S_{m+n},s_{m+n})) \cong \cH((S_m,s_m) \otimes (S_n,s_n))|_{O_{m,n}}.
\]
Finally, since the factors of $(S_m,s_m) \otimes (S_n,s_n)$ pullback from $A_m$ and $A_n$, they depend on different variables and we have $\cH((S_m,s_m) \otimes (S_n,s_n)) \cong \cH(S_m,s_m) \otimes \cH(S_n,s_n)$ from which the result follows.
\end{proof}

We can now define our factorizable sequences.
\begin{lemma}
	\label{thm:compareSqrtK}
	Let $\F_n := \cH(R \nu_{n*} \cH(S_n,s_n))$ for all $n>0$.\footnote{To a bounded complex $\mathcal{C}\udot$, we associate the $\Z/2$-graded coherent sheaf $\cH(\mathcal{C}\udot) :=  H^{\mathrm{odd}}(\mathcal{C}\udot)[1] \oplus  H^{\mathrm{even}}(\mathcal{C}\udot)$ where $H^i(-)$ denotes the $i$th cohomology sheaf.} Then, for each $m,n >0$, we have a $\TT'$-equivariant isomorphism 
	\[
	\phi_{m,n} \colon \F_m \boxtimes \F_n|_{U_{m,n}} \cong \sigma_{m,n}^* \F_{m+n}.
	\]
\end{lemma}
\begin{proof}
	Pulling back the isomorphism of Proposition \ref{prop:localstruccheck} to $\Hilb^m(\C^4)\times \Hilb^n(\C^4)|_{U_{m,n}}$, applying $R \nu_{m,n*}(-)$, using the fact that the bottom square in \eqref{diag:sqrs1} is cartesian, $\sigma_{m,n}$ is \'etale so in particular flat, and applying the K\"unneth formula for derived pushforward, we obtain 
	\begin{align*}
	\sigma_{m,n}^* R \nu_{m+n*}  \cH(S_{m+n},s_{m+n}) &\cong R \nu_{m,n*}  (\cH(S_m,s_m) \otimes  \cH(S_n,s_n))|_{U_{m,n}} \\
	& \cong R \nu_{m*} \cH(S_m,s_m) \boxtimes^{L} R \nu_{n*} \cH(S_n,s_n)|_{U_{m,n}}.
	\end{align*}
	Applying $\cH(-)$, using again that $\sigma_{m,n}$ is flat, and applying the K\"unneth formula for cohomology sheaves, the result follows. 
\end{proof}

\begin{proposition} \label{prop:factor}
The following are factorizable sequences of $\mathbb T'$-equivariant $\Z/2$-graded coherent sheaves on $\Sym^\mdot(\C^4)$ 
\begin{align*}
\Big\{ \cH(R \nu_{n*}  ( \cH(S_n,s_n) )) \Big\}_{n=1}^{\infty}, \quad \Big\{ \cH(R \nu_{n*} (\Lambda^\mdot (\cV^\vee_n \otimes y))) \Big\}_{n=1}^{\infty}, \quad \Big\{ \cH(R \nu_{n*} (\cN^\sheaf_{1,n})) \Big\}_{n=1}^{\infty}.
\end{align*}
\end{proposition}
\begin{proof}
Let us start with factorizability of the first sequence. Denote by $O_{m,n,p} \subset A_m \times A_n \times A_p$ the open locus representing triples $$((X_1, \ldots, X_4,u), (Y_1, \ldots, Y_4,v), (Z_1, \ldots, Z_4,w)),$$ for which there exists an $(a_1, \ldots, a_4) \in \C^4 \setminus 0$ such that $\sum_i a_i X_i$, $\sum_i a_i Y_i$, $\sum_i a_i Z_i$ have mutually disjoint spectra. We then have a commutative diagram
\begin{equation} \label{diag:sqrs}
\begin{tikzcd} 
O_{m,n,p}  \arrow[r, " "] & A_{m+n+p} \\
\Hilb^{m}(\C^4) \times \Hilb^{n}(\C^4) \times \Hilb^p(\C^4) |_{U_{m,n,p}} \arrow[r, " "] \arrow[u, hook] \arrow[d, "\nu_{m,n,p}"] & \Hilb^{m+n+p}(\C^4) \arrow[u, hook] \arrow[d, "\nu_{m+n+p}"] \\
 U_{m,n,p} \arrow[r, "\sigma_{m,n,p}"] & \Sym^{m+n+p}(\C^4)
\end{tikzcd}
\end{equation}
in which the squares are cartesian, the top horizontal arrow is unramified and the bottom two horizontal arrows are \'etale. We apply Lemma \ref{lem:associativity} to the two ways of decomposing $\sigma_{m,n,p}$
\begin{displaymath}
\xymatrix
{
& O_{m+n,p} \ar[rd]^{\sigma_{m+n,p}} & \\
O_{m,n,p} \ar[ru]^{\sigma_{m,n} \times \mathrm{id}} \ar[rd]_{\mathrm{id} \times \sigma_{n,p}} \ar^{\sigma_{m,n,p}}[rr] & & A_{m+n+p} \\
& O_{m,n+p} \ar_{\sigma_{m,n+p}}[ru] &
}
\end{displaymath}
We then obtain a diagram (suppressing some obvious pullbacks to open subsets and suppressing the sections for notational brevity)
\begin{equation}
\label{eqn:actualCommutativity}
\vcenter{\hbox{\xymatrix@C=0.1pc@M=2pt
{
& \sigma_{m,n}^* \cH(S_{m+n}) \otimes \cH(S_p) \ar[rd] & \\
\sigma_{m,n,p}^* \cH(S_{m+n+p}) \ar[ru] \ar[rd] & & \cH(S_m) \otimes \cH(S_n) \otimes \cH(S_p) \\
& \cH(S_m) \otimes \sigma_{n,p}^* \cH(S_{n+p}) \ar[ru] &
}}}
\end{equation}
however this diagram may \emph{not} commute. The reason for this is that, for each pair $(m,n) \in \Z_{>0}^2$, we pick a $\TT'$-equivariant isomorphism of a square root of $\det(N)$ with $\O$ at the end of the proof of Proposition \ref{prop:localstruccheck}. In particular, going through the previous diagram via the top two arrows requires picking $\TT'$-equivariant isomorphisms of square roots of $\det(N_{A_m \times A_n \times A_p/A_{m+n} \times A_p})$ and $\det(N_{A_{m+n} \times A_p / A_{m+n+p}})$ with $\O$, whereas going through the diagram via the bottom two arrows requires picking $\TT'$-equivariant isomorphisms of square roots of $\det(N_{A_m \times A_n \times A_p/A_{m} \times A_{n+p}})$ and $\det(N_{A_m \times A_{n+p} / A_{m+n+p}})$ with $\O$. Therefore, the diagram is only guaranteed to commute up to an element of 
\[
H^0_{\TT'}(\Hilb^{m}(\C^4) \times \Hilb^{n}(\C^4) \times \Hilb^p(\C^4) |_{U_{m,n,p}}, \O^{\times}).
\]
The trivializations of $\det(N_\bullet)$ from Lemma \ref{lem:detN} can be made explicit using the exact sequence \eqref{eqn:tangentBundleExactSequence}.
They are compatible up to a global scalar,\footnote{The authors have not ruled out the possibility that this scalar is in fact 1.} in the sense that for all $m,n,p$, the composed homomorphism
\begin{align*}
\O = \O \otimes \O &\cong \det(N_{A_m \times A_n \times A_p/A_{m+n} \times A_p}) \otimes \det(N_{A_{m+n} \times A_p/A_{m+n+p}}) \\
&\cong \det(N_{A_m \times A_n \times A_p/A_{m+n+p}}) \cong \O
\end{align*}
is constant.

Now for each square root line bundle $(\det(N_\bullet))^{1/2}$, we may choose a trivialization such that $1 \otimes 1$ maps to $1$.
As this is done on a connected scheme, there are precisely two such trivializations, differing by a sign.
It follows that the trivializations of the square roots are compatible with the tensor products up to a global scalar, and so the diagram \eqref{eqn:actualCommutativity} commutes up to a scalar.
This is enough by Proposition \ref{prop:uptocst}. 

Similarly, applying Lemma \ref{lem:commutativity}, one verifies the commutativity property of Definition \ref{def:factorisableSequence}. 
Again, commutativity may only hold up to a non-zero constant due to our choices of trivialization of $\det(N_{A_m \times A_n/A_{m + n}})^{1/2}$, arguing as above.
Using Proposition \ref{prop:uptocst}, this constant equals either $1$ or $(-1)^{mn}$, implying that either (case A) $\{\cH(R \nu_{n*}  ( \cH(S_n,s_n)))\}$ or (case B) $\{\cH(R \nu_{n*}  ( \cH(S_n,s_n)))[n]\}$ is a factorizable sequence.

Leaving the question of whether case A or B holds aside for the moment, note that factorizability of the second sequence is straightforward by using the fact that $\cV_{m+n}|_{O_{m,n}} \cong \cV_m \oplus \cV_n$ and therefore
\[
\Lambda^\mdot(\cV_{m+n}^\vee \otimes y)|_{O_{m,n}} \cong  \Lambda^\mdot(\cV_m^\vee \otimes y) \otimes \Lambda^\mdot(\cV_n^\vee \otimes y).
\]

Taking the tensor product of the isomorphisms on $\Hilb^m(\C^4) \times \Hilb^n(\C^4)|_{U_{m,n}}$ used to show factorizability of the first two sequences, we get factorizability of either (case A) $\{\cH(R \nu_{n*} (\cN^\sheaf_{1,n}))\}$ or (case B) $\{\cH(R \nu_{n*} (\cN^\sheaf_{1,n}))[n]\}$.
We now prove that case B cannot hold.
Assume for a contradiction that it did.
In Section \ref{sec:local}, we determine $\mathsf{G}_1$ by $\TT'$-localization and in particular, by Proposition \ref{signprop}, we find
\[
\sum_{n=0}^\infty \chi(\Sym^n(\C^4), R \nu_{n*} (\cN^\sheaf_{1,n})[n])q^n = \mathsf{G}_1(-q) = \mathsf{Z}_1^{\NP}(-q).
\]
Moreover, the methods of Section \ref{sec:local} readily determine this generating series explicitly modulo $q^3$. By Lemma \ref{prop:Oko} (combined with the localization argument in the proof of the next corollary), we also have
\[
\mathsf{Z}_1^{\NP}(-q) = \Exp\left(\frac{G_1}{[t_1][t_2][t_3][t_4]}\right)
\]
for some $G_1 \in qK_0^{\mathbb T'}(\pt)[[q]]$.
But taking the plethystic logarithm of $\mathsf{Z}_1^{\NP}(-q)$ shows that the $q^2$-term of $G_1$ must then be
\[
-[t_1t_2][t_1t_3][t_2t_3][y^2] - \frac{[t_1^2t_2^2][t_1^2t_3^2][t_2^2t_3^2][y^2]}{(1+t_1)(1+t_2)(1+t_3)(1+t_4)},
\]
which is not a Laurent polynomial.
This gives a contradiction, and so case A holds, proving our claim.
\end{proof}

Recall that $\mathsf{G}_1 = \sum_n N_{1,n}^{\glob} q^n$.
\begin{corollary} \label{Exp}
There exists a unique $G_1 \in q K_0^{\T'}(\pt) [[q]]$ such that
\begin{align*}
\mathsf{G}_{1} = \mathrm{Exp}\Bigg(\frac{G_1}{[t_1][t_2][t_3][t_4]}\Bigg).
\end{align*}
\end{corollary}
\begin{proof}
The coefficients of $\mathsf{G}_1$ are
\[
N^\glob_{1,n} = \chi(\Hilb^n(\C^4),\cN^\sheaf_{1,n}) = \chi(\Sym^n(\C^4), R \nu_{n*} (\cN^\sheaf_{1,n})), 
\]
and by Proposition \ref{prop:factor}, the sequence $\{\cH(R \nu_{n*} (\cN^\sheaf_{1,n}))\}$ is factorizable.
Hence, by Lemma \ref{prop:Oko}, there exist $\mathbb T'$-equivariant $K$-theory classes $\G_{n}$ on $\C^4$ such that
\[
\mathsf{G}_{1} = \mathrm{Exp}\Bigg(\sum_{n=1}^\infty \chi(\C^4, \G_ n) \, q^n\Bigg).
\]
Since $\C^4$ has $0$ as its only fixed point, the $K$-theoretic localization formula implies
$$
\G_n = \frac{\iota_* (\G_n|_0)}{(1-t_1)(1-t_2)(1-t_3)(1-t_4)},
$$
where $\G_n|_0 \in  K_0^{\T'}(\pt)$ and $\iota \colon \{0\} \hookrightarrow \C^4$ is the inclusion of the origin.
\end{proof}

\begin{remark}
Using the localization techniques discussed in Section \ref{sec:local}, one can compute $\chi(\Hilb^n(\C^4),\widehat{\O}^{\vir})$ for $n=1,2$. This reveals that 
$$
\sum_{n=0}^{\infty} \chi(\Hilb^n(\C^4),\widehat{\O}^{\vir}) \, q^n
$$
is \emph{not} of the form stated in Lemma \ref{prop:Oko}. This was observed by Arbesfeld. 
Indeed, in this section, we do \emph{not} take square roots of the line bundles $\det(T_{A_n})$, but only of the product $\det(T_{A_n}) \otimes  \det (\cV_n^\vee \otimes y)$. 
In particular, the  $K$-theory classes $R \nu_{n*} \widehat{\O}^{\vir}$ cannot come from a factorizable sequence.
On the other hand, by Proposition \ref{prop:factor}, factorizability \emph{does} apply to the classes $R \nu_{n*} (\widehat{\O}^{\vir} \otimes  \det(T_{A_n}|_{\Hilb^n(\CC^4)} )^{\frac{1}{2}} )$.
\end{remark}

\section{Local} \label{sec:local}

In this section, we apply virtual localization (specifically, Proposition \ref{prop:nP}) to show that $\mathsf{G}_r = \mathsf{Z}_r^{\mathrm{NP}}$, i.e., the globally defined invariants on $\Quot_r^n(\C^4)$ from Section \ref{sec:definv} indeed give the weighted generating function of $r$-tuples of solid partitions introduced by Nekrasov--Piazzalunga. This section does not require the factorizability results from Section \ref{sec:factorforNek}.

\subsection{Localization for Quot schemes} \label{sec:vertex}

As before, consider $\Quot_r^n(\C^4)$ on which we have an action of $\T'$ (recall the definition of $\T'$ from Remark \ref{def:Tprime}). 
By \cite{Bif} and \cite[Lem.~3.6]{CK}, the $\T'$-fixed locus of $\Quot_r^n(\C^4)$ is scheme theoretically isomorphic to
$$
\bigsqcup_{n = n_1 + \cdots +n_r} \prod_{\alpha=1}^{r} \Hilb^{n_\alpha}(\C^4)^{(\C^*)^4},
$$
where we use the standard action of $(\C^*)^4$ on $\C^4$ lifted to $\Hilb^m(\C^4)$. Moreover, $\Hilb^{m}(\C^4)^{(\C^*)^4}$ is 0-dimensional reduced and identified with the set of monomial ideals of colength $m$, or equivalently solid partitions of size $m$. 
Therefore, we can index the points $P_{\vec{\pi}}$  of $\Quot_r^n(\C^4)^{\T'}$ by $r$-tuples of solid partitions $\vec{\pi} = (\pi_1, \ldots, \pi_r)$ with total size $n$.

Consider a $\T'$-fixed point
$$
P_{\vec{\pi}} =[\O_{\C^4}^{\oplus r} \twoheadrightarrow Q], \quad  Q = \bigoplus_{\alpha=1}^{r} \O_{Z_\alpha},
$$ 
where $Z_\alpha \subset \C^4$ corresponds to the solid partition $\pi_\alpha$. By the Oh--Thomas localization formula, as discussed in Section \ref{sec:localfixisolated}, we are interested in 
$$
(T_{A_n} - \Lambda)|_{P_{\vec{\pi}}} \in K_0^{\T'}(\pt),
$$
where $A_n:=\ncQ_{r}^{n}(\C^4)$ denotes the ambient non-commutative Quot scheme. 
Recall our choice of $\Lambda$ \eqref{defLambda}.
For computational reasons, and for easy comparison to Nekrasov--Piazzalunga's work \cite{NP}, it will be more convenient to work with $(T_{A_n} - \Lambda^\vee)|_{P_{\vec{\pi}}}$.
Recall that $A_n$ was described as a quotient of a smooth space by a free action (Section \ref{sec:zeroloc}), from which we obtain a short exact sequence
$$
0 \rightarrow \End(V) \rightarrow \C^4 \otimes \End(V) \oplus \Hom(\C^r,V) \rightarrow T_{A_n}|_{P_{\vec{\pi}}} \rightarrow 0.
$$
Hence, in $K_0^{\T'}(\mathrm{pt})$, we have
\begin{align*}
T_{A_n} |_{P_{\vec{\pi}}} = (\C^4 - 1) \cdot \End(V) + \Hom(\C^r,V), \\  
\Lambda^\vee|_{P_{\vec{\pi}}} = (\langle e_4 \rangle \wedge \C^3)^\vee \cdot \End(V) \cong \Lambda^2 \C^3 \otimes \End(V),
\end{align*}
where
\begin{align*}
\C^4 &= t_1^{-1}+t_2^{-1}+t_3^{-1}+t_4^{-1},  \\
\C^r &= w_1 + \cdots + w_r, \\
V &= Z_1 w_1 + \cdots + Z_r w_r,  \\
\Lambda^2 \C^3 &= t_1^{-1}t_2^{-1} +t_1^{-1} t_3^{-1} + t_2^{-1} t_3^{-1}, \\
Z_\alpha &= \sum_{(i,j,k,l) \in \pi_\alpha} t_1^{i} t_2^j t_3^k t_4^l. 
\end{align*}
Therefore 
\begin{align*}
&(T_{A_n} - \Lambda^\vee)|_{P_{\vec{\pi}}} = \\
&\sum_{\alpha,\beta=1}^{r} \frac{w_\beta}{w_\alpha}\Big( Z_\beta - (1- t_1^{-1} - t_2^{-1} - t_3^{-1} - t_4^{-1} + t_1^{-1}t_2^{-1} + t_1^{-1}t_3^{-1} + t_2^{-1}t_3^{-1}) Z^\vee_\alpha Z_\beta \Big),
\end{align*}
where $(-)^\vee$ denotes the dual of a class in $K$-theory. 
We add to this the representation 
$$
- \hom(\cW,\cV)|_{P_{\vec{\pi}}}^\vee = \sum_{\alpha,\beta=1}^{r} \frac{w_\beta}{w_\alpha}\Big( - y_\beta Z_\alpha^\vee \Big).
$$
For the class of a line bundle $L = t^a w^b y^c \in K_0^{\T'}(\textrm{pt})$, we define
$$
[t^a w^b y^c] = t^{\frac{a}{2}} w^{\frac{b}{2}} y^{\frac{c}{2}} - t^{-\frac{a}{2}} w^{-\frac{b}{2}} y^{-\frac{c}{2}},
$$
where we use multi-index notation for $t=(t_1,t_2,t_3, t_4)$, $w = (w_1, \ldots, w_r)$, and $y = (y_1, \ldots, y_r)$. This operation extends to any class $V \in K_0^{\T'}(\textrm{pt})$ with $V^f \geq 0$ by
$$
[L_1+L_2] = [L_1][L_2], \quad [L_1-L_2] = \frac{[L_1]}{[L_2]}.
$$
Geometrically, this operation can be described as $[V] = \widehat{\Lambda}^{\bullet} V^\vee$.

\begin{definition} \label{def:NPweight}
Define
\begin{align*}
\mathsf{v}_{\vec{\pi}, \alpha\beta}^{\textrm{pre}} &= \frac{w_\beta}{w_\alpha}   \Big( Z_\beta  - (1-t_1^{-1})(1-t_2^{-1})(1-t_3^{-1}) Z^\vee_\alpha Z_\beta \Big), \quad \mathsf{v}_{\vec{\pi}}^{\textrm{pre}} = \sum_{\alpha,\beta=1}^{r} \mathsf{v}_{\vec{\pi},\alpha\beta}^{\textrm{pre}}, \\
\mathsf{v}_{\vec{\pi}, \alpha\beta} &= \mathsf{v}_{\vec{\pi}, \alpha\beta}^{\textrm{pre}} + \frac{w_\beta}{w_\alpha}   \Big( - y_\beta Z^\vee_\alpha  \Big), \quad \mathsf{v}_{\vec{\pi}} = \sum_{\alpha,\beta=1}^{r} \mathsf{v}_{\vec{\pi},\alpha\beta}.
\end{align*}
The weight assigned to $\vec{\pi}$ by Nekrasov--Piazzalunga \cite{NP} equals $[-\mathsf{v}_{\vec{\pi}}]$. 
\end{definition}

By \cite[Sect.~2.4.1]{NP}, $\mathsf{v}_{\vec{\pi}, \alpha\beta}^{\textrm{pre}}$ and $\mathsf{v}_{\vec{\pi}}$ have no fixed terms (with respect to the torus $\T'$). By the proposition below, this also implies 
$$
\vd P_{\vec{\pi}} = 0.
$$

\begin{proposition} \label{Ktheorycalc}
We have the following equality in $K_0^{\T'}(\mathrm{pt})$
\begin{align*}
\mathsf{v}_{\vec{\pi}}^{\mathrm{pre}} + (\mathsf{v}_{\vec{\pi}}^{\mathrm{pre}})^\vee = \Big((T_{A_n} - \Lambda^\vee) + (T_{A_n} - \Lambda^\vee)^\vee \Big) \Big|_{P_{\vec{\pi}}}.
\end{align*}
In particular, $(T_{A_n} - \Lambda^\vee)|_{P_{\vec{\pi}}}$ has no fixed term. Denoting the kernel of the quotient $P_{\vec{\pi}} =[\O_{\C^4}^{\oplus r} \twoheadrightarrow Q]$ by $\F$, we have
$$
\mathsf{v}_{\vec{\pi}}^{\mathrm{pre}} + (\mathsf{v}_{\vec{\pi}}^{\mathrm{pre}})^\vee = R\Hom(\O_{\C^4}^{\oplus r}, \O_{\C^4}^{\oplus r}) - R\Hom(\F,\F).
$$
Finally, we have
\begin{align*}
&\Big[ \Big( \hom(\cW,\cV)^\vee + \Lambda^\vee- T_{A_n} \Big) \Big|_{P_{\vec{\pi}}} \Big] = (-1)^{n + k_{\vec{\pi}}} [-\mathsf{v}_{\vec{\pi}}], \\
&k_{\vec{\pi}} = \sum_{\alpha=1}^{r} |\{(i,j,k,l) \in \pi_\alpha \, : \, l \neq \min(i,j,k) \}|.
\end{align*}
\end{proposition}
\begin{proof}
The summand corresponding to $\alpha,\beta$ in $(T_{A_n} - \Lambda^\vee)|_{P_{\vec{\pi}}}$ differs from $\mathsf{v}_{\vec{\pi}, \alpha\beta}^{\textrm{pre}}$ by
\begin{equation} \label{diff}
\frac{w_\beta}{w_\alpha} (t_4^{-1} - t_4) Z_\alpha^\vee Z_\beta,
\end{equation}
from which the first equality follows. 

The class $R\Hom(\O_{\C^4}^{\oplus r}, \O_{\C^4}^{\oplus r}) - R\Hom(\F,\F)$ can be determined using a \v{C}ech calculation similar to \cite{MNOP1} and was worked out for $\Quot_r^n(\C^3)$ in \cite{FMR} and for $\Hilb^n(\C^4)$ in \cite{CK}. The result is
$$
\sum_{\alpha,\beta=1}^{r} \frac{w_\beta}{w_\alpha}   \Big( Z_\beta + \frac{Z^\vee_\alpha}{t_1t_2t_3t_4}  - (1-t_1^{-1})(1-t_2^{-1})(1-t_3^{-1})(1-t_4^{-1}) Z^\vee_\alpha Z_\beta \Big).
$$ 
The second equality of the proposition now follows from direct calculation. 

Let $Q = \sum_{\alpha} Z_\alpha w_\alpha$, then 
$$
\dim (t_4^{-1} Q^\vee Q)^m = n^2 - \dim (t_4^{-1} Q^\vee Q)^f = n^2 - \sum_{\alpha = 1}^{r} \dim (t_4^{-1} Z_\alpha^\vee Z_\alpha)^f,
$$
where $(-)^m$ denotes the moving part and $(-)^f$ denotes the fixed part. Thus, by \eqref{diff}, the final equality of the proposition follows from the following combinatorial lemma.
\end{proof}

\begin{lemma} \label{induct}
For any solid partition $\pi$ and corresponding 0-dimensional subscheme $Z$ we have
$$
\dim (t_1 t_2 t_3 Z^\vee Z)^f \equiv |\{(i,j,k,l) \in \pi \, : \, l \neq \min(i,j,k) \}| \pmod 2.
$$
\end{lemma}
\begin{proof}
Writing $Z = \sum_{(i,j,k,l) \in \pi} t_1^{i-l} t_2^{j-l} t_3^{k-l}$, the desired dimension equals
$$
|\{((i,j,k,l),(i',j',k',l')) \in \pi \times \pi \, : \, i'-i = j' - j = k' - k = l' - l-1\}|.
$$
The lemma can now be proved by induction on $|\pi|$. We order the elements $(i,j,k,l) \in \pi$ of a solid partition $\pi$ using the lexicographic order $x_4 > x_3 > x_2 > x_1$. Let $\pi$ be a solid partition of size $n$. Then $\pi$ is obtained from a solid partition $\xi$ of size $n-1$ by adding a box $b$ with bigger order than any box in $\xi$. Suppose the coordinates of $b$ are $(i_0,j_0,k_0,l_0)$. By the induction hypothesis, it suffices to show that 
$$
|\{ (i,j,k,l) \in \xi \, : \, i_0 - i = j_0 - j = k_0 - k = l_0 - l \pm 1\}| \equiv \left\{ \begin{array}{cc} 1 \pmod 2 & \mathrm{if \, }  l_0 \neq \min(i_0,j_0,k_0) \\ 0 \pmod 2 & \mathrm{otherwise.} \end{array} \right.
$$
Since $b$ has maximal order, we only need to compare it to boxes $(i,j,k,l) \in \xi$ with $i \leq i_0$, $j \leq j_0$, $k \leq k_0$, and $l \leq l_0$. In other words, we may assume $\pi$ is of the form $[0,A-1] \times [0,B-1] \times [0,C-1] \times [0,D-1]$ in which case the above claim easily follows.
\end{proof}

We will now apply the Oh--Thomas localization formula as discussed in Section \ref{sec:localfixisolated} and Proposition \ref{prop:nP}
\begin{align}
\begin{split} \label{virloc2} 
\widehat{\O}_M^{\vir} &= \sum_{P \in M^T} (-1)^{n_P^\Lambda} \, \iota_{P*} \Bigg( \frac{1}{\widehat{\Lambda}^{\mdot} (\Omega_{A} - \Lambda^\vee)|_P } \Bigg), \\
n_P^\Lambda &\equiv  \dim  \mathrm{coker}(p_{\Lambda} \circ ds|_P)^f  \pmod 2,
\end{split}
\end{align}
where $s$ is the isotropic section in Section \ref{sec:zeroloc}.
Note that
$$
\widehat{\Lambda}^{\mdot} \Lambda^\vee|_{P}^m = (-1)^{\rk \Lambda|_{P}^m} \widehat{\Lambda}^{\mdot} \Lambda|_{P}^m.
$$
By Lemma \ref{lem:nnprime}, equation \eqref{virloc2} can be rewritten as (keeping the orientation on $(E,q)$ induced by $\Lambda$)
\begin{align}
\begin{split} \label{virloc3} 
\widehat{\O}_M^{\vir} &= (-1)^{\rk \Lambda} \sum_{P \in M^T} (-1)^{n_P^{\Lambda^\vee}} \,  \iota_{P*} \Bigg( \frac{1}{\widehat{\Lambda}^{\mdot} (\Omega_{A} - \Lambda)|_P } \Bigg), \\
n_P^{\Lambda^\vee} &\equiv  \dim  \mathrm{coker}(p_{\Lambda^\vee} \circ ds|_P)^f  \pmod 2.
\end{split}
\end{align}
We prefer using \eqref{virloc3} over \eqref{virloc2} in order to match existing conventions for the vertex formalism in the literature. Moreover, with our current choice of $\Lambda$ in \eqref{defLambda}, it is easier to calculate the parity of $n_{P}^{\Lambda^\vee}$ as opposed to $n_{P}^{\Lambda}$. The calculation of the parity of $n_{P}^{\Lambda^\vee}$ will occupy the next section. 
 
\begin{proposition} \label{globaltolocal}
We have
\begin{align*}
N_{r,n}^\glob = (-1)^{rn} \sum_{\vec{\pi}} (-1)^{n_{P_{\vec{\pi}}}^{\Lambda^\vee} + k_{\vec{\pi}}} [-\mathsf{v}_{\vec{\pi}}],
\end{align*}
where the sum is over all $r$-tuples $\vec{\pi} = (\pi_1, \ldots, \pi_r)$ of total size $n$.
\end{proposition}
\begin{proof}
By the virtual localization formula, Proposition \ref{Ktheorycalc} (and Remark \ref{dualVW}), the desired invariant equals 
\begin{align*}
N_{r,n}^\glob &= \sum_{\vec{\pi}} (-1)^{n^{\Lambda}_{P_{\vec{\pi}}}} \cdot \widehat{\Lambda}^{\mdot} \hom(\cV,\cW)|_{P_{\vec{\pi}}} \cdot \widehat{\Lambda}^{\mdot} (\Lambda^\vee - \Omega_{A_n})|_{P_{\vec{\pi}}} \\
&=(-1)^{(r-1)n} \sum_{\vec{\pi}} (-1)^{n^{\Lambda^\vee}_{P_{\vec{\pi}}}} \cdot \widehat{\Lambda}^{\mdot} \hom(\cW,\cV)|_{P_{\vec{\pi}}} \cdot \widehat{\Lambda}^{\mdot} (\Lambda - \Omega_{A_n})|_{P_{\vec{\pi}}},
\end{align*}
where we suppress equivariant pushforward to a point. For the class of a line bundle $L$ in $K_0^{\T'}(\mathrm{pt})$, we have $\widehat{\Lambda}^{\mdot} L^\vee = [L]$. Since $\widehat{\Lambda}^{\mdot} (V_1 + V_2) = \widehat{\Lambda}^{\mdot} V_1 \cdot  \widehat{\Lambda}^{\mdot} V_2$, the above reduces to
$$
(-1)^{(r-1)n} \sum_{\vec{\pi}} (-1)^{n^{\Lambda^\vee}_{P_{\vec{\pi}}}}   \Big[ \Big( \hom(\cW,\cV)^\vee + \Lambda^\vee - T_{A_n} \Big) \Big|_{P_{\vec{\pi}}} \Big].
$$
The result now follows from Proposition \ref{Ktheorycalc}.
\end{proof}

\subsection{Sign rule}

As in the previous section, we consider a $\T'$-fixed point of $\Quot_r^n(\C^4)$
$$
P_{\vec{\pi}}=[\O_{\C^4}^{\oplus r} \twoheadrightarrow Q], \quad  Q = \bigoplus_{\alpha=1}^{r} \O_{Z_\alpha},
$$ 
where $Z_\alpha \subset \C^4$ corresponds to a solid partition $\pi_\alpha$. In terms of the ``standard model'' $M:=Z(s) \subset A$, $s \in \Gamma(A,E)$ discussed in the introduction and Section \ref{sec:global}, we want to determine the sign \eqref{virloc3}
$$
(-1)^{n^{\Lambda^\vee}_{P_{\vec{\pi}}}}, \quad n_P^{\Lambda^\vee} \equiv  \dim  \mathrm{coker}(p_{\Lambda^\vee} \circ ds|_P)^f  \pmod 2.
$$

Recall the definitions of $\mathsf{G}_r$ and  $\sfZ_r^{\NP}$ in \eqref{defG} and \eqref{NPZ}.

\begin{proposition} \label{signprop}
We have
$$
n^{\Lambda^\vee}_{P_{\vec{\pi}}} + k_{\vec{\pi}} \equiv \mu_{\vec{\pi}} \pmod 2, \quad \mu_{\vec{\pi}} := \sum_{\alpha=1}^{r}  |\{(i,i,i,j) \in \pi_\alpha \, : \, j > i\}|.
$$
In particular, $\mathsf{G}_r = \sfZ_r^{\NP}$.
\end{proposition}
\begin{proof}
For any $P = [\O_{\C^4}^{\oplus r} \twoheadrightarrow Q] \in A_n:=\ncQ_{r}^{n}(\C^4)$, we have a short exact sequence (Section \ref{sec:zeroloc})
$$
0 \rightarrow \End(V) \rightarrow \C^4 \otimes \End(V) \oplus \Hom(\C^r,V) \rightarrow T_{A_n}|_P \rightarrow 0.
$$
Consider the composition 
$$
\C^4 \otimes \End(V) \oplus \Hom(\C^r,V) \twoheadrightarrow T_{A_n}|_P \stackrel{ds}{\rightarrow} E|_P = \Lambda^2 \C^4 \otimes \End(V).
$$
Let $(\sum_i e_i \otimes f_i^P, u_1^P, \ldots, u_r^P)$ be a representative of $P \in A_n = U / \GL(V)$. Then the above composition is given by
$$
 \Big(\sum_{i=1}^{4} e_i \otimes f_i , u_1, \ldots, u_r \Big) \mapsto \sum_{i,j=1}^{4} e_i \wedge e_j \cdot (f^P_i \circ f_j - f_j \circ f^P_i).
$$
Recall that $\Lambda^\vee|_P  = (\langle e_4 \rangle \wedge \C^3)^\vee \otimes \End(V) \cong \Lambda^2 \C^3 \otimes \End(V)$, where $\C^3 = \langle e_1, e_2, e_3 \rangle \subset \C^4$, so the image of $T_{A_n}|_{P} \to  \Lambda^\vee|_{P}$ equals the image of 
\begin{align}
\begin{split} \label{impds}
\C^3 \otimes \End(V)  &\rightarrow  \Lambda^2 \C^3 \otimes \End(V) \\
 \sum_{i=1}^{3} e_i \otimes f_i  &\mapsto \sum_{i,j=1}^{3} e_i \wedge e_j \cdot (f^P_i \circ f_j - f_j \circ f^P_i).
\end{split}
\end{align}

Consider the projection $p \colon \C^4 \rightarrow \C^3$ onto the first three coordinates. Then the cohomology of $R \Hom(p_* Q, p_* Q)$ is computed by the cohomology of a ($\T'$-equivariant) complex of vector spaces
$$
\End(V) \rightarrow \C^3 \otimes \End(V) \stackrel{\theta_1}{\rightarrow} \Lambda^2 \C^3 \otimes \End(V) \stackrel{\theta_2}{\rightarrow} \Lambda^3 \C^3 \otimes \End(V), 
$$
where the map $\theta_1$ coincides with \eqref{impds}. 

Let $\vec{\pi}$ be the $r$-tuple corresponding to $P = P_{\vec{\pi}}$. Furthermore, we write $Q = \bigoplus_\alpha \O_{Z_\alpha}$, where $Z_\alpha$ corresponds to a solid partition $\pi_\alpha$. Then
$$
\Ext^2(p_* Q, p_* Q)^f = \bigoplus_{\alpha=1}^{r} \Ext^2(p_* \O_{Z_\alpha}, p_* \O_{Z_\alpha})^f = 0,
$$
where the second equality follows from the lemma below. Therefore
\begin{align*}
\mathrm{coker}(\theta_1)^f \cong \Bigg( \frac{\Lambda^2 \C^3 \otimes \End(V)}{\ker(\theta_2)} \Bigg)^f \cong \mathrm{im}(\theta_2)^f,
\end{align*}
and
\begin{align*}
\dim \mathrm{im}(\theta_2)^f &= \dim (\Lambda^3 \C^3 \otimes \End(V))^f - \dim \Ext^3(p_* Q, p_* Q)^f \\
&=\sum_{\alpha=1}^{r} \Big\{ \dim (t_1^{-1}t_2^{-1}t_3^{-1} Z_\alpha^\vee Z_\alpha)^f - \dim \Ext^3(p_* \O_{Z_\alpha}, p_* \O_{Z_\alpha})^f \Big\}.
\end{align*}
By Lemma \ref{induct} and the lemma below, this equals
$$
\sum_{\alpha=1}^{r} (k_{\pi_\alpha} - \mu_{\pi_\alpha}) \pmod 2
$$
as desired. The proposition now follows from the definitions of $\mathsf{G}_r$ \eqref{defG}, $\mathsf{Z}_r^{\NP}$ \eqref{NPZ}, and Proposition \ref{globaltolocal}.
\end{proof}

\begin{lemma}
Denote by $p \colon \C^4 \rightarrow \C^3$ the projection onto the first three coordinates. For any $T_t$-fixed $0$-dimensional subscheme $Z \subset \C^4$ corresponding to a solid partition $\pi$ we have
\begin{align*}
&\dim \Hom(p_* \O_Z, p_* \O_Z)^f = |\{(i,i,i,j) \in \pi \, : \, j \geq i\}|, \\
&\dim \Ext^1(p_* \O_Z, p_* \O_Z)^f = \dim \Ext^2(p_* \O_Z, p_* \O_Z)^f  = 0, \\
&\dim \Ext^3(p_* \O_Z, p_* \O_Z)^f = |\{(i,i,i,j) \in \pi \, : \, j > i\}|. 
\end{align*}
\end{lemma}
\begin{proof}
Let $Z \subset Z' \subset \C^3$ be $T_t$-fixed 0-dimensional subschemes corresponding to \emph{plane} partitions $\pi,\pi'$ and let $(w_1,w_2,w_3) \in \Z^3$.\footnote{Recall that the torus $T_t \subset (\C^*)^4$ is defined by the relation $t_1t_2t_3t_4=1$. Projecting onto the first three coordinates $T_t \cong (\C^*)^3$, which acts in the standard way on $\C^3$.} We first establish the following basic building blocks
\begin{align*}
\Hom(\O_{Z'}, \O_Z \otimes t_1^{-w_1}t_2^{-w_2}t_3^{-w_3})^f &= \C, \quad \mathrm{if \ } (w_1,w_2,w_3) \in \pi \\
\Hom(\O_{Z'}, \O_Z \otimes t_1^{-w_1}t_2^{-w_2}t_3^{-w_3})^f &= 0, \quad \mathrm{if \ } (w_1,w_2,w_3) \notin \pi \\
\Ext^k(\O_{Z'}, \O_Z \otimes t_1^{-w_1}t_2^{-w_2}t_3^{-w_3})^f &= 0, \quad \mathrm{if \ } w_1,w_2,w_3 \geq 0, \, k>0 \\
\Ext^3(\O_{Z}, \O_{Z'} \otimes t_1^{-w_1}t_2^{-w_2}t_3^{-w_3})^f &= \C, \quad \mathrm{if \ } (-w_1-1,-w_2-1,-w_3-1) \in \pi \\ 
\Ext^3(\O_{Z}, \O_{Z'} \otimes t_1^{-w_1}t_2^{-w_2}t_3^{-w_3})^f &= 0, \quad \mathrm{if \ } (-w_1-1,-w_2-1,-w_3-1) \notin \pi \\ 
\Ext^k(\O_{Z}, \O_{Z'} \otimes t_1^{-w_1}t_2^{-w_2}t_3^{-w_3})^f &= 0, \quad \mathrm{if \ } w_1,w_2,w_3 \leq -1, \, k<3.
\end{align*}
The last three statements follow from the first three by using $(\C^*)^3$-equivariant Serre duality and $K_{\C^3} = \O_{\C^3} \otimes t_1t_2t_3$. For the first two statements, we use the short exact sequence
$0 \to I_{Z'} \to \O_{\C^3} \to \O_{Z'} \to 0$, which gives an exact sequence
\begin{align*}
0 \to \Hom(\O_{Z'},\O_Z  \otimes t_1^{-w_1}t_2^{-w_2}t_3^{-w_3}) &\to \Hom(\O_{\C^3},\O_Z  \otimes t_1^{-w_1}t_2^{-w_2}t_3^{-w_3}) \\
&\to \Hom(I_{Z'},\O_Z  \otimes t_1^{-w_1}t_2^{-w_2}t_3^{-w_3}).
\end{align*}
Since $Z \subset Z'$, the third map is zero and we obtain
$$
\Hom(\O_{Z'},\O_Z  \otimes t_1^{-w_1}t_2^{-w_2}t_3^{-w_3}) \cong \Hom(\O_{\C^3},\O_Z  \otimes t_1^{-w_1}t_2^{-w_2}t_3^{-w_3}),
$$
which has 1-dimensional fixed part if and only if $(w_1,w_2,w_3) \in \pi$ and 0-dimensional fixed part otherwise. For the third statement, we use the Taylor resolution \cite{Tay}
$$
0 \to R_t \to \cdots \to R_1 \to \O_{\C^3} \to \O_{Z'} \to 0,
$$
which is a $(\C^{*})^3$-equivariant resolution where each $R_i$ is a direct sum of $(\C^{*})^3$-equivariant line bundles $\O_{\C^3} \otimes t_1^{a_1}t_2^{a_2}t_3^{a_3}$ with $(a_1,a_2,a_3) \in \Z_{\geq 0}^3 \setminus \pi'$. For any such line bundle and $w_1,w_2,w_3 \geq 0$, we have 
$$
\Hom(\O_{\C^3} \otimes t_1^{a_1}t_2^{a_2}t_3^{a_3}, \O_Z \otimes t_1^{-w_1}t_2^{-w_2}t_3^{-w_3})^f =0,
$$
because $(a_1+w_1,a_2+w_2,a_3+w_3) \in \Z_{\geq 0}^3 \setminus \pi' \subset \Z_{\geq 0}^3 \setminus \pi$. This establishes the third statement.

To prove the lemma, we take $Z \subset \C^4$ a $T_t$-fixed $0$-dimensional subscheme corresponding to a \emph{solid} partition $\pi$. Then 
$$
p_* \O_Z = \bigoplus_{i \geq 0} \O_{Z_i} \otimes t_4^{i} = \bigoplus_{i \geq 0} \O_{Z_i} \otimes (t_1t_2t_3)^{-i},
$$
where $Z_0 \supset Z_1 \supset \cdots$ is a finite flag of $(\C^*)^3$-invariant 0-dimensional subschemes of $\C^3$. We denote the corresponding plane partitions by $\pi_0 \supset \pi_1 \supset \cdots$. For the first statement of the lemma, we consider
$$
\Hom(\O_{Z_i} \otimes (t_1t_2t_3)^{-i}, \O_{Z_j} \otimes (t_1t_2t_3)^{-j})^f.
$$
When $i \leq j$, this is 1-dimensional if and only if $(j-i,j-i,j-i) \in \pi_j$ and 0-dimensional otherwise. When $i > j$, this is 0-dimensional. The first statement of the lemma follows. For the second statement of the lemma, we note that for $k=1,2$ we have
$$
\Ext^k(\O_{Z_i}, \O_{Z_j} \otimes (t_1t_2t_3)^{i-j})^f = 0,
$$
for all $i,j$. For the final statement of the lemma, we consider
$$
\Ext^3(\O_{Z_i}, \O_{Z_j} \otimes (t_1t_2t_3)^{i-j})^f. 
$$
When $i \leq j$, this is 0-dimensional. When $i>j$, this is 1-dimensional if and only if $(i-j-1,i-j-1,i-j-1) \in \pi_i$ and 0-dimensional otherwise. The third statement of the lemma follows.
\end{proof}

\section{Limits} \label{sec:limits}

We have established equality of the following power series (Proposition \ref{signprop})
\begin{align*}
\mathsf{G}_{r} = \sum_{n=0}^{\infty} N_{r,n}^\glob \, q^n, \quad \mathsf{Z}_{r}^{\NP} = \sum_{\vec{\pi} = (\pi_1, \ldots, \pi_r)}  (-1)^{ \mu_{\vec{\pi}}} [ -\mathsf{v}_{\vec{\pi}} ]  \, ((-1)^{r} q)^{|\vec{\pi}|}.
\end{align*}
For $r=1$, combining with factorizability (Corollary \ref{Exp}) gives
$$
\mathsf{Z}_{1}^{\NP} = \mathrm{Exp}\Bigg(\frac{G_1}{[t_1][t_2][t_3][t_4]}\Bigg),
$$
for some $G_1 \in qK_0^{\T'}(\pt)[[q]]$. As we will show in this section, this is enough information to first determine $G_1$ and then complete the proof of the Nekrasov--Piazzalunga conjecture for all $r>0$.

\subsection{Rigidity and boundedness}

In this section, we consider $\mathsf{Z}_{1}^{\NP}$, i.e., $r=1$.  
From Section \ref{sec:vertex}, it is clear that $\mathsf{Z}_{1}^{\NP}$ is independent of $w:=w_1$ and we will write $y:=y_1$. 
We first proceed in a similar fashion to Okounkov's proof of Nekrasov's conjecture on $\C^3$ \cite[Sect.~3.5]{Oko}. The first result is the 4-fold analog of \cite[Sect.~3.5.8]{Oko} and \cite[Lem.~5]{MNOP2}.
\begin{proposition} \label{H}
There exists  $H_1 \in q K_0^{\T'}(\pt) [[q]]$ such that
\begin{align*}
\mathsf{Z}_{1}^{\NP} = \mathrm{Exp}\Bigg(\frac{[t_1t_2][t_1t_3][t_2t_3]H_1}{[t_1][t_2][t_3][t_4]}\Bigg).
\end{align*}
\end{proposition}
\begin{proof}
Let $\pi$ be a non-empty solid partition.
It suffices to show that, for all $1 \leq i < j \leq 3$, the numerator of $[-\mathsf{v}_{\pi}]$ is divisible by $t_i^n t_j^n-1$ for some $n>0$ (and no such terms appear in the denominator). By symmetry, we may take $i=1$ and $j=2$. We write
$$
\mathsf{v}_{\pi} = \sum_{i,j,k,l \in \Z} c^{\pi}_{ijkl} \, t_1^i t_2^j t_3^k y^l.
$$
Then the claim is equivalent to 
$$
\sum_{i \in \Z} c^{\pi}_{ii00} < 0.
$$
Let $t_1 = u$, $t_2 = u^{-1}$, and $t_3 = v$. We define coefficients $a_{ij}^{\pi}, b_{ijk}^{\pi}$ by
\begin{align*}
Z_\pi &=  \sum_{(i,j,k,l) \in \pi} t_1^i t_2^j t_3^k t_4^l = \sum_{(i,j,k,l) \in \pi} u^{i-j} v^{k-l} = \sum_{i,j \in \Z} a^{\pi}_{ij} \, u^i v^ j, \\ 
\mathsf{v}_{\pi} &= \sum_{i,j,k \in \Z} b^{\pi}_{ijk} \, u^i v^ j y^k,
\end{align*}
and we have to show $b^{\pi}_{000}<0$. Concretely
$$
a_{ij}^{\pi} = |\{(a+i,a,b+j,b) \in \pi \, : \, a,b \geq 0\}|.
$$

By the same calculation as in \cite[Lem.~5]{MNOP2}, we find
\begin{align*}
b^{\pi}_{000} &= a^{\pi}_{00} + \sum_{i,j \in \Z} \Big\{ (2a^{\pi}_{i,j+1}a^{\pi}_{ij} - a^{\pi}_{i,j+1}a^{\pi}_{i+1,j} - a^{\pi}_{i+1,j+1} a^{\pi}_{ij}) - (2(a^{\pi}_{ij})^2 - 2a^{\pi}_{ij} a^{\pi}_{i+1,j}) \Big\} \\
&= a^\pi_{00} - \frac{1}{2}\sum_{i,j \in \Z}(a_{ij}^\pi - a_{i+1,j}^\pi - a_{i,j+1}^\pi + a_{i+1,j+1}^\pi)^2.
\end{align*}
Now let $s(i,j) = 1$ if $i \ge 0$  and $j \ge 0$, or if $i<0$ and $j < 0$. Otherwise let $s(i,j) = -1$.
Then
\begin{align*}
\sum_{i,j \in \Z}(a_{ij}^\pi - a_{i+1,j}^\pi - a_{i,j+1}^\pi + a_{i+1,j+1}^\pi)^2 &\ge \sum_{i,j \in \Z}|a_{ij}^\pi - a_{i+1,j}^\pi - a_{i,j+1}^\pi + a_{i+1,j+1}^{\pi}| \\
&\ge \sum_{i,j \in \Z} s(i,j)(a_{ij}^\pi - a_{i+1,j}^\pi - a_{i,j+1}^\pi + a_{i+1,j+1}^{\pi}) \\
&= 4a^\pi_{00},
\end{align*}
since the sum over each quadrant $\{(i,j) : i,j \geq 0\}$, $\{(i,j) : i,j < 0\}$, $\{(i,j) : i \geq 0, j < 0\}$, $\{(i,j) : i < 0, j \geq 0\}$ in $\Z^2$ equals $a^\pi_{00}$.
Thus we get $b^{\pi}_{000} \le -a^{\pi}_{00} < 0$, since $\pi$ is non-empty.
\end{proof}

Our next result is the 4-fold analog of \cite[Prop.~3.5.11]{Oko}.
\begin{proposition} \label{tindep}
The class $H_1 \in q K_0^{\T'}(\pt) [[q]]$ is independent of $t_1,t_2,t_3$, i.e., only depends on $y$.
\end{proposition}
\begin{proof}
We expand $H_1$ in $q$ 
$$
H_1 = \sum_{n=1}^{\infty} H_{1,n}(t_1,t_2,t_3,y) \, q^n.
$$ 
We argue, by induction on $n$, that $H_{1,n}$ is independent of the $t_i$. For $n=1$, this is obvious by direct calculation. For $n>1$, expanding the plethystic exponential gives 
\begin{align} 
\begin{split} \label{nthterm}
N_{1,n}^\glob  = \frac{[t_1t_2][t_1t_3][t_2t_3]}{[t_1][t_2][t_3][t_4]} H_{1,n}(t_1,t_2,t_3,y) + C,
\end{split}
\end{align}
where $C$ is a polynomial combination of terms of the form
$$
\frac{[t_1^{\ell}t_2^{\ell}][t_1^{\ell}t_3^{\ell}][t_2^{\ell}t_3^{\ell}]}{[t_1^{\ell}][t_2^{\ell}][t_3^{\ell}][t_4^{\ell}]} H_{1,m}(y^{\ell}),
$$
for various $\ell \geq 1$ and $m<n$. Here we used the induction hypothesis to conclude that $H_{1,m}$ is independent of $t_1,t_2,t_3$. Since $H_{1,n}$ is (a priori) a Laurent polynomial in the variables $t_i^{\pm 1}$ and $y^{\pm \frac{1}{2}}$, it suffices to show that its limits for $t_i^{\pm 1} \to \infty$ exist, i.e.~$H_{1,n}$ is bounded as $t_i^{\pm 1} \to \infty$. This is known as the \emph{rigidity principle}.\footnote{We thank the anonymous referee for pointing out that this principle dates back to the study of the rigidity of the $\widehat{A}$-genus by Atiyah--Hirzebruch \cite{AH}. In $K$-theoretic enumerative geometry, it was used very effectively by Okounkov in \cite{Oko} (where we learned it from).}

Without loss of generality, let $i=1$. Clearly 
$$
 \frac{[t_1^{\ell}t_2^{\ell}][t_1^{\ell}t_3^{\ell}][t_2^{\ell}t_3^{\ell}]}{[t_1^{\ell}][t_2^{\ell}][t_3^{\ell}][t_4^{\ell}]} =  \frac{[t_1^{\ell}t_2^{\ell}][t_1^{\ell}t_3^{\ell}][t_2^{\ell}t_3^{\ell}]}{[t_1^{\ell}][t_2^{\ell}][t_3^{\ell}][(t_1t_2t_3)^{-\ell}]} 
$$
is bounded as $t_1^{\pm 1} \to \infty$. By the lemma below, the left hand side of \eqref{nthterm} is also bounded as $t_1^{\pm 1} \to \infty$. The result follows.
\end{proof}

\begin{lemma}
Let $i=1,2,3$. The limit
$$
\lim_{t_i^{\pm 1} \to \infty} N_{1,n}^\glob
$$
exists, i.e., $N_{1,n}^\glob$ is bounded as $t_i^{\pm 1} \to \infty$.
\end{lemma}
\begin{proof}
Without loss of generality, we consider the $i=1$ case. 
Let $\pi$ be a solid partition of size $n$. We consider the Laurent monomials $t_1^{i_1}t_2^{i_2}t_3^{i_3} y^j$ appearing in $\mathsf{v}_{\pi}$ (where $t_4$ is eliminated by $t_4 = (t_1t_2t_3)^{-1}$). 
It is known that $\mathsf{v}_{\pi}$ has no fixed terms by \cite[Sect.~2.4.1]{NP}.
We also note that the class $\mathsf{v}_{\pi}$ has rank zero (Definition \ref{def:NPweight}). 

For any Laurent monomial $t_1^{i_1}t_2^{i_2}t_3^{i_3} \neq 1$ appearing in $Z$, there is a corresponding term $- y t_1^{-i_1}t_2^{-i_2}t_3^{-i_3}$ in $- y Z^\vee$.
Taken together, their contribution to $[-\mathsf{v}_{\pi}]$ is
$$
\frac{    t_1^{-\frac{i_1}{2}} t_2^{-\frac{i_2}{2}} t_3^{-\frac{i_3}{2}}  y^{\frac{1}{2}} - t_1^{\frac{i_1}{2}} t_2^{\frac{i_2}{2}} t_3^{\frac{i_3}{2}}  y^{-\frac{1}{2}}   }{ t_1^{\frac{i_1}{2}} t_2^{\frac{i_2}{2}} t_3^{\frac{i_3}{2}}  - t_1^{-\frac{i_1}{2}} t_2^{-\frac{i_2}{2}} t_3^{-\frac{i_3}{2}}   }.
$$
As $t_1^{\pm 1} \to \infty$, either $i_1=0$ and this is constant, or $i_1 \neq 0$ and this is $-y^{\pm \frac{1}{2}}$. 

Consider a Laurent monomial $t_1^{i_1}t_2^{i_2}t_3^{i_3}$ appearing in $Z Z^\vee$.
Suppose, after multiplying $t_1^{i_1}t_2^{i_2}t_3^{i_3}$ by $(1-t_1^{-1})(1-t_2^{-1})(1-t_3^{-1})$, there are no $T_t$-fixed terms.
Then its contribution to $[-\mathsf{v}_{\pi}]$ is
\begin{align*}
&\frac{   (t_1^{\frac{i_1}{2}} t_2^{\frac{i_2}{2}} t_3^{\frac{i_3}{2}}  -   t_1^{-\frac{i_1}{2}} t_2^{-\frac{i_2}{2}} t_3^{-\frac{i_3}{2}}  )    (t_1^{\frac{i_1-1}{2}} t_2^{\frac{i_2-1}{2}} t_3^{\frac{i_3}{2}}  - t_1^{-\frac{i_1-1}{2}} t_2^{-\frac{i_2-1}{2}} t_3^{-\frac{i_3}{2}} ) }{ (t_1^{\frac{i_1-1}{2}} t_2^{\frac{i_2}{2}} t_3^{\frac{i_3}{2}}  -   t_1^{-\frac{i_1-1}{2}} t_2^{-\frac{i_2}{2}} t_3^{-\frac{i_3}{2}}  )   (t_1^{\frac{i_1}{2}} t_2^{\frac{i_2-1}{2}} t_3^{\frac{i_3}{2}}  -   t_1^{-\frac{i_1}{2}} t_2^{-\frac{i_2-1}{2}} t_3^{-\frac{i_3}{2}}  )      } \\
&\times \frac{     (t_1^{\frac{i_1-1}{2}} t_2^{\frac{i_2}{2}} t_3^{\frac{i_3-1}{2}}  - t_1^{-\frac{i_1-1}{2}} t_2^{-\frac{i_2}{2}} t_3^{-\frac{i_3-1}{2}} )        (t_1^{\frac{i_1}{2}} t_2^{\frac{i_2-1}{2}} t_3^{\frac{i_3-1}{2}}  - t_1^{-\frac{i_1}{2}} t_2^{-\frac{i_2-1}{2}} t_3^{-\frac{i_3-1}{2}} )}{  (t_1^{\frac{i_1}{2}} t_2^{\frac{i_2}{2}} t_3^{\frac{i_3-1}{2}}  -   t_1^{-\frac{i_1}{2}} t_2^{-\frac{i_2}{2}} t_3^{-\frac{i_3-1}{2}}  )   (t_1^{\frac{i_1-1}{2}} t_2^{\frac{i_2-1}{2}} t_3^{\frac{i_3-1}{2}}  -   t_1^{-\frac{i_1-1}{2}} t_2^{-\frac{i_2-1}{2}} t_3^{-\frac{i_3-1}{2}}  )     }.
\end{align*}
The terms involving $t_1^{\pm \frac{i_1}{2}}$ (respectively $t_1^{\pm \frac{i_1-1}{2}}$) in the numerator and denominator come in pairs so this product is bounded for $t_1^{\pm 1} \to \infty$ as well. Since the $T_t$-fixed terms of  $(1-t_1^{-1})(1-t_2^{-1})(1-t_3^{-1})Z^\vee Z$ cancel the $T_t$-fixed term of $Z$, the result follows. 
\end{proof}

\subsection{Rank one} \label{sec:rk1}

We start with the following observation, which is a special case of \cite[Prop.~2.1]{CKM} (which also deals with subschemes of dimension 1). For completeness, we include the proof.
\begin{proposition} \label{dimred}
Let $Z \subset \C^4$ be a $T_t$-fixed 0-dimensional subscheme corresponding to a solid partition $\pi$. Then 
$[-\mathsf{v}_\pi] |_{y=t_4} = 0$ unless $Z \subset \C^3 = \{x_4=0\}$. For $Z \subset \C^3 = \{x_4=0\}$, $\mathsf{v}_\pi |_{y=t_4}$ is identical to the (fully equivariant) character of the Donaldson--Thomas virtual tangent space at $Z$, as given in \cite[Sect.~4.7]{MNOP1}. In particular
$$
\mathsf{Z}_1^{\NP} |_{y=t_4} = \sum_{n=0}^{\infty} \chi(\Hilb^n(\C^3), \widehat{\O}^{\vir}) \, (-q)^n.
$$
\end{proposition}
\begin{proof}
For $Z \subset \{x_4=0\}$, the statement follows at once by comparing Definition \ref{def:NPweight} to \cite[Sect.~4.7]{MNOP1}. 
Suppose $Z$ is not scheme theoretically supported in $\{x_4=0\}$. The vertex $\mathsf{v}_{\pi}$ does not have a $T_t$-fixed part by \cite[Sect.~2.4.1]{NP}. Therefore, it suffices to consider the $T_t$-fixed terms arising from setting $y=t_4 = (t_1t_2t_3)^{-1}$ in the $y$-dependent part of $\mathsf{v}_{\pi}$, i.e.~$-y Z^\vee$. Since $Z$ contains the term $t_4$, setting $y=t_4$, it contributes $-yt_4^{-1}=-1$, so $[-\mathsf{v}_\pi] |_{y=t_4} = 0$.

The 0-dimensional $(\C^*)^3$-fixed subschemes $Z \subset \C^3$ correspond to \emph{plane} partitions $\pi$. By \cite[Sect.~3]{Oko}, we have
$$
\sum_{n=0}^{\infty} \chi(\Hilb^n(\C^3), \widehat{\O}^{\vir}) \, q^n = \sum_{\pi} [-\mathsf{V}_{\pi}^{\mathrm{DT}}] \, q^{|\pi|}.
$$
The result of the proposition follows by noting that $(-1)^{n + \mu_\pi} = (-1)^n$ for all plane partitions $\pi$ (Propositions \ref{globaltolocal} and \ref{signprop}).
\end{proof}

The proof of Theorem \ref{mainthm} for $r=1$ can now be completed as follows. 
\begin{proof}[Proof of Theorem \ref{mainthm} for $r=1$]
By Propositions \ref{H} and \ref{tindep}, taking plethystic log yields
$$
H_1(y) = \frac{[t_1][t_2][t_3][t_4]}{[t_1t_2][t_1t_3][t_2t_3]} \, \mathrm{Log}\big(\mathsf{Z}_1^{\NP} \big),
$$
where we stress that $H_1(y)$ only depends on $y:=y_1$ and the formal variable $q$. By Proposition \ref{dimred} and Okounkov's Theorem \ref{Okounkovthm} 
$$
 \mathrm{Log} \big( \mathsf{Z}_1^{\NP} \big) \Big|_{y=t_4} = \frac{[t_1t_2][t_1t_3][t_2t_3]}{[t_1][t_2][t_3][\kappa^{\frac{1}{2}} q ]   [\kappa^{\frac{1}{2}} q^{-1} ]},
$$
where $\kappa = t_1t_2t_3 = t_4^{-1}$. Hence 
$$
H_1(t_4) = \frac{[t_4]}{ [t_4^{-\frac{1}{2}} q ]   [t_4^{-\frac{1}{2}} q^{-1} ] } = \frac{[t_4]}{[t_4^{\frac{1}{2}} q ]   [t_4^{\frac{1}{2}} q^{-1} ] },
$$
and the proof is complete.
\end{proof}

\subsection{Higher rank} \label{sec:higherrk}

In this section, we complete the proof of Theorem \ref{mainthm} for $r>1$. It will follow from showing independence of the framing parameters $w_1, \ldots, w_r$ and reducing to the $r=1$ case. We learned this strategy from \cite{FMR}. The rank reduction in this section does not require factorizability.
We start with the following analog of \cite[Lem.~6.7]{FMR}. Recall the definition of $\mathsf{v}_{\vec{\pi},\alpha\beta}$ from Definition \ref{def:NPweight}.
\begin{proposition} \label{offdiag}
Let $\vec{\pi} = (\pi_1, \ldots, \pi_r)$ be an $r$-tuple of solid partitions. For any $\alpha < \beta$, we have
$$
\lim_{L \rightarrow \infty} [-\mathsf{v}_{\vec{\pi},\alpha\beta}] [-\mathsf{v}_{\vec{\pi},\beta\alpha}]\Big|_{(w_1 = L, w_2 = L^2, \ldots, w_r = L^r)} = \frac{(-y_\beta^{\frac{1}{2}})^{|\pi_\alpha|}}{(-y_\alpha^{\frac{1}{2}})^{|\pi_\beta|}}.
$$
\end{proposition}
\begin{proof}
Eliminating $t_4 = (t_1t_2t_3)^{-1}$, we write
\begin{align*}
Z_\alpha = \sum_{(i_1,i_2,i_3)} t_1^{i_1}t_2^{i_2}t_3^{i_3}, \quad Z_\beta = \sum_{(j_1,j_2,j_3)} t_1^{j_1}t_2^{j_2}t_3^{j_3},
\end{align*}
for certain finite sums with possibly repeated terms. The first sum has $|\pi_\alpha|$ terms and the second has $|\pi_\beta|$ terms. For any monomial $\pm w_\beta w_\alpha^{-1} t_1^{k_1} t_2^{k_2} t_3^{k_3}$ appearing in $-\mathsf{v}_{\vec{\pi},\alpha\beta}$, we have
$$
[\pm w_\beta w_\alpha^{-1} t_1^{k_1} t_2^{k_2} t_3^{k_3}] = \Big( (L^{\beta-\alpha} t_1^{k_1} t_2^{k_2} t_3^{k_3})^{\frac{1}{2}}(1-(L^{\beta-\alpha} t_1^{k_1} t_2^{k_2} t_3^{k_3})^{-1}) \Big)^{\pm 1}, \quad \beta - \alpha>0.
$$
In the limit $L \to \infty$, the contributions from $Z_\alpha^\vee Z_\beta$ to $[-\mathsf{v}_{\vec{\pi},\alpha\beta}]$ cancel and contribute 1 (and similarly for $Z_\beta^\vee Z_\alpha$ and $[-\mathsf{v}_{\vec{\pi},\beta\alpha}]$). Therefore, the desired limit equals the limit of $L \to \infty$ of
$$
\Bigg\{ L^{\frac{\beta-\alpha}{2}( - |\pi_\beta| + |\pi_\alpha|) } \frac{\prod_{(i_1,i_2,i_3)}  t_1^{-\frac{i_1}{2}}t_2^{-\frac{i_2}{2}}t_3^{-\frac{i_3}{2}} y_\beta^{\frac{1}{2}}}{\prod_{(j_1,j_2,j_3)}  t_1^{\frac{j_1}{2}}t_2^{\frac{j_2}{2}}t_3^{\frac{j_3}{2}}} \Bigg\}    \Bigg\{  (-1)^{|\pi_\alpha| + |\pi_\beta|} L^{\frac{\beta-\alpha}{2}( |\pi_\beta| - |\pi_\alpha|) } \frac{\prod_{(j_1,j_2,j_3)}  t_1^{\frac{j_1}{2}}t_2^{\frac{j_2}{2}}t_3^{\frac{j_3}{2}} y_\alpha^{-\frac{1}{2}}}{\prod_{(i_1,i_2,i_3)}  t_1^{-\frac{i_1}{2}}t_2^{-\frac{i_2}{2}}t_3^{-\frac{i_3}{2}}} \Bigg\},
$$
which gives the answer.
\end{proof}

\begin{proposition} \label{windep}
The generating series $\sfZ_r^{\NP}$ is independent of $w_1, \ldots, w_r$.
\end{proposition}
\begin{proof}
We endow $\Sym^n(\C^4)$ with the trivial $(T_w \times \widetilde{T}_y)$-action. Then the Quot-Chow morphism \cite[Cor.~7.15]{Ryd}
$$
\nu_n \colon \Quot_r^n(\C^4) \to \Sym^n(\C^4)
$$
is proper, $\T'$-equivariant, and we have
\begin{align} \label{pushtoChow}
N_{r,n}^\glob = \chi( \Sym^n(\C^4), R \nu_{n*}(\cN_{r,n}^\glob)). 
\end{align}
Since $\Sym^n(\C^4)$ has a trivial $T_w$-action, the class
$$
R \nu_{n*}(\cN_{r,n}^\glob) \in K_0^{\T'}(\Sym^n(\C^4))
$$
is a finite sum 
$$
\sum_{(i_1, \ldots, i_r)} V_{i_1, \ldots, i_r} \, w_1^{i_1} \cdots w_r^{i_r}, \quad V_{i_1, \ldots, i_r} \in K_0^{T_t \times \widetilde{T}_y}(\Sym^n(\C^4)).
$$
Hence, the right hand side of \eqref{pushtoChow} is a \emph{Laurent polynomial} in the variables $w_i$.

By Proposition \ref{offdiag}, we know that \eqref{pushtoChow} is bounded for 
$$
(w_1, \ldots, w_r) = (L, \ldots, L^r), \quad \textrm{as \ } L \to \infty.
$$ 
The proof of Proposition \ref{offdiag} also shows that, for any choice of $0 < n_1 < \cdots < n_r$, we have that \eqref{pushtoChow} is bounded for 
$$
(w_1, \ldots, w_r) = (L^{n_1}, \ldots, L^{n_r}), \quad L^{\pm 1} \to \infty.
$$
Therefore \eqref{pushtoChow} is independent of the $w_i$.
\end{proof}

Before we finish the proof of Theorem \ref{mainthm}, we note the following identity 
\begin{equation} \label{funid}
\frac{[a]}{[a^{\frac{1}{2}} b^{\frac{1}{2}} c] [a^{\frac{1}{2}} b^{-\frac{1}{2}} c^{-1}]} + \frac{[b]}{[a^{-\frac{1}{2}} b^{\frac{1}{2}} c] [a^{\frac{1}{2}} b^{\frac{1}{2}} c^{-1}]} = \frac{[ab]}{[(ab)^{\frac{1}{2}} c] [(ab)^{\frac{1}{2}} c^{-1}]}.
\end{equation}

\begin{proof}[Proof of Theorem \ref{mainthm} for $r>1$]
By Propositions \ref{globaltolocal} and \ref{signprop}, we have
$$
\sfZ_r^{\NP} = \sum_{\vec{\pi} = (\pi_1, \ldots, \pi_r)} ((-1)^{r}q)^{|\vec{\pi}|} (-1)^{\mu_{\vec{\pi}}} \prod_{\alpha=1}^{r} [-\mathsf{v}_{\vec{\pi},\alpha \alpha}] \cdot \prod_{1 \leq \alpha < \beta \leq r} [-\mathsf{v}_{\vec{\pi},\alpha\beta}] [-\mathsf{v}_{\vec{\pi},\beta\alpha}].
$$
By Proposition \ref{windep}, we may specialize $(w_1 = L, w_2 = L^2, \ldots, w_r = L^r)$ and take the limit $L \to \infty$. Consider the rank 1 generating series
$$
\sfZ_1^{\NP} = \sfZ_1^{\NP}(t_1,t_2,t_3,y,q),
$$
and recall that we proved the rank 1 case of Theorem \ref{mainthm} in Section \ref{sec:rk1}. Using Proposition \ref{offdiag}, we obtain
\begin{align*}
\sfZ_r^{\NP} &=\sum_{\vec{\pi} = (\pi_1, \ldots, \pi_r)} ((-1)^{r}q)^{|\vec{\pi}|}  \cdot \prod_{\alpha=1}^{r} (-1)^{\mu_{\pi_\alpha}} [-\mathsf{v}_{\pi_\alpha}] \cdot \prod_{1 \leq \alpha < \beta \leq r}\frac{(-y_\beta^{\frac{1}{2}})^{|\pi_\alpha|}}{(-y_\alpha^{\frac{1}{2}})^{|\pi_\beta|}} \\
&= \prod_{\alpha=1}^{r} \sfZ_1^{\NP}\big( t_1,t_2,t_3,y_\alpha,q \prod_{\beta \neq \alpha} y_\beta^{\frac{1}{2} \mathrm{sgn}(\beta-\alpha)} \big).
\end{align*}
Now write $\sfZ_r^{\NP} = \mathrm{Exp}(\cdots)$ and consider the exponent, i.e.,

\begin{equation} \label{eqn:summingExp}
\frac{[t_1t_2][t_1t_3][t_2t_3]}{[t_1][t_2][t_3][t_4]  } \sum_{\alpha = 1}^{r} \frac{[y]}{[y^{\frac{1}{2}} q ]   [y^{\frac{1}{2}} q^{-1} ] } \Bigg|_{(y,q)=(  y_\alpha, q \prod_{\beta \neq \alpha} y_\beta^{\frac{1}{2} \mathrm{sgn}(\beta-\alpha)} )}.
\end{equation}

Next, we apply \eqref{funid} to $a= y_1$, $b = y_2$, $c =  y_3 \cdots y_r q$. We deduce that the first two terms of the sum in \eqref{eqn:summingExp} add up to
$$
 \frac{[y_1 y_2]}{[(y_1 y_2)^{\frac{1}{2}} (y_3 \cdots y_r)^{\frac{1}{2}}  q] [(y_1 y_2)^{\frac{1}{2}} (y_3 \cdots y_r)^{-\frac{1}{2}}  q^{-1}]}.
$$
Subsequently, we apply \eqref{funid} to $a=y_1y_2$, $b=y_3$, $c = y_4 \cdots y_r q$. We find that the first three terms of the sum in \eqref{eqn:summingExp} add up to
$$
\frac{[y_1y_2y_3]}{[(y_1y_2y_3)^{\frac{1}{2}} (y_4 \cdots y_r)^{\frac{1}{2}}  q] [(y_1y_2y_3)^{\frac{1}{2}} (y_4 \cdots y_r)^{-\frac{1}{2}}  q^{-1}]}.
$$
Repeating this procedure $r-1$ times, we arrive at the desired result
\begin{equation*}
\frac{[y]}{[y^{\frac{1}{2}} q] [y^{\frac{1}{2}} q^{-1}]}, \quad y:=y_1 \cdots y_r. \qedhere
\end{equation*}
\end{proof}

We end by remarking that the Magnificent Four formula from Theorem \ref{mainthm} reduces to the Awata--Kanno conjecture and the generating function of so-called tetrahedron invariants.
\begin{remark} \label{FMRrmk}
The analog of Proposition \ref{dimred} is obviously true in the higher rank case as well. If a $\T'$-fixed quotient $[\O_{\C^4}^{\oplus r} \twoheadrightarrow Q]$ corresponding to $\vec{\pi} = (\pi_1, \ldots, \pi_r)$ does \emph{not} factor through $\iota_* \O_{\C^3}^{\oplus r}$, where $\iota \colon \C^3 = \{x_4=0\} \hookrightarrow \C^4$ denotes the inclusion, then 
$$
[-\mathsf{v}_{\vec{\pi}}] \big|_{y_1 = \cdots = y_r = t_4} = 0.
$$
If it does factor, i.e.~each $\pi_\alpha$ is a plane partition, then $[-\mathsf{v}_{\vec{\pi}}] |_{y_1 = \cdots = y_r = t_4}$ is identical to the (fully equivariant) vertex of \cite{FMR}. Furthermore, for any $\vec{\pi} = (\pi_1, \ldots, \pi_r)$, where each $\pi_\alpha$ is a plane partition, we have
$$
(-1)^{|\vec{\pi}| + \mu_{\vec{\pi}}} = (-1)^{|\vec{\pi}|}.
$$
Therefore, we deduce the Awata--Kanno conjecture (Theorem \ref{FMRthm}), previously proved in \cite{FMR, AK1}:
$$
\sum_{n=0}^{\infty} \chi(\Quot^{n}_{r}(\C^3), \widehat{\O}^{\vir}) \, ((-1)^r q)^n = \sfZ_r^{\NP} \Big|_{(y_1 = \cdots = y_r = t_4)} = \mathrm{Exp}\Bigg(\frac{[t_1t_2][t_1t_3][t_2t_3][\kappa^r]}{[t_1][t_2][t_3][\kappa][\kappa^{\frac{r}{2}} q ]   [\kappa^{\frac{r}{2}} q^{-1} ]}\Bigg),
$$
where $\kappa = t_1t_2t_3 = t_4^{-1}$.
\end{remark}

\begin{remark}
Consider the singular 3-fold $$\Delta = \{x_1x_2x_3x_4 = 0\} = \bigcup_{i=1}^4 \C_i^3 \subset \C^4, \quad \quad \mathcal{E}_{\vec{r}} := \bigoplus_{i=1}^{4} \O_{\C_i^3}^{\oplus r_i},$$ where $\C_i^3 = \{x_i = 0\}$ and $\vec{r} = (r_1, r_2, r_3, r_4)$ are non-negative integers.  Based on the physics literature \cite{PYZ}, in \cite{FM}, the authors study so-called tetrahedron invariants. These are defined by an Oh--Thomas class of the Quot scheme $\Quot_{\Delta}^n(\mathcal{E}_{\vec{r}})$ parametrizing 0-dimensional quotients of length $n$ of $\mathcal{E}_{\vec{r}}$. 
In fact, for $r = \sum_i r_i$, $\Quot_{\Delta}^n(\mathcal{E}_{\vec{r}}) \subset \Quot_r^n(\C^4)$ is naturally a closed subscheme. Specializing $y_1 = \cdots = y_{r_1} = t_1$, $ y_{r_1+1} = \cdots = y_{r_1+r_2} = t_2$, $\ldots$, the generating function of Theorem \ref{mainthm} reduces to the one for tetrahedron invariants obtained in \cite{FM}. The fact that Theorem \ref{mainthm} implies the formula for the generating series of tetrahedron invariants follows from the results in \cite[App.~A]{FM} (which partially relies on the sign analysis of this paper). It is worth stressing that the proof of the formula for the tetrahedron generating function in \cite{FM} does \emph{not} require the factorizability results of this paper. 
\end{remark}

\printbibliography

\noindent {\tt{m.kool1@uu.nl, jorgeren@uio.no}}
\end{document}